\documentclass[11pt]{article}
\usepackage[pagebackref,colorlinks=true,pdfpagemode=none,urlcolor=blue,linkcolor=blue,citecolor=blue]{hyperref}
\usepackage{subfigure}
\usepackage{amsthm, amssymb, bm}
\usepackage{soul, color}
\usepackage[]{amsmath}
\usepackage[]{amsfonts}
\usepackage[]{fancyhdr}
\usepackage[]{graphicx}
\usepackage[]{empheq}
\usepackage[toc,page]{appendix}
\newtheorem{theorem}{Theorem}[section]

\newtheorem{remark}[theorem]{Remark}
\newtheorem{hypothesis}[theorem]{Hypothesis}

\makeatletter

\newcommand{\Rmnum}[1]{\expandafter\@slowromancap\romannumeral #1@}
\makeatother

\def \Imm {\mathbb{I}}

\def \Rm {\mathbb{R}}

\def\C{\mathcal{C}}

\newcommand{\be}{\mathbf e}

\newcommand{\bE}{\bm E}

\newcommand{\tr}{ {\text{tr }}}

\newcommand{\cout}[1]{}

\def \i {\boldsymbol\iota}

\title{Imaging of complex-valued tensors for two-dimensional Maxwell's equations}

\author{Chenxi Guo\thanks{Department of Applied Physics and Applied Mathematics, Columbia University,  New York NY, 10027; cg2597@columbia.edu} \and Guillaume Bal\thanks{Department of Applied Physics and Applied Mathematics, Columbia University,  New York NY, 10027; gb2030@columbia.edu}}

\begin{document}
\maketitle
\begin{abstract}
This paper concerns the imaging of a complex-valued anisotropic tensor $\gamma=\sigma+\i\omega\varepsilon$ from knowledge of several inter magnetic fields $H$ where $H$ satisfies the anisotropic Maxwell system on a bounded domain $X\subset\Rm^2$ with prescribed boundary conditions on $\partial X$.  We show that $\gamma$ can be uniquely reconstructed with a loss of two derivatives from errors in the acquisition $H$. A minimum number of five {\em well-chosen} functionals guaranties a local reconstruction of $\gamma$ in dimension two. The explicit inversion procedure is presented in several numerical simulations, which demonstrate the influence of the choice boundary conditions on the stability of the reconstruction. This problem finds applications in the medical imaging modalities Current Density Imaging and Magnetic Resonance Electrical Impedance Tomography. 
\end{abstract}

\section{Introduction}
The electrical properties of a biological tissue are characterized by $\gamma = \sigma+\i\varepsilon$, where $\sigma$ and $\varepsilon$ denote the conductivity and the permittivity. The properties that indicate conditions of tissues can provide important diagnostic information. Extensive studies have been made to produce medical images inside human body by applying currents to probe their electrical properties. This technique is called Electrical Impedance Tomography (EIT). This imaging technique displays high contrast between health and non-healthy tissues. However, EIT uses boundary measurements of current-voltage data and suffers from very low resolution capabilities. This leads to an inverse boundary value problem (IBVP) called the Calder\'{o}n's problem, which is severely ill-posed and unstable \cite{Sylvester1987,Uhlmann2009}. For IBVP in electrodynamics, we refer the reader to \cite{Caro2009, Kenig2011, Ola1993, Ola1996, Somersalo1992}. Moreover, well-known obstructions show that anisotropic admittivities cannot be uniquely reconstructed from boundary measurements, see \cite{Kohn1984,Uhlmann2009}.

Several recent medical imaging modalities, which are often called coupled-physic modalities or hybrid imaging modalities, aim to couple the high contrast modality with the high resolution modality.  These new imaging techniques are two steps processes. In the first step, one uses boundary measurements to reconstruct internal functionals inside the tissue. In the second step, one reconstructs the electrical properties of the tissue from given internal data, which greatly improve the resolution of quantitative reconstructions. An incomplete list of such techniques includes ultrasound modulated optical tomography, magnetic resonance electrical impedance tomography, photo-acoustic tomography and transient elastography. We refer the reader to \cite{Ammari2008, Bal2012e, Guo2014, Bal2010, Kuchment2011a, Kuchment2011, Monard2012a, Nachman2009} for more details. For other techniques in inverse problems, such as inverse scattering, we refer the reader to \cite{Kabanikhin2011, Nivikov2005, Takhtadzhan1979}.

In this paper we are interested in the hybrid inverse problem of reconstructing $(\sigma, \varepsilon)$ in the Maxwell's system from the internal magnetic fields $H$. Internal magnetic fields can be obtained using a Magnetic Resonance Imaging (MRI) scanner, see \cite{Ider1997} for details.  The explicit reconstructions we propose require that all components of the magnetic field $H$ be measured. This may be challenging in many practical settings as it requires a rotation of the domain being imaged or of the MRI scanner. The reconstruction of $(\sigma, \varepsilon)$ from knowledge of only some components of $H$, ideally only one component for the most practical experimental setup, is open at present. We assume that the above first step is done and we are given the data of the internal magnetic fields in the domain. 

In the isotropic case, a reconstruction for the conductivity was given in \cite{Seo2012}. In \cite{Guo2013b}, an explicit reconstruction procedure was derived for an arbitrary anisotropic, complex-valued tensor $\gamma= \sigma+\i\varepsilon$ in the Maxwell's equations in $\Rm^3$. This explicit reconstruction method requires that some matrices constructed from measurements satisfy appropriate conditions of linear independence. In the present work, we provide numerical simulations in two dimensions to demonstrate the reconstruction procedure 
for both smooth and rough coefficients. 

The rest of the paper is organized as follows. The main results and the reconstruction formulas are presented in Section \ref{main results}. The reconstructibility hypothesis is proved in Section \ref{Fulfilling Hypothesis}. The numerical implementations of the algorithm with synthetic data are shown in Section \ref{num simu}. Section \ref{se:conclu} gives some concluding remarks.

\section{Statements of the main results}\label{main results}

\subsection{Modeling of the problem}

Let $X$ be a simply connected, bounded domain of $\Rm^2$ with smooth boundary. The smooth anisotropic electric permittivity, conductivity, and the constant isotropic magnetic permeability are respectively described by $\varepsilon(x)$, $\sigma(x)$ and $\mu_0$, where $\varepsilon(x)$, $\sigma(x)$ are tensors and $\mu_0$ is a constant scalar, known, coefficient. We denote $\gamma=\sigma+\i\omega\varepsilon$, where $\omega >0$ is the frequency of the electromagnetic wave. We assume that $\varepsilon(x)$ and $\sigma(x)$ are uniformly bounded from below and above, i.e., there exist constants $\kappa_{\varepsilon},\kappa_{\sigma}>1$ such that for all $\xi\in \Rm^2$,
\begin{align}\label{positive definite}
\begin{split}
  \kappa_{\varepsilon}^{-1}\|\xi\|^2  &\le \xi\cdot\varepsilon\xi \le \kappa_{\varepsilon}\|\xi\|^2 \quad \text{in} \enspace  X\\ 
  \kappa_{\sigma}^{-1}\|\xi\|^2  &\le\xi\cdot\sigma\xi \le \kappa_{\sigma}\|\xi\|^2 \quad \text{in} \enspace  X.
 \end{split}
\end{align}
Let $\bm E=( E^1, E^2)'\in\mathbb{C}^2 $ and $H\in\mathbb{C}$ denote the electric and magnetic fields inside the domain $X$. Thus $\bm E$ and $H$ solve the following time-harmonic Maxwell's equations: 
\begin{align}\label{Eq:maxwell} 
\left\{\begin{array}{lll}
\nabla\times\bm E+\i\omega\mu_0H =0\\
\bm\nabla\times H-\gamma\bm E=0,
\end{array}\right.
\end{align}
with the boundary condition 
\begin{align}\label{Eq:boundary}
\bm\nu\times \bm E:= \nu_1E^2-\nu_2 E^1=f, \quad \text{on} \enspace  {\partial X},
\end{align}
where $\bm\nu=(\nu_1,\nu_2)$ is the exterior unit normal vector on the boundary $\partial X$. The standard well-posedness theory for MaxwellÕs equations \cite{Lions1993} states that given $f\in H^{\frac{1}{2}}(\partial X)$, the equation \eqref{Eq:maxwell}-\eqref{Eq:boundary} admits a unique solution in the Sobolev space $H^1(X)$. In this paper, the notations $\nabla$ and $\bm\nabla$ distinguish between the scalar and vector curl operators:
\begin{align}
\nabla\times\bm E = \frac{\partial  E^2}{\partial x_1} - \frac{\partial  E^1}{\partial x_2} \quad \text{and} \quad \bm\nabla\times H=(- \frac{\partial H}{\partial x_2},  \frac{\partial H}{\partial x_1})'.
\end{align}
Although \eqref{Eq:maxwell} can be reduced to a scalar Laplace equation for $H$, we treat it as a system. The reconstruction method holds for the full 3 dimensional case. In this paper,  we assume that the conductivity $\sigma$ and the permeability $\varepsilon$ are independent 
of the third component in $\Rm^3$ and give the numerical simulations in two dimension to validate the reconstruction method. 
\subsection{Local reconstructibility condition}
We select $5$ boundary conditions $\{f_i\}_{1\leq i\leq 5}$ such that the corresponding electric and magnetic fields $\{\bm E_i,H_i\}_{1\leq i\leq 5}$ satisfy the Maxwell's equations \eqref{Eq:maxwell}. Assuming that over a sub-domain $X_0\subset X$, the two electric fields $\bE_1$, $\bE_2$ satisfy the following positive condition, 
\begin{align}\label{posi condition}
	\inf_{x\in X_0} |\det (\bE_1,\bE_2)| \ge c_0>0.
\end{align}
Thus the $3$ additional solutions $\{\bE_{2+j}\}_{1\leq j\leq 3}$ can be decomposed as linear combinations in the basis $(\bE_1,\bE_2)$,
\begin{align}
\bE_{2+j}=\lambda^j_1 \bE_1+\lambda^j_2 \bE_2, \quad  1\leq j\leq 3,
\label{ln dep}
\end{align}
where the coefficients $\{\lambda^j_1,\lambda^j_2\}_{1\leq j\leq 3}$ can be computed by Cramer's rule as follows:
\begin{align}\label{cramer rule}
\begin{split}
\lambda^j_1 &=\frac{\det(\bE_{2+j},\bE_2)}{\det(\bE_1,\bE_2)}=\frac{\det(\bm\nabla\times H_{2+j},\bm\nabla\times H_2)}{\det(\bm\nabla\times H_1,\bm\nabla\times H_2)},\\
\lambda^j_2 &=\frac{\det(\bE_1,\bE_{2+j})}{\det(\bE_1,\bE_2)}=\frac{\det(\bm\nabla\times H_1,\bm\nabla\times H_{2+j})}{\det(\bm\nabla\times H_1,\bm\nabla\times H_2)}.
\end{split}
\end{align}
Thus these coefficients can be computed from the available magnetic fields. The reconstruction procedures will make use of the matrices $Z_j$ defined by
\begin{align}
    Z_j=\left[\bm\nabla\times\lambda^j_1 | \bm\nabla\times\lambda^j_2\right],\quad \text{where } \enspace  1\leq j\leq 3.
\label{Y Z}
\end{align}
These matrices are also uniquely determined from the known magnetic fields. Denoting the matrix $H:=[\bm\nabla\times H_1 | \bm\nabla\times H_2]$ and the skew-symmetric matrix $J=\left[\begin{smallmatrix} 0 &-1\\ 1 & 0 \end{smallmatrix}\right]$, we construct three matrices as follows,
\begin{align}\label{constraint matrice}
M_j = (Z_jH^T)^{sym}, \quad  1\leq j\leq 3,
\end{align}
where $A^T$ denotes the transpose of a matrix $A$ and $A^{sym}:= (A + A^T)/2$. The calculations in the following section show that condition \eqref{posi condition} and the linear independence of $\{M_j\}_{1\leq j\leq 3}$ in $S_2(\mathbb{C})$ are sufficient to guarantee local reconstruction of $\gamma$. The required conditions, which allow us to set up our reconstruction formulas, are listed in the following hypotheses. The reconstructions are {\em local} in nature: the reconstruction of $\gamma$ at $x_0\in X$ requires the knowledge of $\{H_j(x)\}_{1\leq j\leq J}$ for $x$ only in the vicinity of $x_0$.

\begin{hypothesis}\label{2 hypo}
Given Maxwell's equations \eqref{Eq:maxwell} with smooth $\varepsilon$ and $\sigma$ satisfying the uniform ellipticity conditions \eqref{positive definite}, there exist a set of illuminations $\{f_i\}_{1\leq i\leq 5}$ such that the corresponding electric fields $\{\bm E_i\}_{1\leq i\leq 5}$ satisfy the following conditions:
\begin{enumerate}
  \item  $\inf_{x\in X_0} |\det (\bE_1,\bE_2)| \ge c_0>0$ holds on a sub-domain $X_0\subset X$,
   \item The matrices $\{M_j\}_{1\leq j\leq 3}$ constructed in \eqref{constraint matrice} are linearly independent in $S_2(\mathbb{C})$ on $X_0$, where $S_2(\mathbb{C})$ denotes the space of $2\times 2$ symmetric matrices.
\end{enumerate}
\end{hypothesis}
 
\begin{remark}
Note that both conditions in Hypothesis \ref{2 hypo} can be expressed in terms of the measurements $\{H_j\}_j$, and thus can be checked during the experiments. When the above constant $c_0$ is deemed too small, or the matrices $M_j$ are not sufficiently independent, then additional measurements might be required. For the $3$ dimensional case, Hypothesis \ref{2 hypo} holds locally, under some smoothness assumptions on $\gamma$, with $6$ well-chosen boundary conditions. The proof is based on the Runge approximation, see \cite{Guo2013b} for details.
\end{remark}

\subsection{Reconstruction approaches and stability results}
The reconstruction approaches were presented in \cite{Guo2013b} for a $3$ dimensional case. To make this paper self-contained, we briefly list the algorithm for the two-dimensional case. 
Denote $M_2(\mathbb{C})$ as the space of $2\times 2$ matrices with inner product $\langle A,B\rangle:=\tr(A^*B)$. We assume that Hypothesis \ref{2 hypo} holds over some $X_0\subset X$ with $5$ electric fields $\{\bm E_i\}_{1\leq i\leq 5}$. In particular, the matrices $\{M_j\}_{1\leq j\leq 3}$ constructed in \eqref{constraint matrice} are linearly independent in $S_2(\mathbb{C})$. We will see that the inner products of $(\gamma^{-1})^*$ with all $M_j$ can be calculated from knowledge of $\{H_j\}_{1\leq j\leq 5}$. Then $\gamma$ can be explicitly reconstructed by least-square method. The reconstruction formulas can be found in Section \ref{rec approach}. This algorithm leads to a unique and stable reconstruction and the stability estimate will be given in Section \ref{stablity estimate}.

\subsubsection{Reconstruction algorithms}\label{rec approach}
We apply the curl operator to both sides of \eqref{ln dep}. Using the product rule, we get the following equation,
\begin{align}
\sum_{i=1,2}\lambda^j_i\nabla\times \bE_i + \bE_i\cdot\bm\nabla\times\lambda^j_i = \nabla\times\bE_{2+j}, \quad \text{for} \enspace  j\geq 3.
\end{align}
Substituting $H_i$ into $\bE_i$ in the above equation, we obtain the following equation after rearranging terms,
\begin{align}
\sum_{i=1,2} \bm\nabla\times\lambda_i^j\cdot(\gamma^{-1}\bm\nabla\times H_i) = \sum_{i=1,2}\i\omega\mu_0(\lambda^j_iH_i-H_{2+j})
\end{align}
Recalling the definition of $Z_j$ by \eqref{Y Z}, the above equation leads to 
\begin{align}\label{rec sys}
\gamma^{-1}:(Z_jH^t)^{sym}=\sum_{i=1,2}\i\omega\mu_0(\lambda^j_iH_i-H_{2+j}),
\end{align}
where the matrix $H=[\bm\nabla\times H_1 | \bm\nabla\times H_2]$. Note that $M_j = (Z_jH^T)^{sym}$ and the RHS of the above equation are computable from the measurements, thus $\gamma$ can be explicitly reconstructed by \eqref{rec sys} provided that $\{M_j\}_{1\leq j\leq 3}$ are of full rank in $S_2(\mathbb{C})$.
\begin{remark}
The reconstruction formulae is local. In practice, we add more measurements and get additional $M_j$ such that $\{M_j\}_j$ is of full rank in $S_2(\mathbb{C})$. The system \eqref{rec sys} becomes overdetermined and $\gamma$ can be reconstructed by solving \eqref{rec sys} using least-square method. 
\end{remark}

\subsubsection{Uniqueness and stability results}\label{stablity estimate}
The algorithm derived in the above section leads to a unique and stable reconstruction in the sense of the following theorem:

\begin{theorem}\label{stability}
Suppose that Hypotheses \ref{2 hypo} hold over some $X_0\subset X$ for two sets of electric fields $\{\bm E_i\}_{1\leq i\leq 5}$ and $\{\bm E'_i\}_{1\leq i\leq 5}$, which solve the Maxwell's equations \eqref{Eq:maxwell} with the complex tensors $\gamma$ and $\gamma'$ satisfying the uniform ellipticity condition \eqref{positive definite}. Then $\gamma$ can be uniquely reconstructed in $X_0$ with the following stability estimate,
\begin{align}
	\|\gamma - \gamma'\|_{W^{s,\infty}(X_0)}\le C \sum_{i=1}^{5} \|H_i-H'_i\|_{W^{s+2,\infty}(X)},
	\label{eq:stability}
\end{align}
 for any integer $s>0$ and some constant $C=C(s)$. 
 \end{theorem}
\begin{proof}
The above estimate is straightforward by noticing that two derivatives are taken in the reconstruction procedure for $\gamma$.
\end{proof}

\section{Fulfilling Hypothesis \ref{2 hypo}} \label{Fulfilling Hypothesis}
In this section, we assume that $\gamma$ is a diagonalizable constant tensor. We will take special CGO-like solutions of the Maxwell's equations \eqref{Eq:maxwell} and demonstrate that Hypothesis \ref{2 hypo} can be fulfilled with these solutions. By definition of the curl operator, it suffices to show that 
\begin{align}\label{tilde constraint matrice}
\tilde M_j = (\tilde Z_j\tilde H^T)^{sym}, \quad  1\leq j\leq 3,
\end{align}
are linearly independent in $S_2(\mathbb{C})$,  where $\tilde Z_j=\left[\nabla\lambda^j_1 | \nabla\lambda^j_2\right]$ and $\tilde H=[\nabla H_1 | \nabla H_2]$. We derive the following Helmholtz-type equation from \eqref{Eq:maxwell},
\begin{align}\label{eq:Helm}
-\nabla\cdot\tilde\gamma^{-1}\nabla H_i + H_i= 0, \quad \text{for} \enspace  1\leq i\leq 5,
\end{align}
where $\tilde\gamma = -\i\omega\mu J^T\gamma J$ and admits a decomposition $\tilde\gamma= QQ^T$ with $Q$ invertible. We take special CGO-like solutions of the form
\begin{align}\label{CGO like}
H_i=  e^{x\cdot Qu_i},
\end{align}
where the $u_i$ are vectors of unit length. Obviously, $u_i$ defined in \eqref{CGO like} satisfy \eqref{eq:Helm} and 
\begin{align}\label{tilde H}
\tilde H = QU\left[\begin{array}{cc} e^{x\cdot Qu_1} &0\\  0 &  e^{x\cdot Qu_2} \end{array}\right],
\end{align}
where $U=[u_1| u_2]$. Therefore, Hypothesis \ref{2 hypo}.1 can be easily fulfilled by choosing independent unit vectors $u_1= \be_1$, $u_2=\be_2$. Using the corresponding additional electric fields $\{\bE_{2+j}\}_{1\leq j\leq 3}$, Cramer's rule as in \eqref{cramer rule} yields the decompositions
\begin{align*}
\bE_{2+j}=\lambda^j_1 \bE_1+\lambda^j_2 \bE_2,  \quad \text{with} \enspace \lambda^j_1=e^{x\cdot Q(u_{2+j}-u_1)}\det(u_{2+j},u_2), \enspace \lambda^j_2=e^{x\cdot Q(u_{2+j}-u_2)}\det(u_1,u_{2+j}).
\end{align*}
Then by definition of $\tilde Z_j$, we get the following expression,
\begin{align}
\tilde Z_j = Q[\frac{H_{2+j}}{H_1}\det(u_{2+j},u_2)(u_{2+j}-u_1),\frac{H_{2+j}}{H_2}\det(u_1,u_{2+j})(u_{2+j}-u_2)].
\end{align}
Together with \eqref{tilde H}, straightforward calculations lead to 
\begin{align}
\tilde Z_j\tilde H^T = H_{2+j}Q[\det(u_{2+j},u_2)(u_{2+j}-u_1),\det(u_1,u_{2+j})(u_{2+j}-u_2]Q^T.
\end{align}
Using the fact that $u_{2+j}=(u_{2+j}\cdot u_1)u_1+(u_{2+j}\cdot u_2)u_2$, the above equation leads to 
\begin{align}
\tilde M_j = H_{2+j}Q \left[\begin{array}{cc} (u_{2+j}\cdot u_1)( (u_{2+j}\cdot u_1)-1) & (u_{2+j}\cdot u_1)(u_{2+j}\cdot u_2)\\  (u_{2+j}\cdot u_1)(u_{2+j}\cdot u_2) &  (u_{2+j}\cdot u_2)( (u_{2+j}\cdot u_2)-1) \end{array}\right]Q^T,
\end{align}
where $u_1= \be_1$, $u_2=\be_2$. Therefore, it is easy to find $u_{2+j}$ vectors of unit length such that $\tilde M_j$ are linearly independent in $S_2(\mathbb{C})$.
\begin{remark}
To derive local reconstruction formulas for more general tensors (e.g. $\C^{1,\alpha}(X)$), we need local independence conditions of $\{M_j\}_j$ and we need to control the local behavior of solutions by well-chosen boundary conditions. This is done by means of a Runge approximation. For details, we refer the reader to \cite{Guo2012a},\cite{Bal2012} and \cite{Guo2013b}.
\end{remark}

\section{Numerical experiments}\label{num simu}

In this section we present some numerical simulations based on synthetic data to validate the reconstruction algorithms from the previous section. 
\subsection{Preliminary}
 We decompose $\gamma=\sigma+\i\omega\varepsilon$ into the following form with six unknown coefficients  $\{\sigma_i\}_{1\leq i\leq 3}$, $\{\varepsilon_i\}_{1\leq i\leq 3}$ respectively for $\sigma$ and $\epsilon$,
 
 \begin{align}\label{3coef}
\gamma= \left[\begin{array}{cc} \sigma_1 & \sigma_2\\  \sigma_2 &  \sigma_3\end{array}\right]+\i\omega\left[\begin{array}{cc} \varepsilon_1 &\varepsilon_2\\ \varepsilon_2 & \varepsilon_3\end{array}\right], 
\end{align}
where each coefficient can be explicitly reconstructed by solving the overdetermined linear system \eqref{rec sys} using least-square method.  

In the numerical experiments below, we take the domain of reconstruction to be the square $X=[-1,1]^2$ and use the notation $\mathbf{x}= (x,y)$. We use a $\mathsf{N+1\times N+1}$ square grid with $\mathsf{N}=80$, the tensor product of the equi-spaced subdivision $\mathsf{x =-1:h:1}$ with $\mathsf{h=2/N}$. The synthetic data $H$ are generated by solving the Maxwell's equations \eqref{Eq:maxwell} for {\em known} conductivity $\sigma$ and electric permittivity $\varepsilon$, using a finite difference method implemented with {\tt MatLab}. We refer to these data as the "noiseless" data. To simulate noisy data, the internal magnetic fields $H$ are perturbed by adding Gaussian random matrices with zero means. The standard derivations $\alpha$ are chosen to be $0.1\%$ of the average value of $|H|$.

We use the relative $L^2$ error to measure the quality of the reconstructions. This error is defined as the $L^2$-norm of the difference between the reconstructed coefficient and the true coefficient, divided by the $L^2$-norm of the true coefficient.  $\mathcal{E}^C_{\sigma_i}$, $\mathcal{E}^N_{\sigma_i}$, $\mathcal{E}^C_{\varepsilon_i}$, $\mathcal{E}^N_{\varepsilon_i}$ with $1\leq i\leq 3$ denote respectively the relative $L^2$ error in the reconstructions from clean and noisy data for $\sigma_i$ and $\varepsilon_i$.
\paragraph{Regularization procedure.} We use a total variation method as the denoising procedure by minimizing the following functional,
\begin{align}
\mathcal{O}(\mathbf{f})= \frac{1}{2}\|\mathbf{f}-\mathbf{f}_{\text{rc}}\|^2_2 +\rho \|\Gamma \mathbf{f}\|_{\text{\scriptsize TV}}
\end{align}
where $\mathbf{f}_{\text{rc}}$ denotes the explicit reconstructions of the coefficients of $\sigma$ and $\varepsilon$, $\Gamma$ denotes discretized version of the gradient operator. We choose the $l^1$-norm as the regularization TV norm for discontinuous, piecewise constant, coefficients. In this case, the minimization problem can be solved using the split Bregman method presented in \cite{Goldstein2009}. To recover smooth coefficients, we minimize the following least square problem with the $l^2$-norm for the regularization term,
\begin{align}
\mathcal{O}(\mathbf{f})= \frac{1}{2}\|\mathbf{f}-\mathbf{f}_{\text{rc}}\|^2_2 +\rho \|\Gamma \mathbf{f}\|_2^2,
\end{align}
where the Tikhonov regularization functional admits an explicit solution $\mathbf{f}=(\Imm+\rho \Gamma^*\Gamma)^{-1}\mathbf{f}_{\text{rc}}$. The regularization methods are used when the data are differentiated. 

\subsection{Simulation results}
\paragraph{Simulation 1.} In the first experiment, we intend to reconstruct the smooth coefficients $\{\sigma_i, \varepsilon_i\}_{1\leq i\leq 3}$ defined in \eqref{3coef}. The coefficients are given by,
\begin{align*}
\left\{\begin{array}{lll}
\sigma_1 = 2+ \sin(\pi x)\sin(\pi y)\\
 \sigma_2 = 0.5\sin(2\pi x)\\
\sigma_3 = 1.8+e^{-15(x^2+y^2)} +e^{-15((x-0.6)^2+(y-0.5)^2)} - e^{-15((x+0.4)^2+(y+0.6)^2)} 
\end{array}\right.
\end{align*}
and 
\begin{align*}
\left\{\begin{array}{lll}
\varepsilon_1 = 2- \sin(\pi x)\sin(\pi y)\\
\varepsilon_2 = 0.5\sin(2\pi y)\\
\varepsilon_3 = 1.8+e^{-12(x^2+y^2)} +e^{-12((x+0.6)^2+(y-0.5)^2)} - e^{-12((x-0.4)^2+(y+0.6)^2)} .
\end{array}\right.
\end{align*}

We performed two sets of reconstructions using clean and noisy synthetic data respectively. The $l_2$-regularization procedure is used in this simulation. For the noisy data, the noise level is $\alpha=0.1\%$. The results of the numerical experiment are shown in Figure \ref{E1sigma} and Figure \ref{E1epsilon}. The relative $L^2$ errors in the reconstructions are $\mathcal{E}^C_{\sigma_1}=0.3\%$, $\mathcal{E}^N_{\sigma_1}=5.1\%$, $\mathcal{E}^C_{\sigma_2}=0.8\%$, $\mathcal{E}^N_{\sigma_2}=33.4\%$, $\mathcal{E}^C_{\sigma_3}=0.2\%$, $\mathcal{E}^N_{\sigma_3}=4.9\%$; $\mathcal{E}^C_{\varepsilon_1}=0.1\%$, $\mathcal{E}^N_{\varepsilon_1}=5.8\%$, $\mathcal{E}^C_{\varepsilon_2}=0.5\%$, $\mathcal{E}^N_{\varepsilon_2}=30.0\%$, $\mathcal{E}^C_{\varepsilon_3}=0.1\%$, $\mathcal{E}^N_{\varepsilon_3}=4.8\%$. 

\begin{figure}[htp]
  \centering
  \subfigure[true $\sigma_1$]{
      \includegraphics[width=37mm,height=35mm]{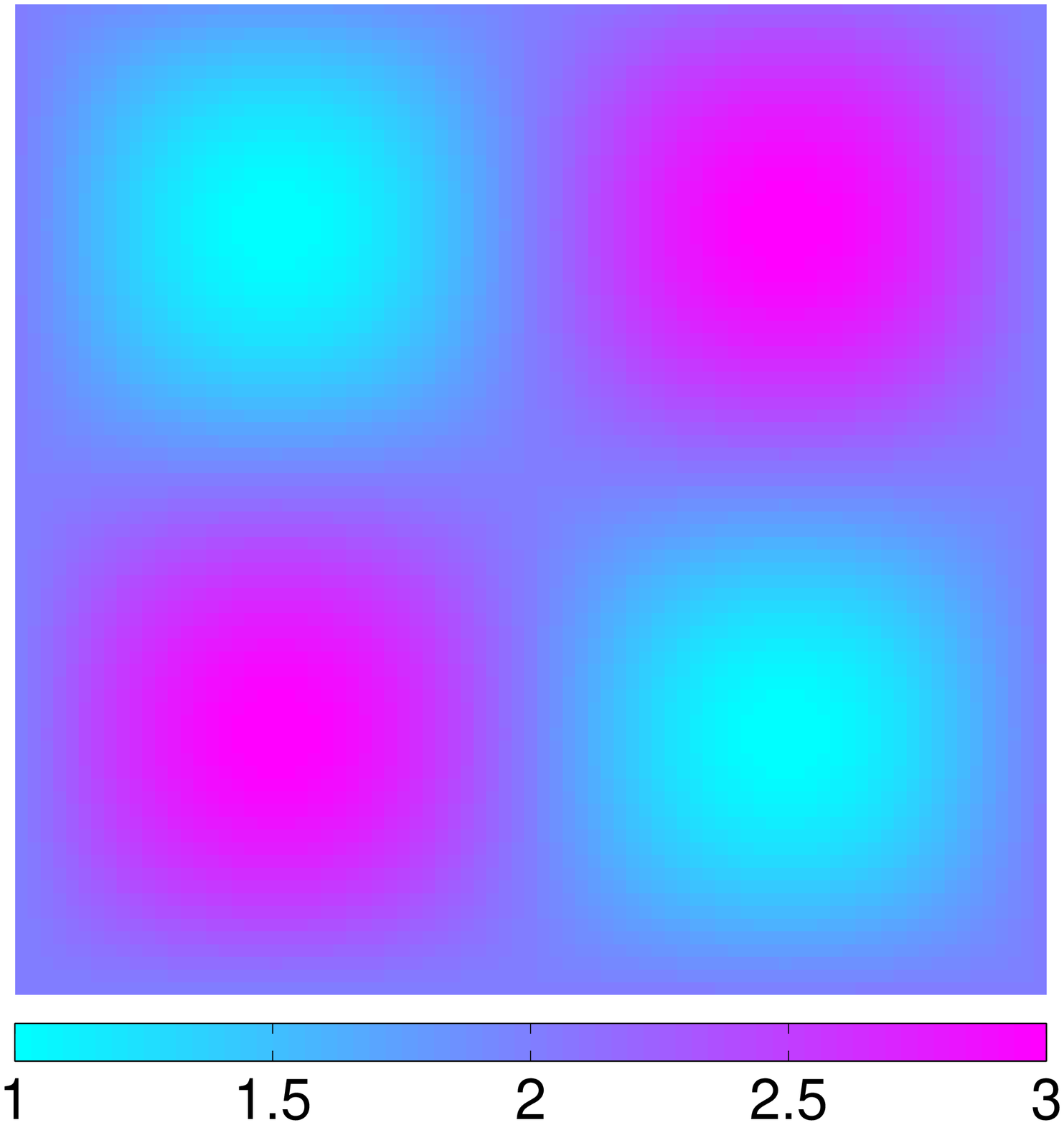}
      \label{ex1txi}
      } 
  \subfigure[$\sigma_1$ ($\alpha=0\%$)]{    
     \includegraphics[width=37mm,height=35mm]{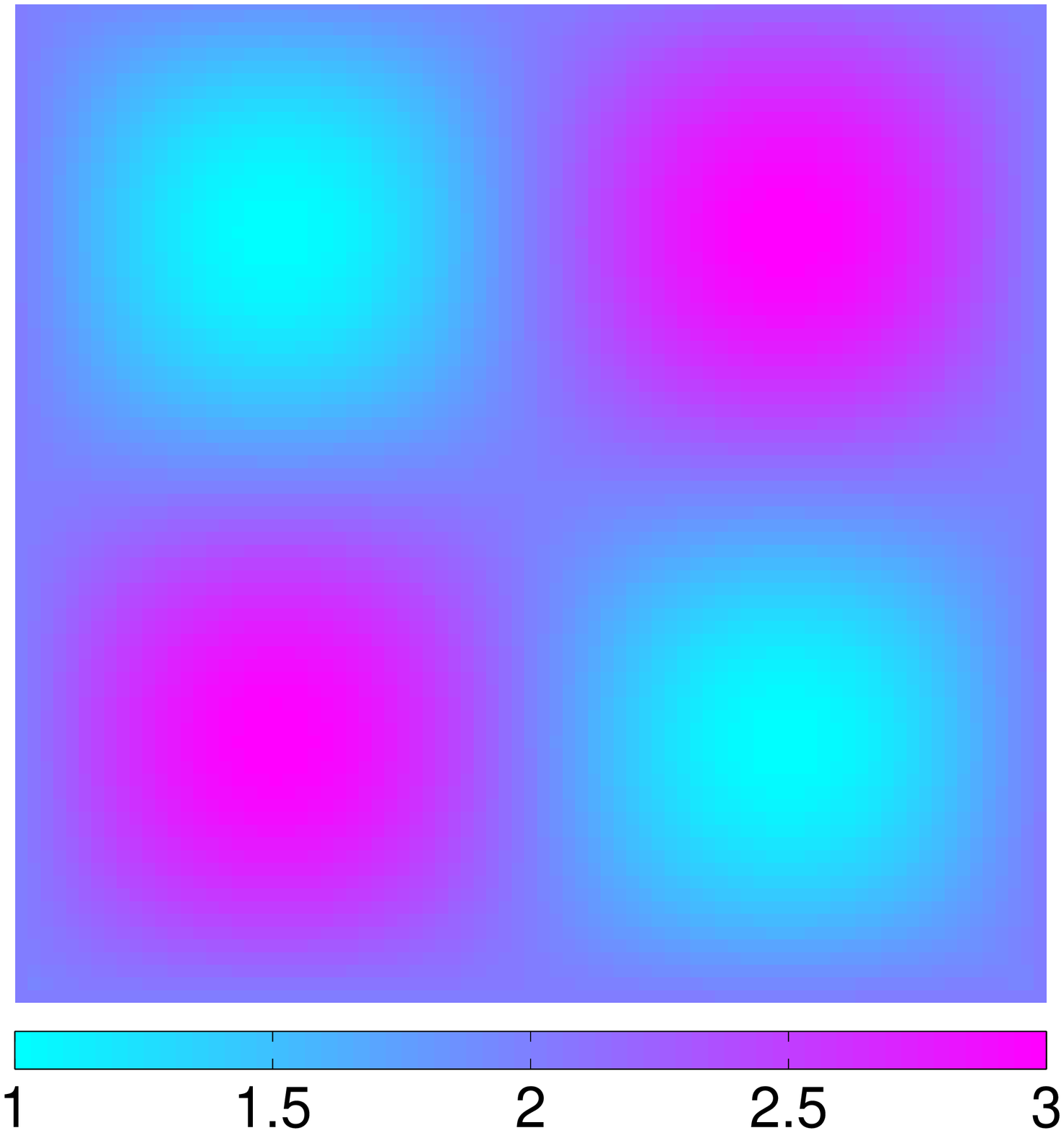}
     \label{ex1cxi}
     }
   \subfigure[$\sigma_1$ ($\alpha=0.1\%$)]{
    \includegraphics[width=37mm,height=35mm]{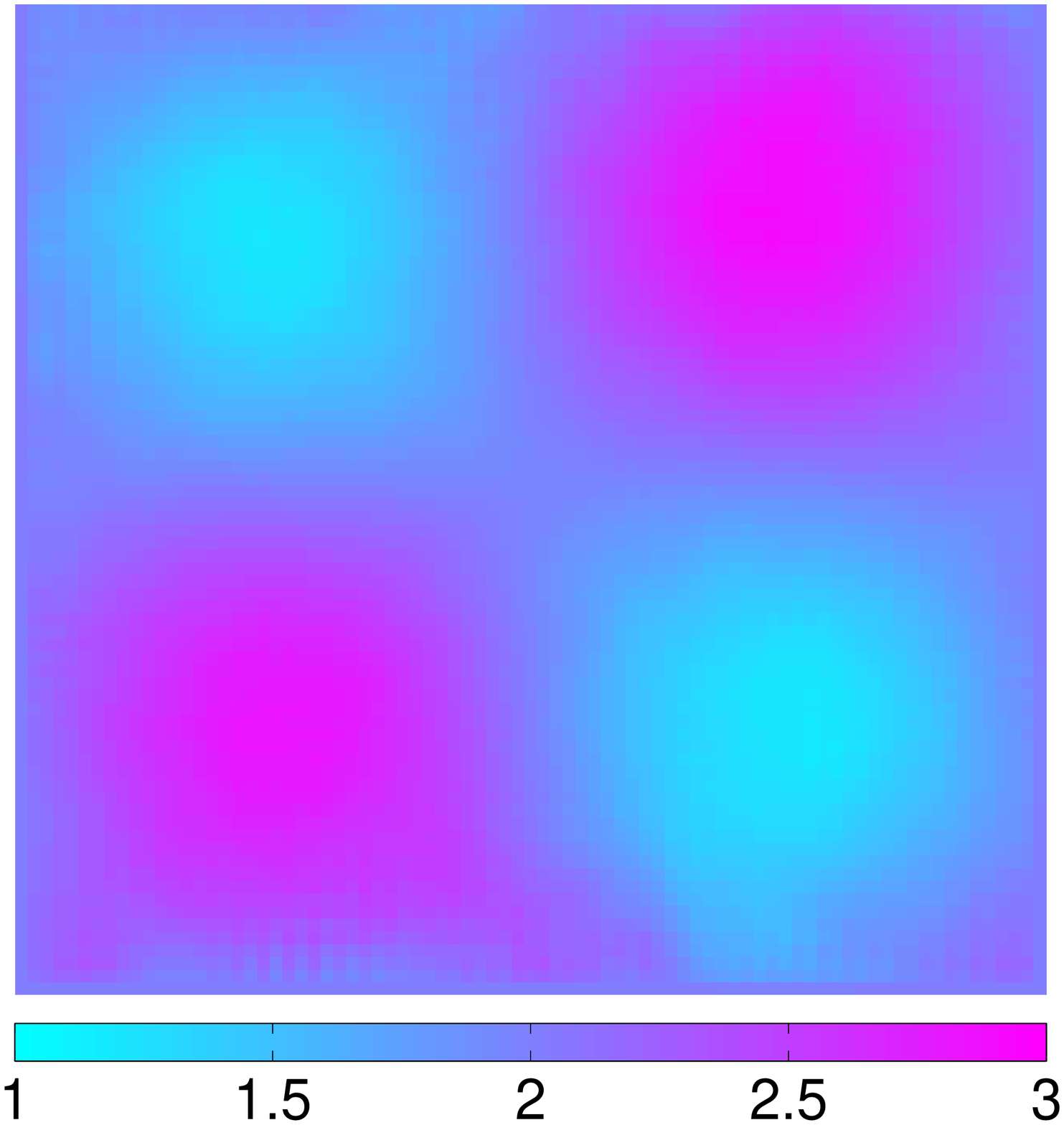}
    \label{ex1nxi}
    }
    \subfigure[$\sigma_1$ at \{$y=-0.5$\}]{
     \includegraphics[trim=10mm 5mm 10mm 0mm,clip,width=35mm,height=38mm]{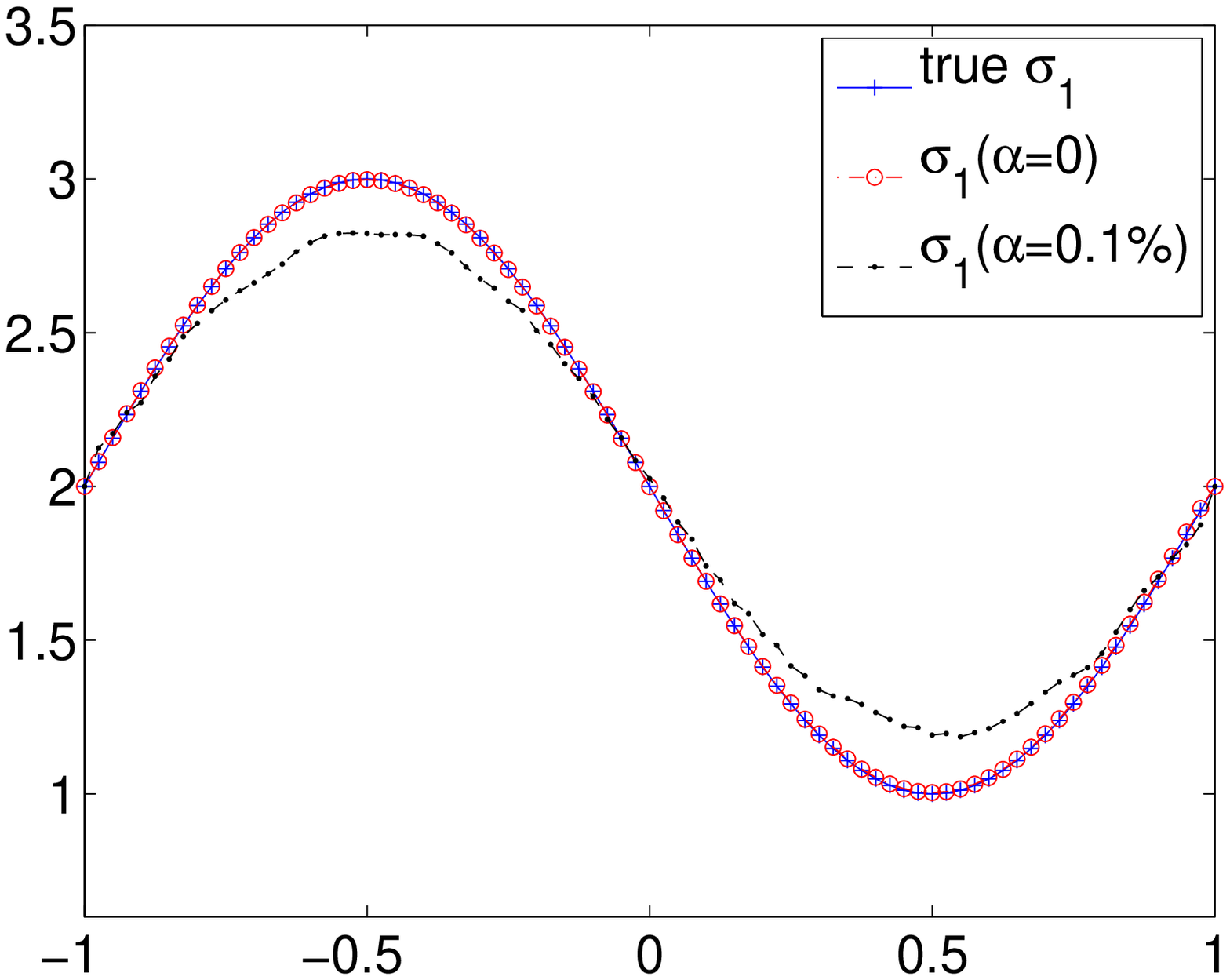} 
     \label{ex1rxi}
     }

     \subfigure[true $\sigma_2$]{
      \includegraphics[width=37mm,height=35mm]{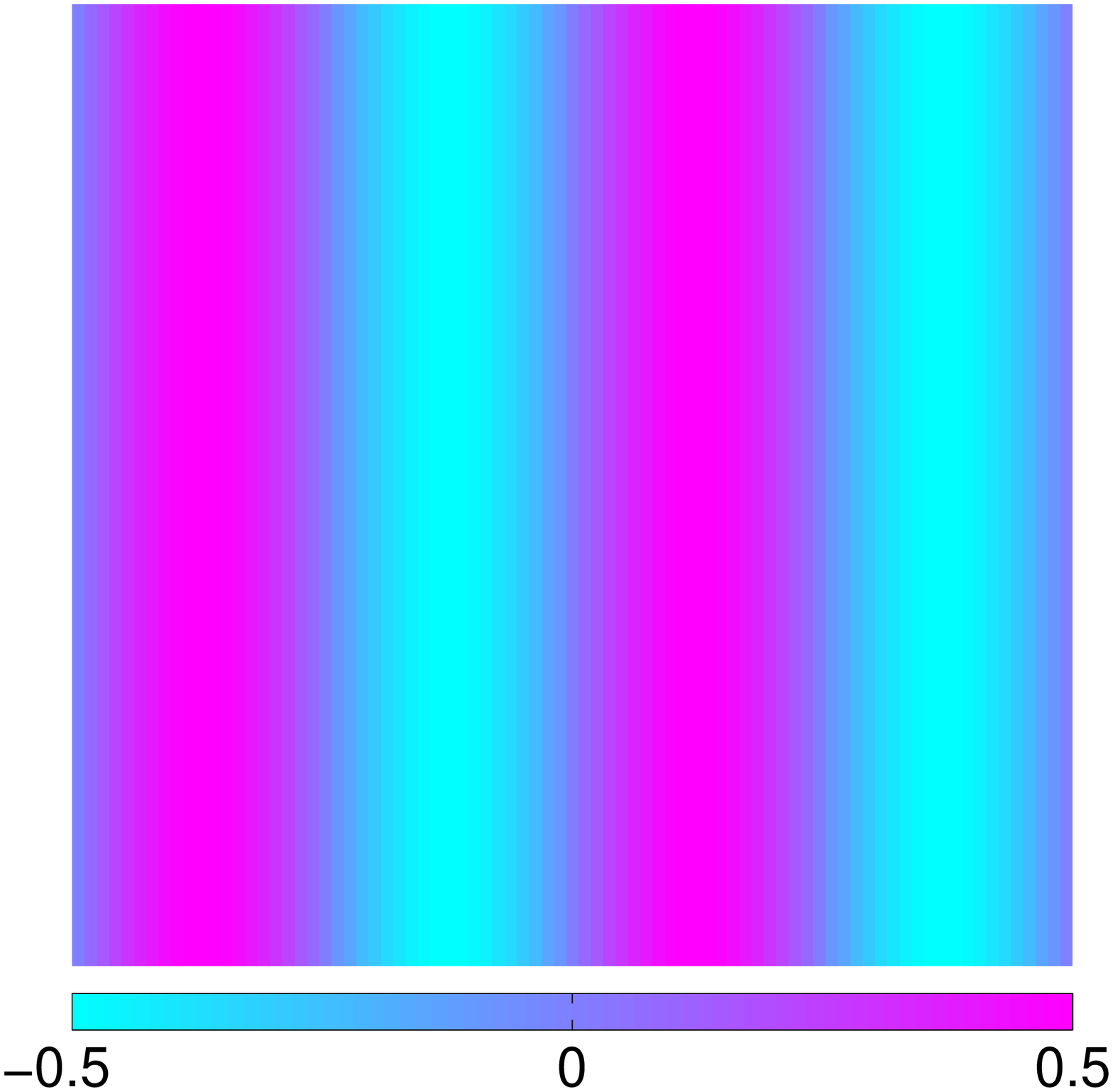}
      \label{ex1ttau}
      } 
    \subfigure[$\sigma_2$ ($\alpha=0\%$)]{ 
     \includegraphics[width=37mm,height=35mm]{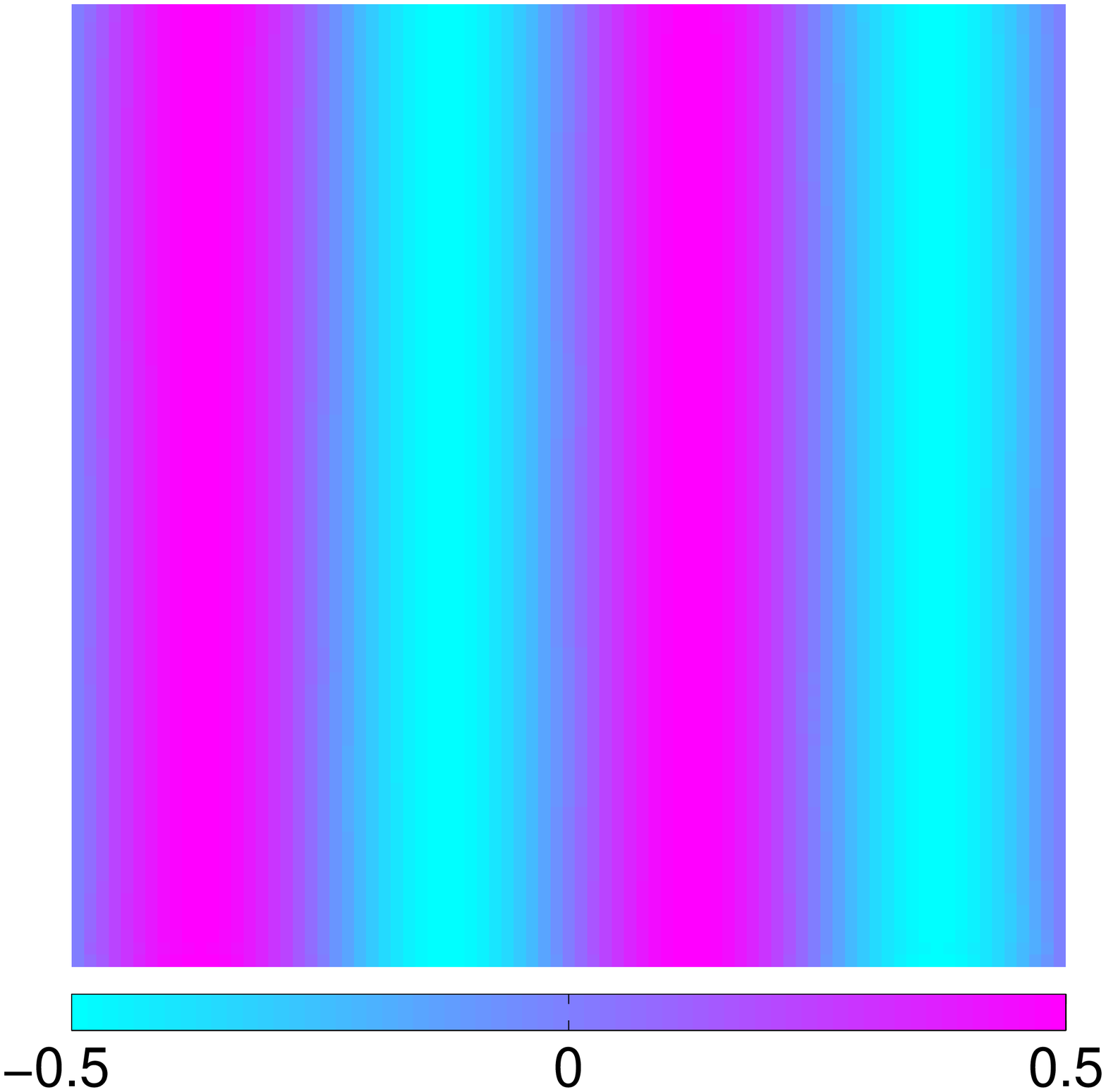}
     \label{ex1ctau}
     }
   \subfigure[$\sigma_2$ ($\alpha=0.1\%$)]{ 
    \includegraphics[width=37mm,height=35mm]{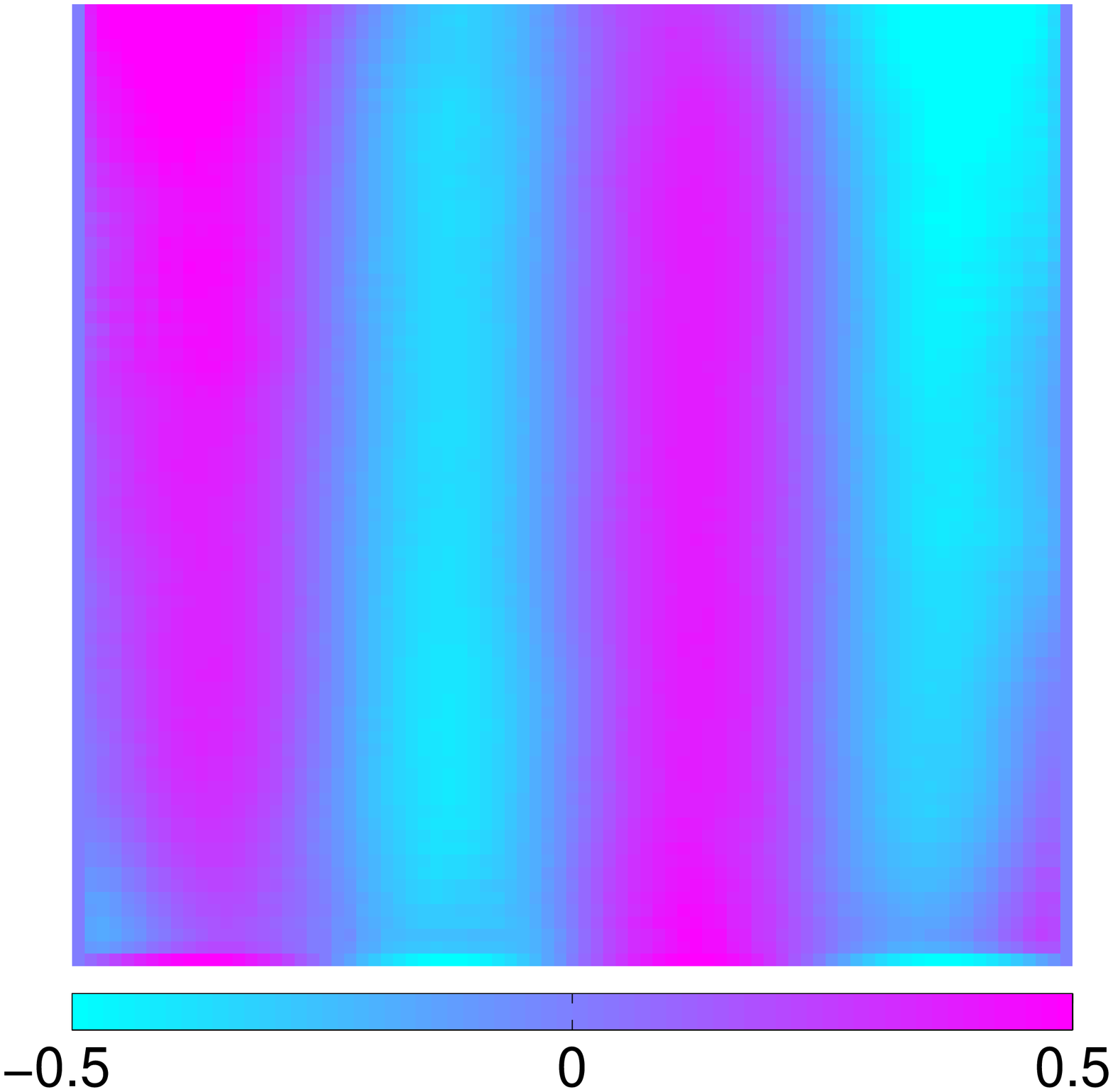}
    \label{ex1ntau}
    }
     \subfigure[$\sigma_2$ at \{$y=-0.5$\}]{
     \includegraphics[trim=10mm 5mm 10mm 0mm,clip,width=35mm,height=38mm]{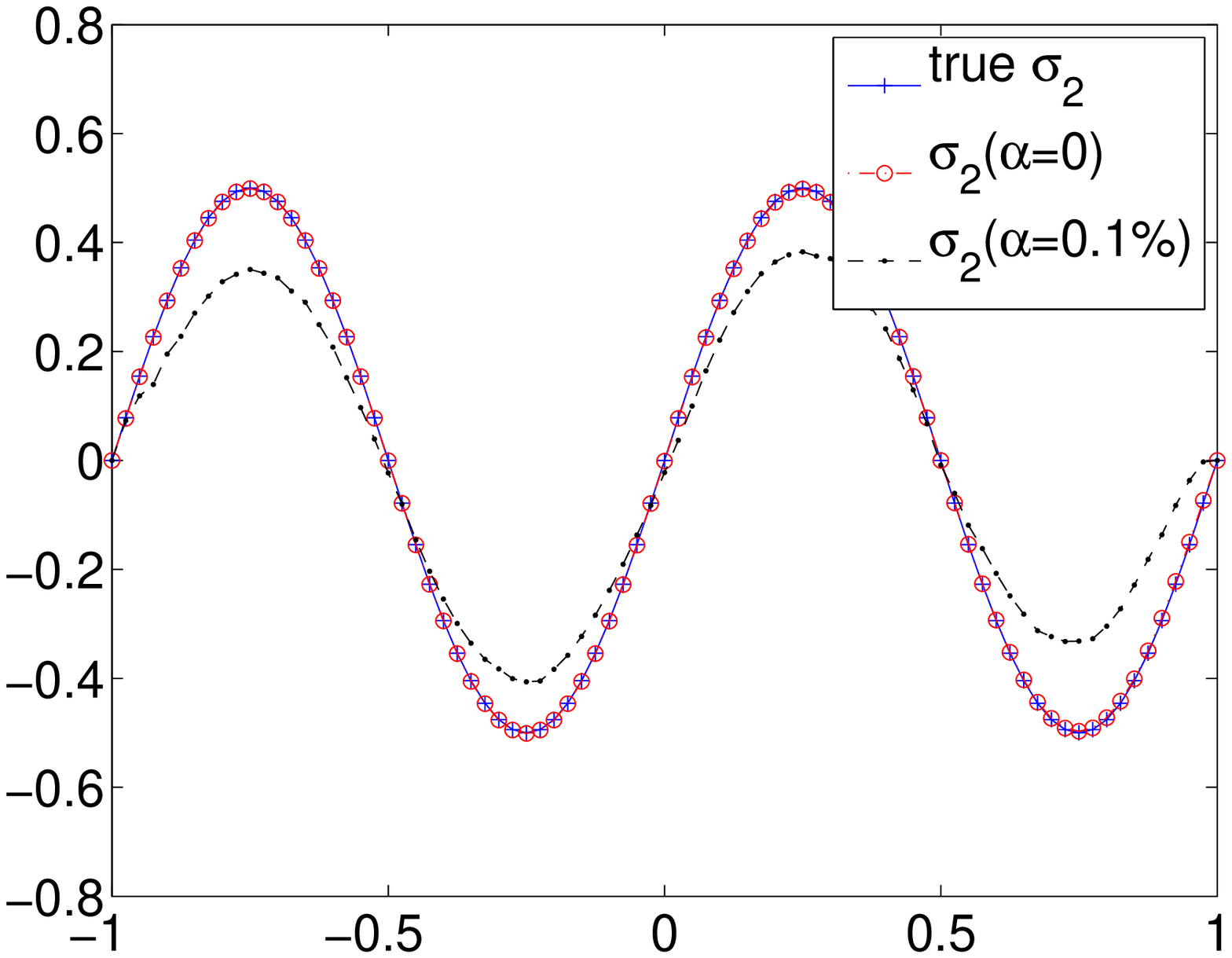} 
     \label{ex1rtau}
     }
     
     \subfigure[true $\sigma_3$]{
      \includegraphics[width=37mm,height=35mm]{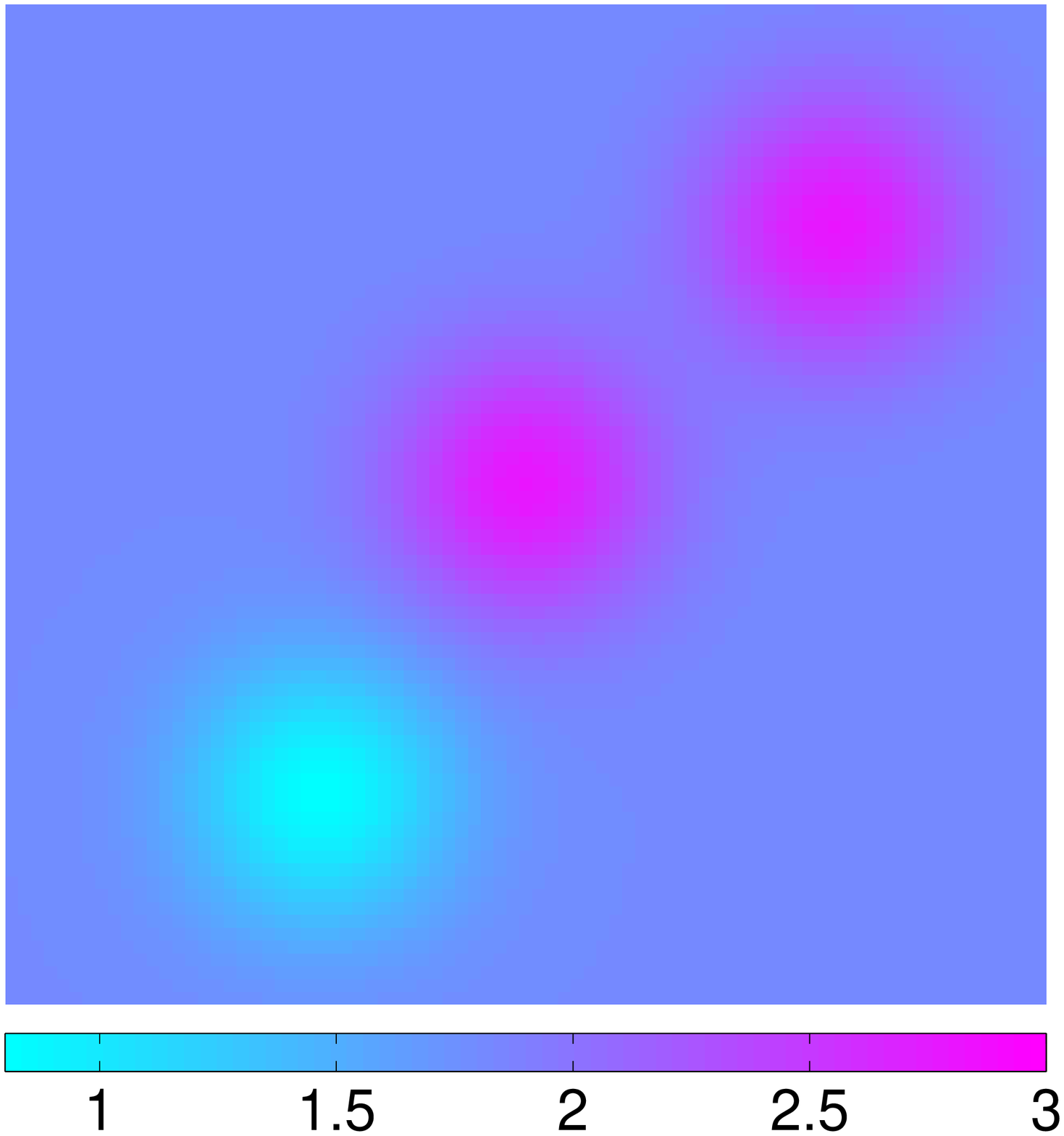}
      \label{ex1tbeta}
      } 
   \subfigure[$\sigma_3$ $(\alpha=0\%)$]{  
     \includegraphics[width=37mm,height=35mm]{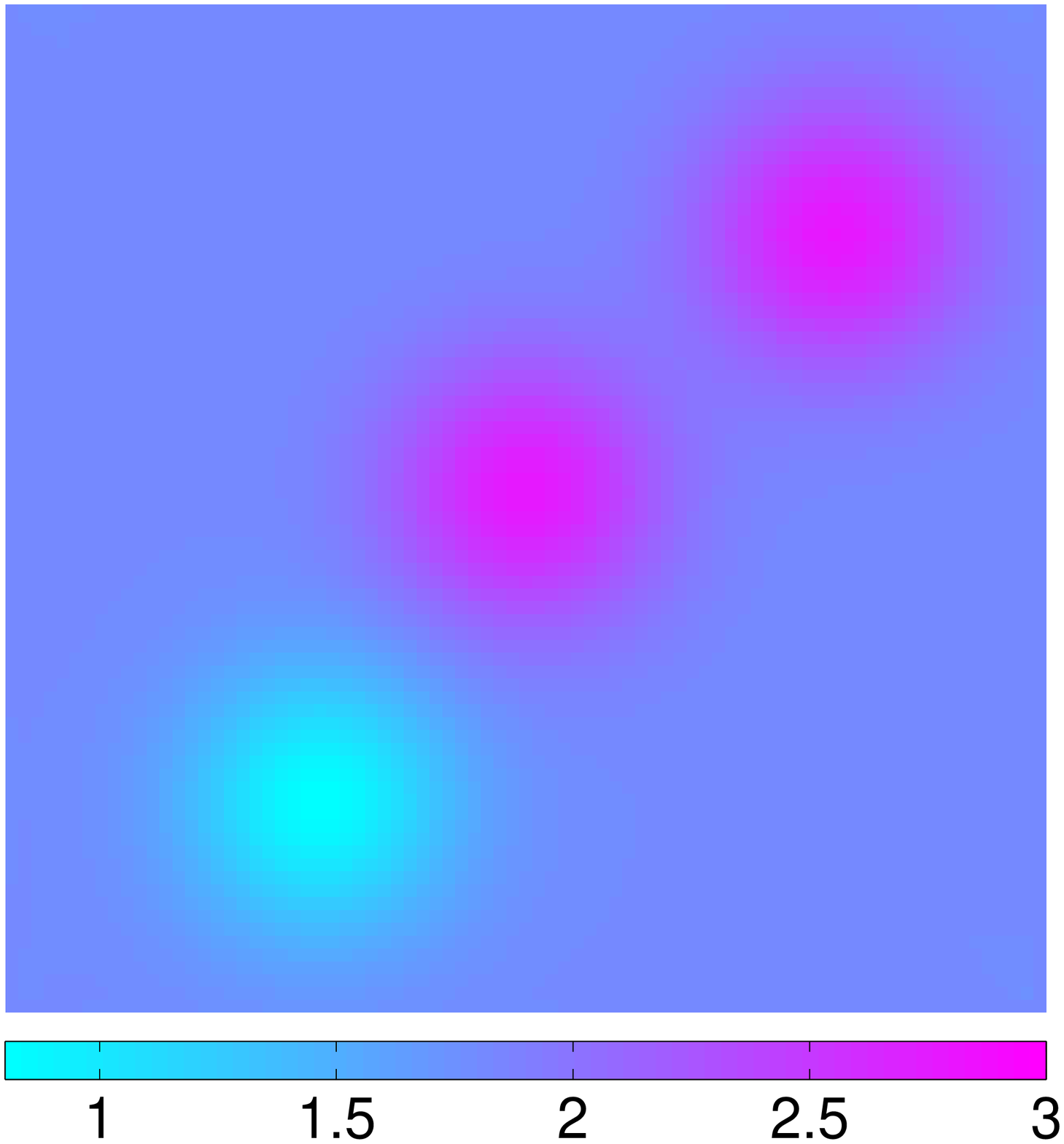}
     \label{ex1cbeta}
     }
   \subfigure[$\sigma_3$ ($\alpha=0.1\%$)]{ 
    \includegraphics[width=37mm,height=35mm]{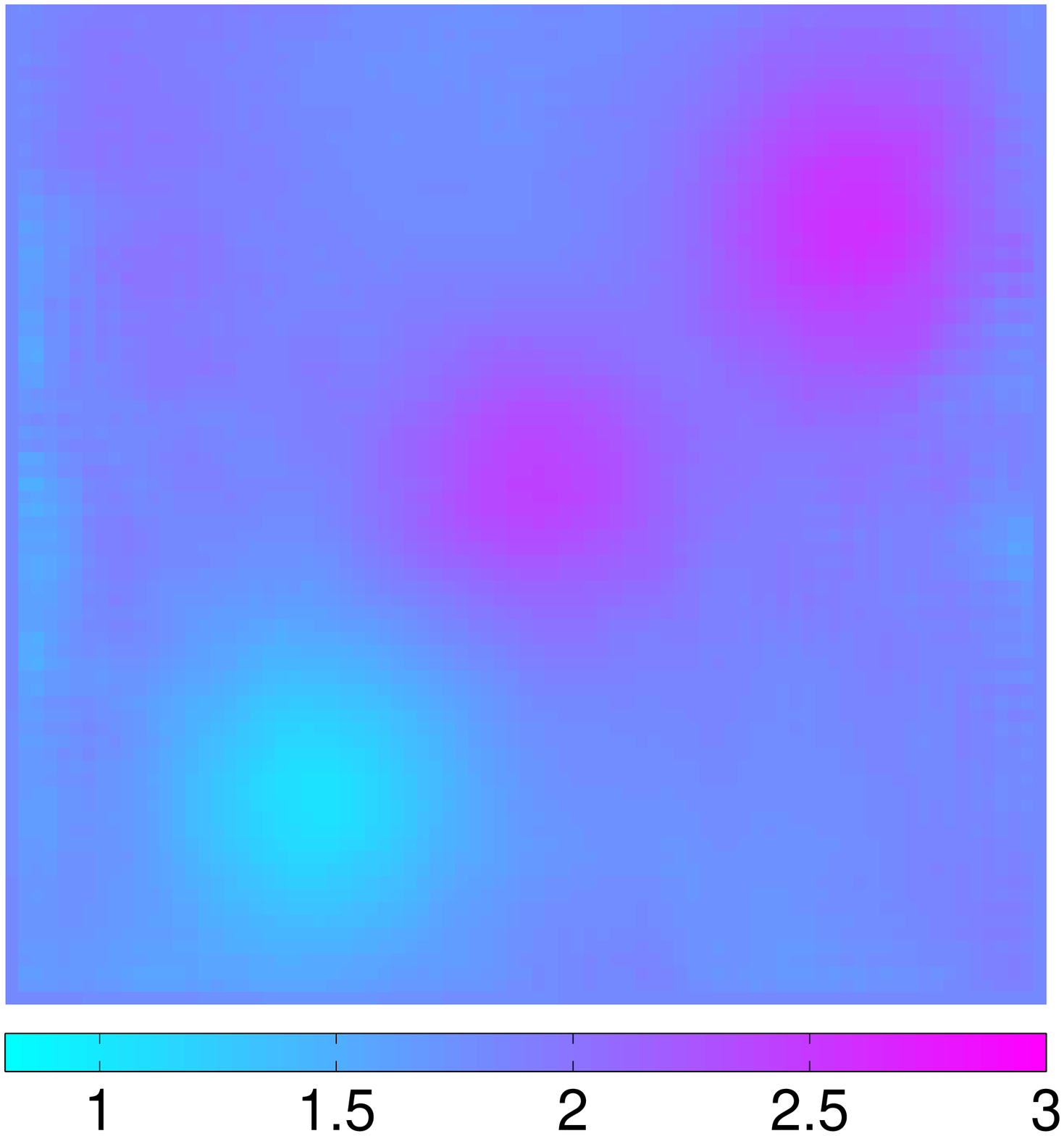}
    \label{ex1nbeta}
    }
   \subfigure[$\sigma_3$ at $\{y=0\}$]{ 
     \includegraphics[trim=10mm 5mm 10mm 0mm,clip,width=35mm,height=38mm]{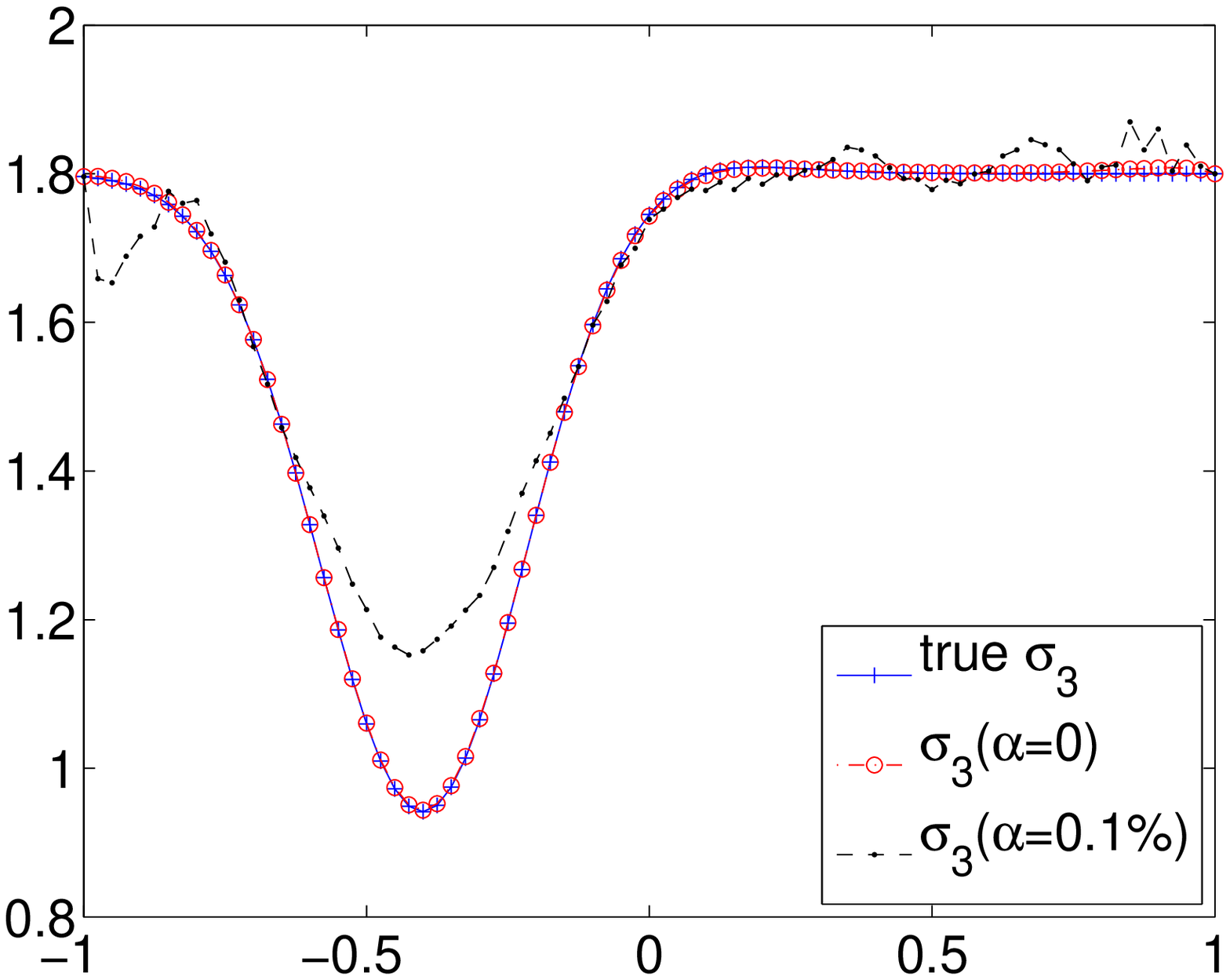} 
     \label{ex1rbeta}
     }
    \caption{$\sigma$ in Simulation 1. \subref{ex1txi}\&\subref{ex1ttau}\&\subref{ex1tbeta}: true values of $(\sigma_1, \sigma_2,\sigma_3)$. \subref{ex1cxi}\&\subref{ex1ctau}\&\subref{ex1cbeta}: reconstructions with noiseless data. \subref{ex1nxi}\&\subref{ex1ntau}\&\subref{ex1nbeta}: reconstructions with noisy data($\alpha=0.1\%$). \subref{ex1rxi}\&\subref{ex1rtau}\&\subref{ex1rbeta}: cross sections along $\{y=-0.5\}$.}
\label{E1sigma}
\end{figure}

\begin{figure}[htp]
  \centering
  \subfigure[true $\varepsilon_1$]{
      \includegraphics[width=37mm,height=35mm]{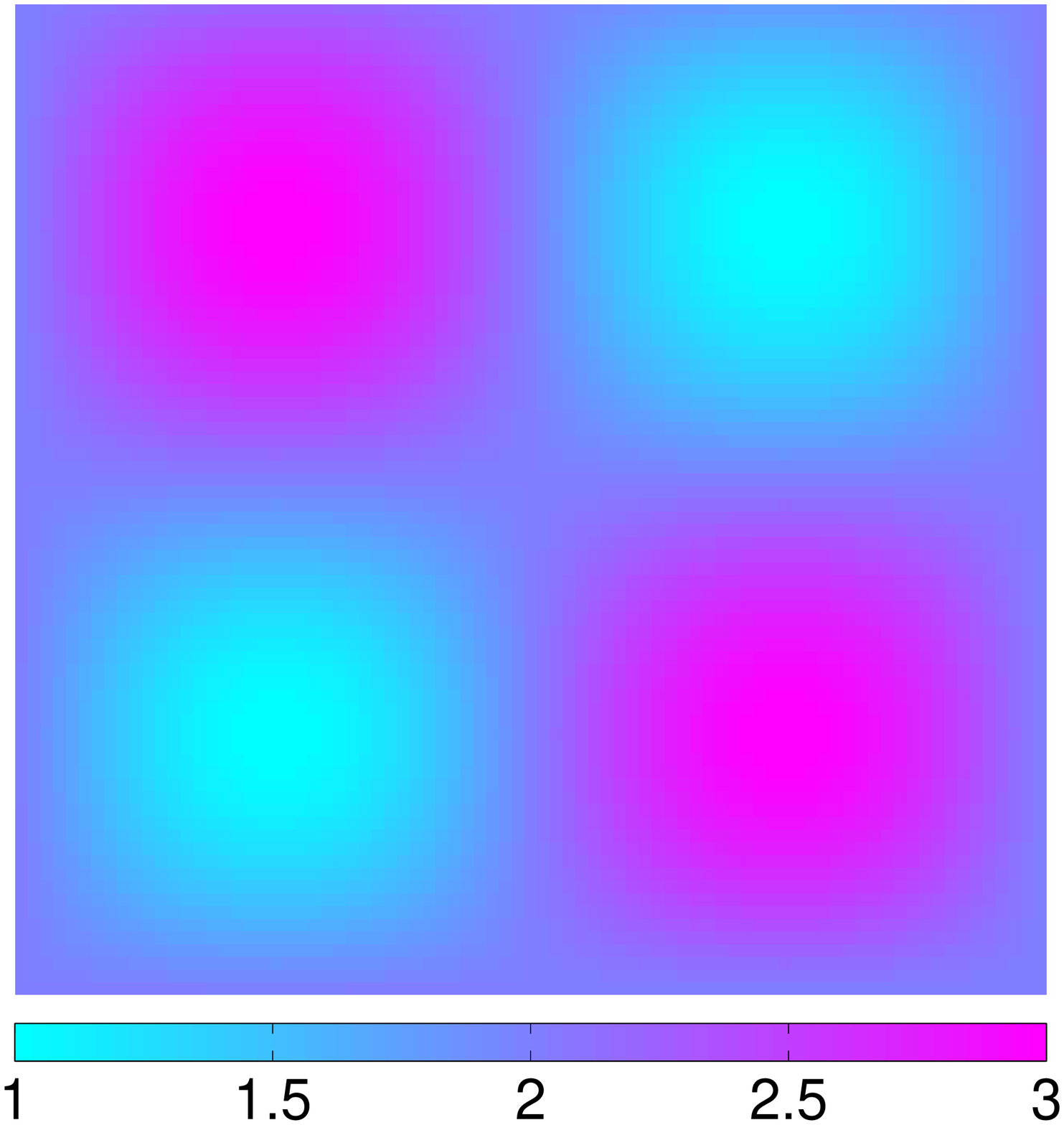}
      \label{ex1txi}
      } 
  \subfigure[$\varepsilon_1$ ($\alpha=0\%$)]{    
     \includegraphics[width=37mm,height=35mm]{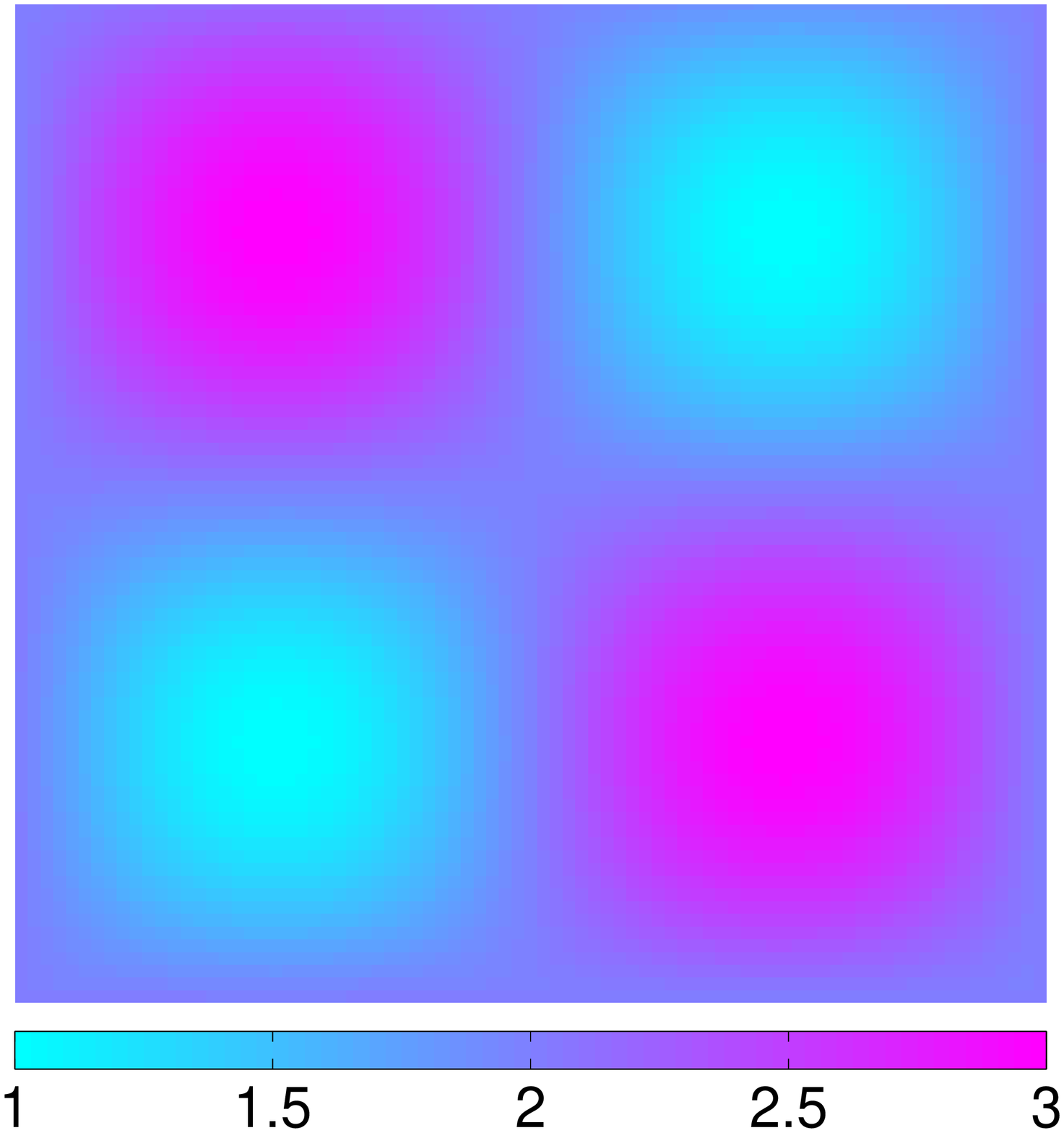}
     \label{ex1cxi}
     }
   \subfigure[$\varepsilon_1$ ($\alpha=0.1\%$)]{
    \includegraphics[width=37mm,height=35mm]{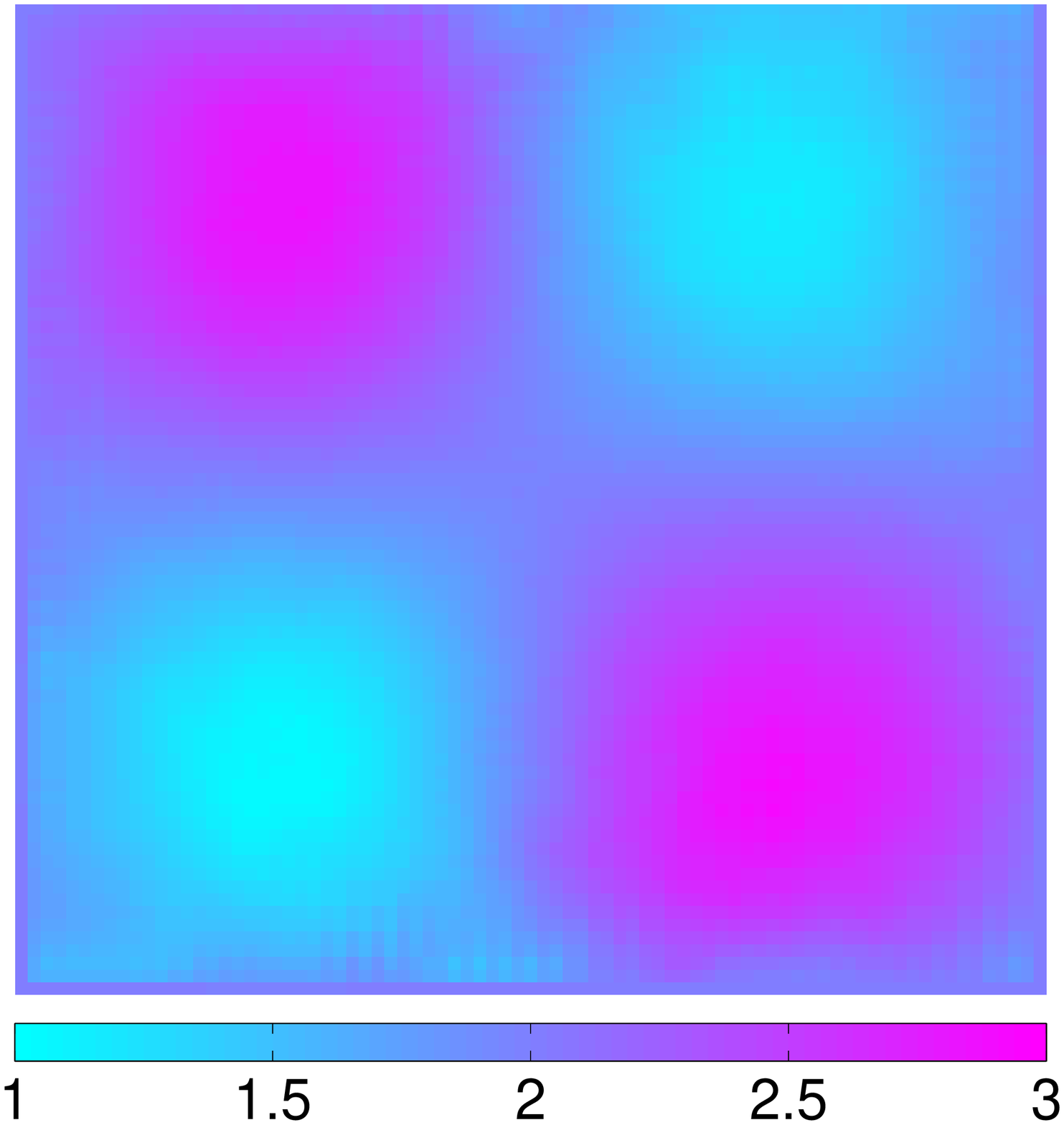}
    \label{ex1nxi}
    }
    \subfigure[$\varepsilon_1$ at \{$y=-0.5$\}]{
     \includegraphics[trim=10mm 5mm 10mm 0mm,clip,width=35mm,height=38mm]{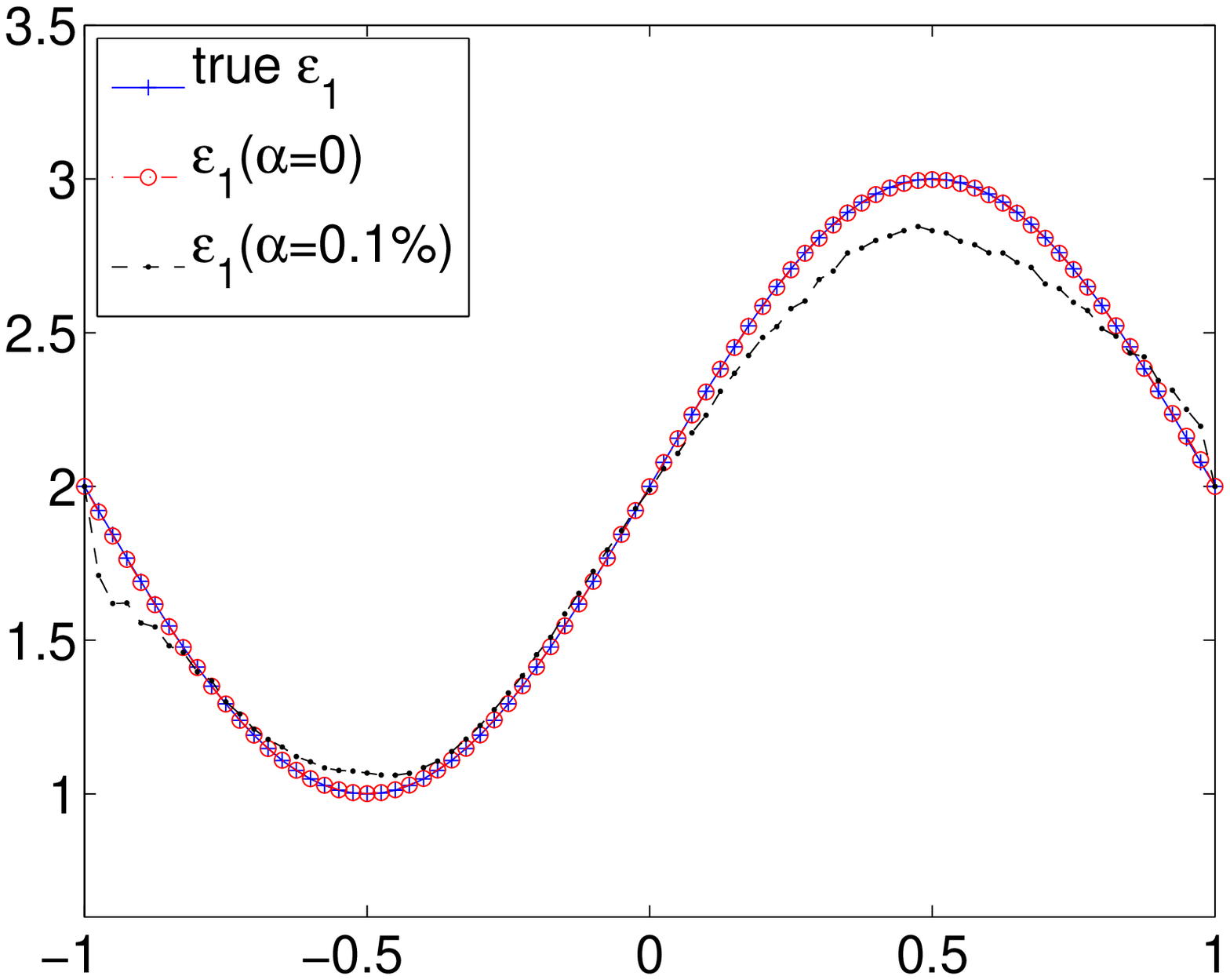} 
     \label{ex1rxi}
     }

     \subfigure[true $\varepsilon_2$]{
      \includegraphics[width=37mm,height=35mm]{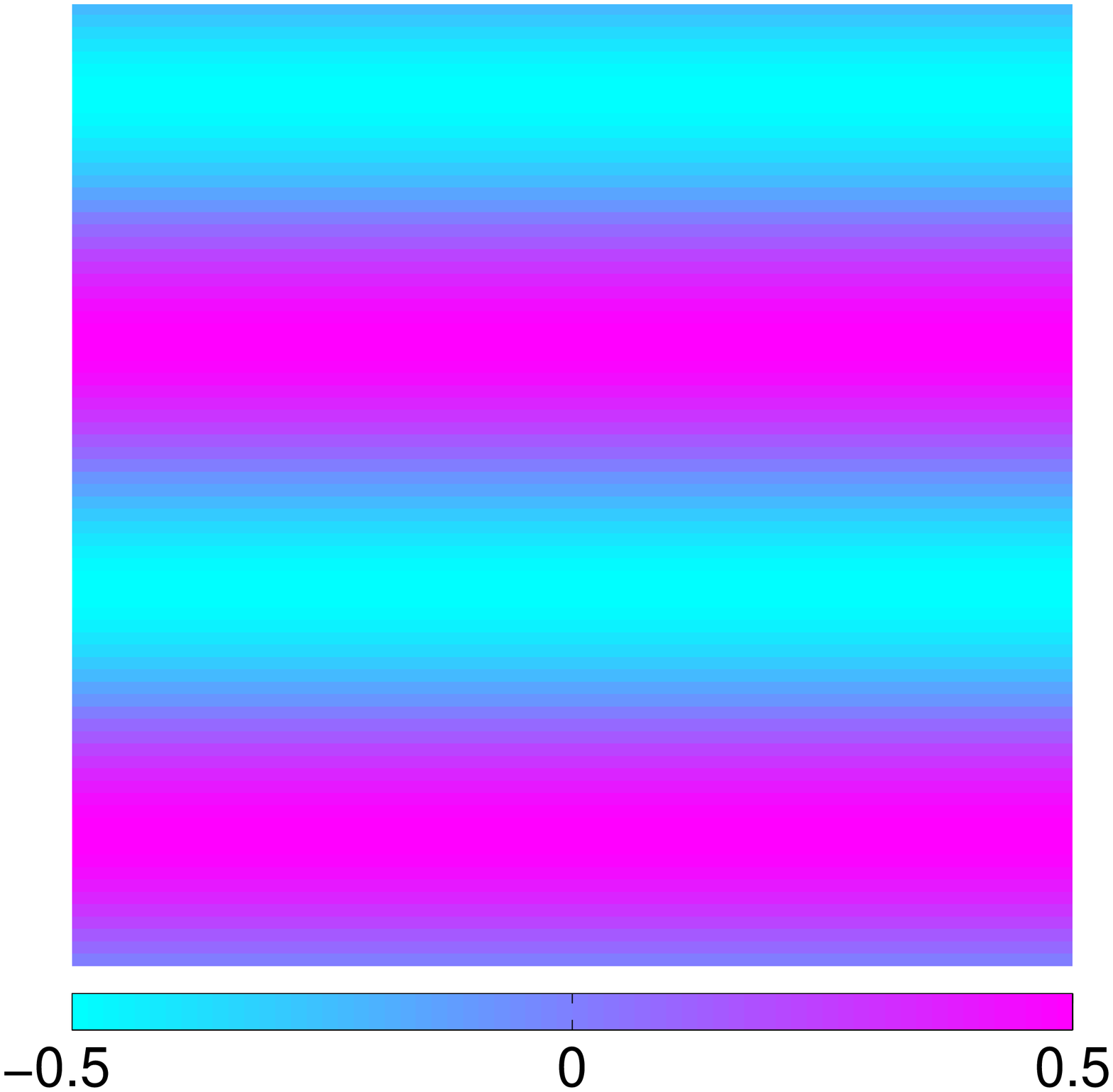}
      \label{ex1ttau}
      } 
    \subfigure[$\varepsilon_2$ ($\alpha=0\%$)]{ 
     \includegraphics[width=37mm,height=35mm]{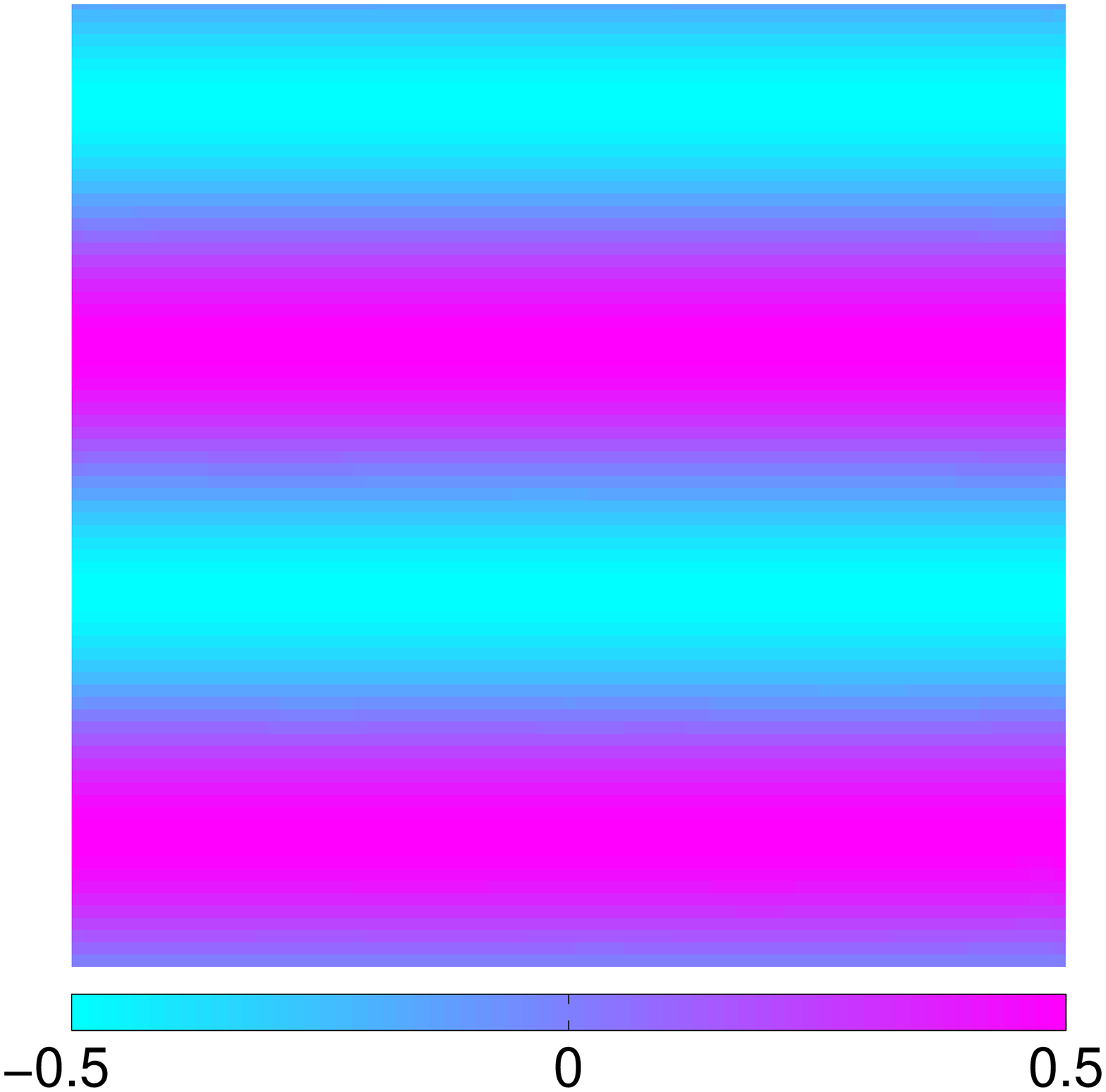}
     \label{ex1ctau}
     }
   \subfigure[$\varepsilon_2$ ($\alpha=0.1\%$)]{ 
    \includegraphics[width=37mm,height=35mm]{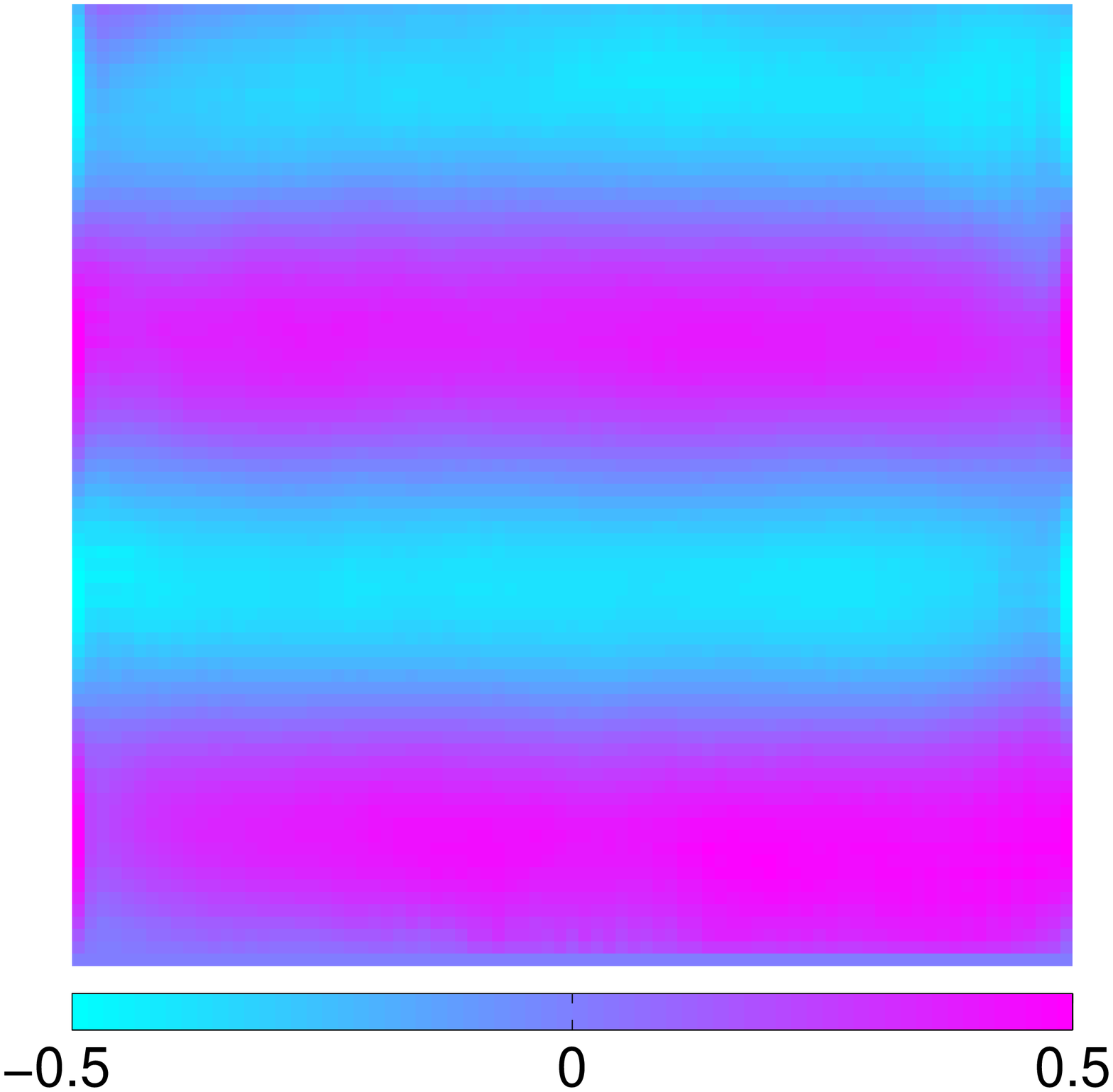}
    \label{ex1ntau}
    }
     \subfigure[$\varepsilon_2$ at \{$y=-0.5$\}]{
     \includegraphics[trim=10mm 5mm 10mm 0mm,clip,width=35mm,height=38mm]{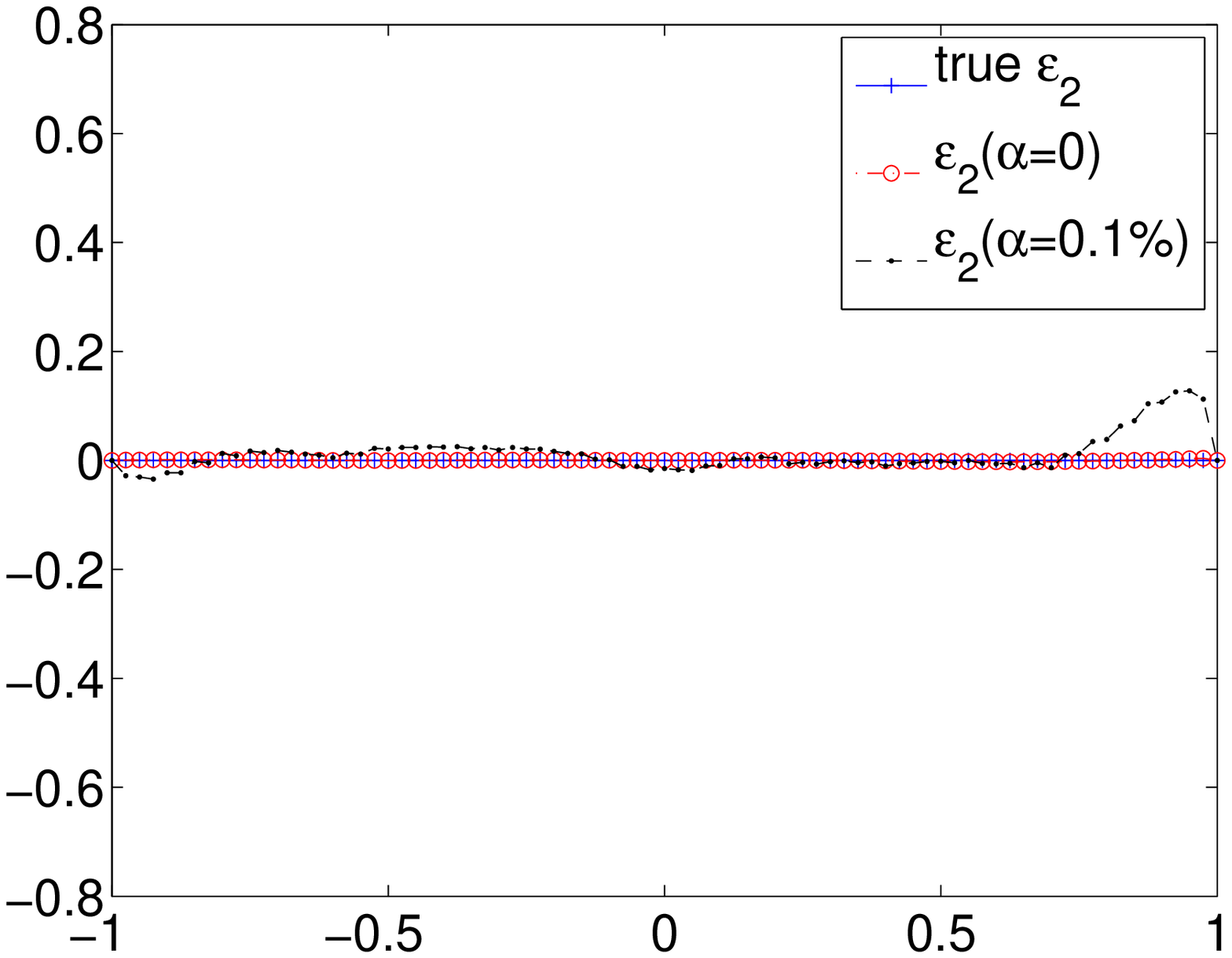} 
     \label{ex1rtau}
     }
     
     \subfigure[true $\varepsilon_3$]{
      \includegraphics[width=37mm,height=35mm]{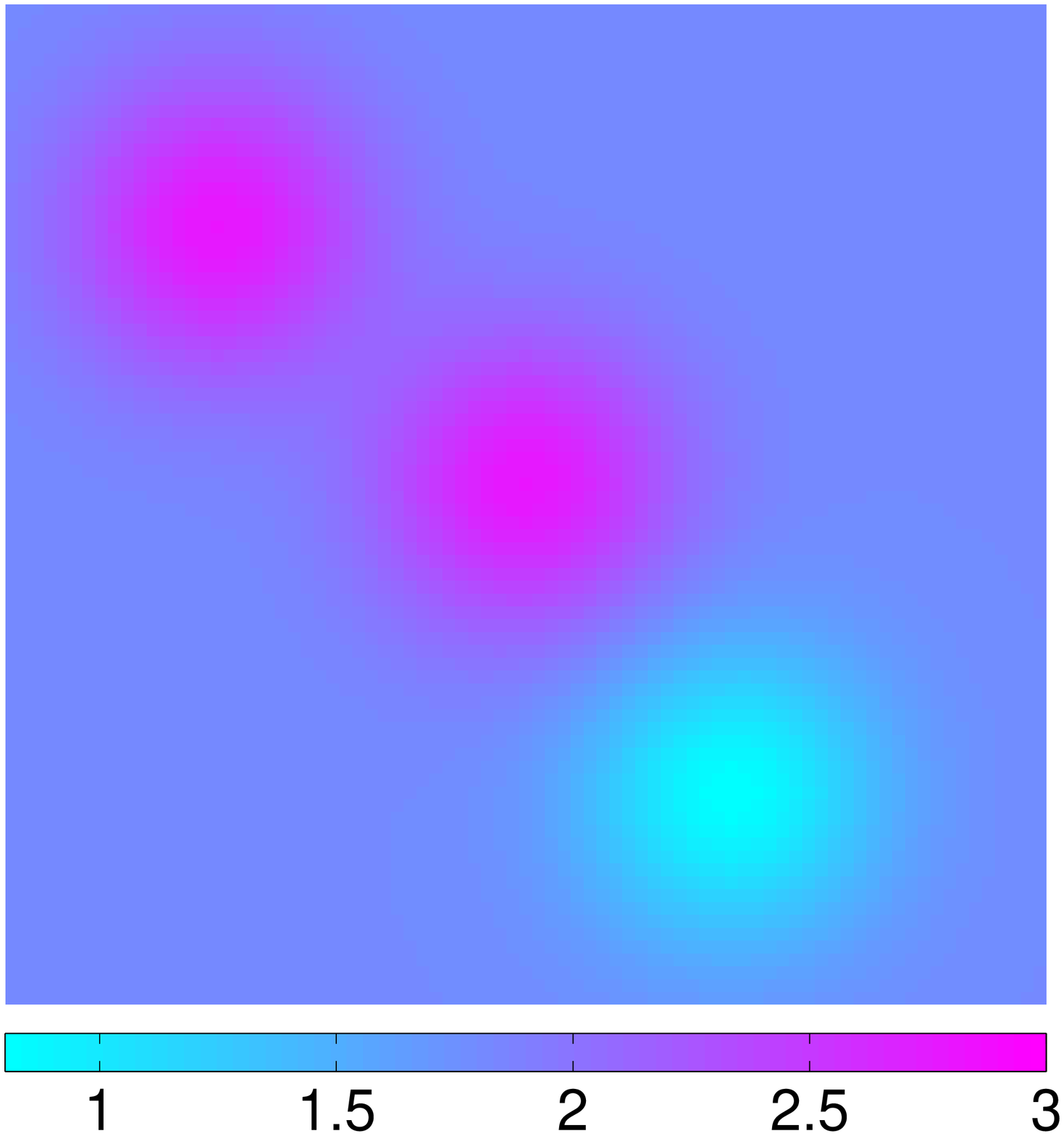}
      \label{ex1tbeta}
      } 
   \subfigure[$\varepsilon_3$ $(\alpha=0\%)$]{  
     \includegraphics[width=37mm,height=35mm]{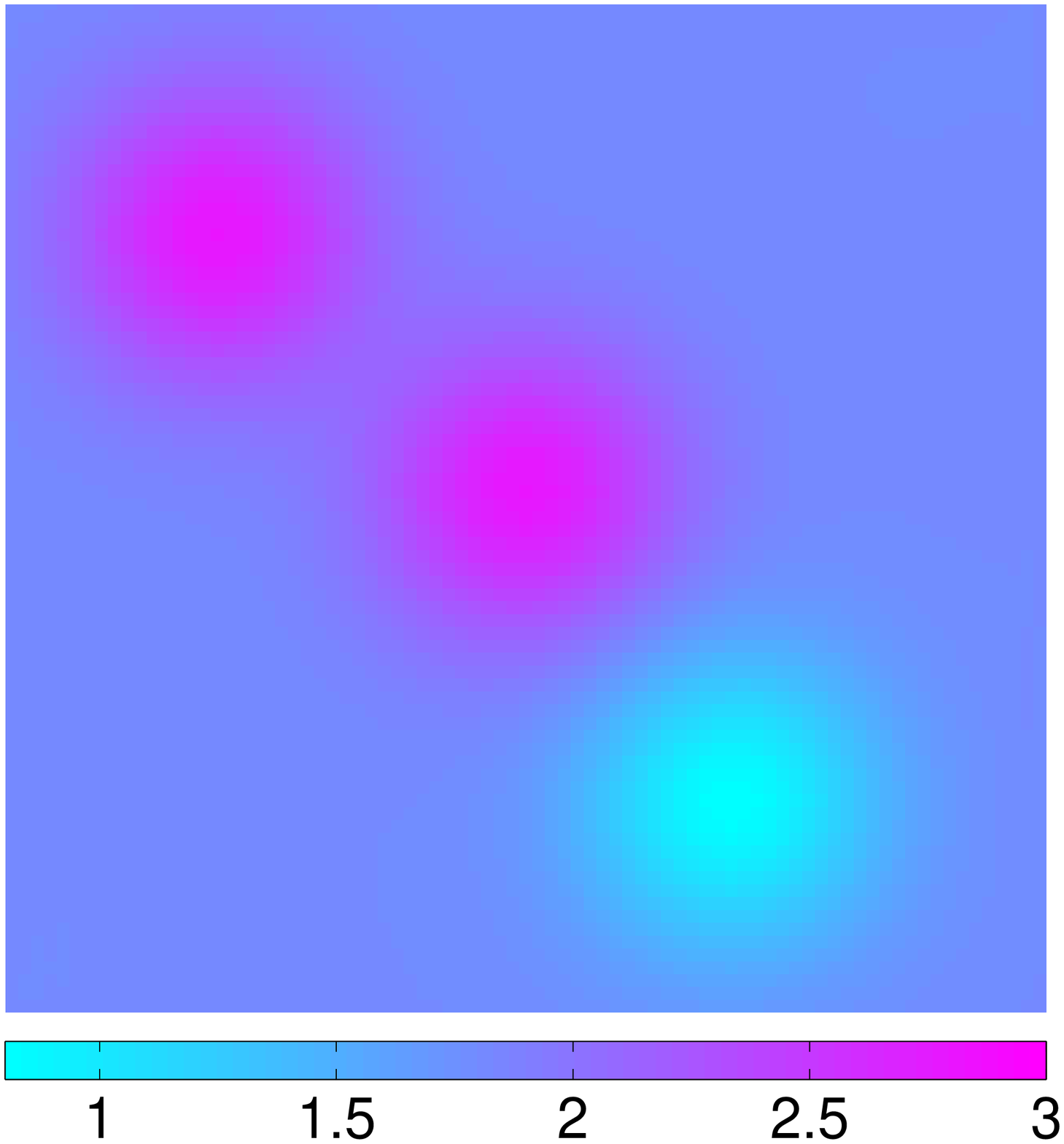}
     \label{ex1cbeta}
     }
   \subfigure[$\varepsilon_3$ ($\alpha=0.1\%$)]{ 
    \includegraphics[width=37mm,height=35mm]{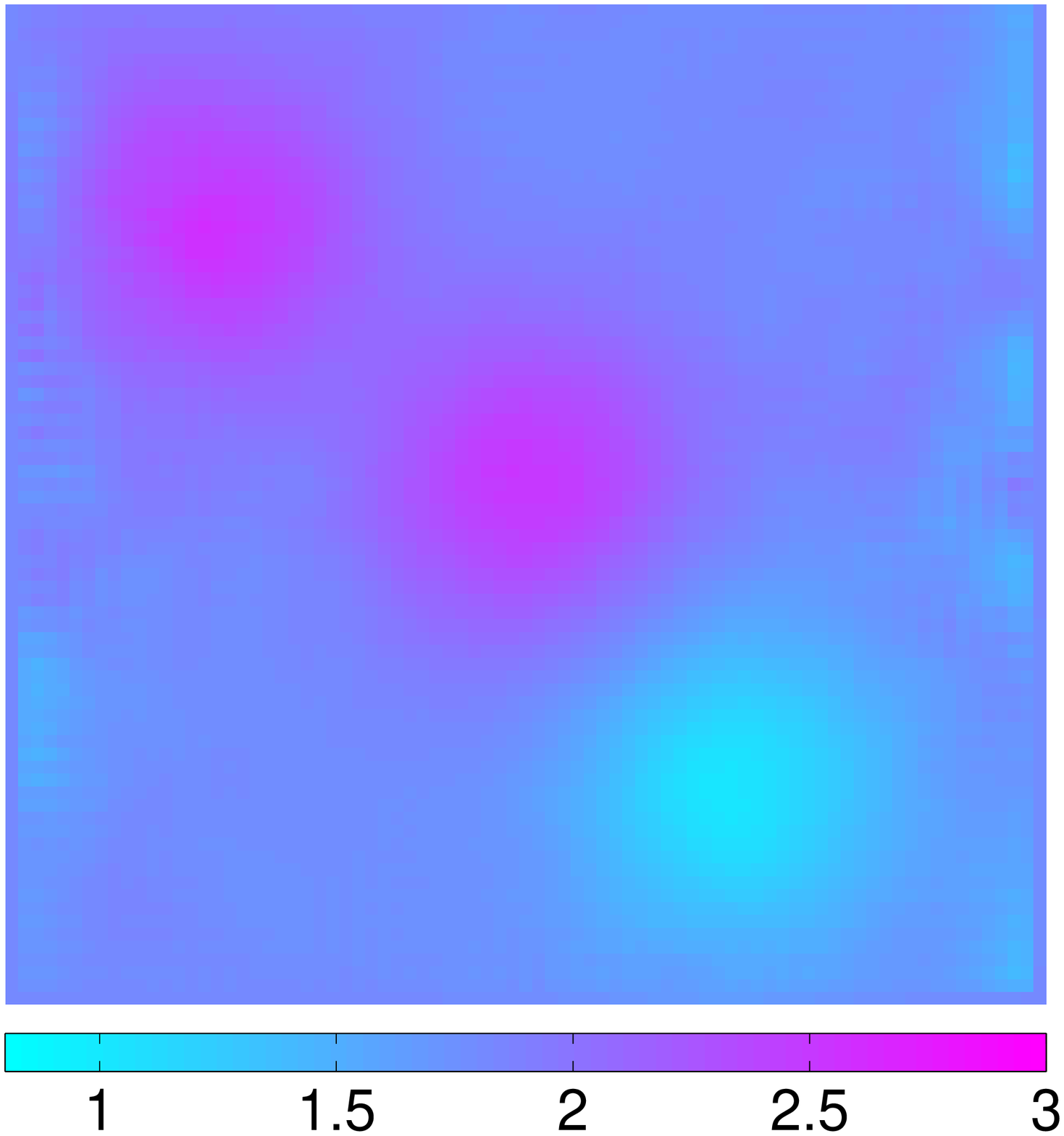}
    \label{ex1nbeta}
    }
   \subfigure[$\varepsilon_3$ at $\{y=0\}$]{ 
     \includegraphics[trim=10mm 5mm 10mm 0mm,clip,width=35mm,height=38mm]{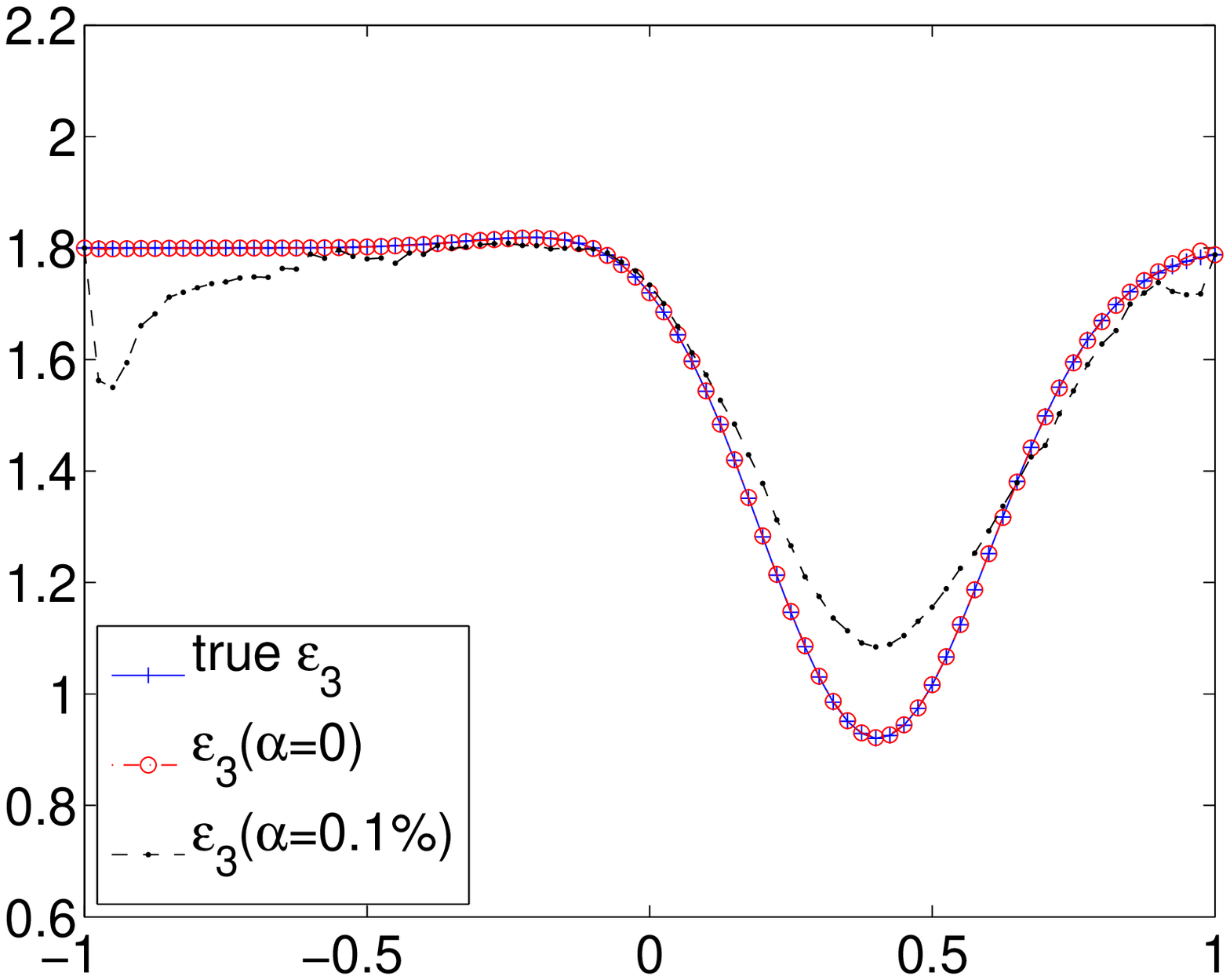} 
     \label{ex1rbeta}
     }
    \caption{$\varepsilon$ in Simulation 1. \subref{ex1txi}\&\subref{ex1ttau}\&\subref{ex1tbeta}: true values of $(\varepsilon_1, \varepsilon_2,\varepsilon_3)$. \subref{ex1cxi}\&\subref{ex1ctau}\&\subref{ex1cbeta}: reconstructions with noiseless data. \subref{ex1nxi}\&\subref{ex1ntau}\&\subref{ex1nbeta}: reconstructions with noisy data($\alpha=0.1\%$). \subref{ex1rxi}\&\subref{ex1rtau}\&\subref{ex1rbeta}: cross sections along $\{y=-0.5\}$.}
\label{E1epsilon}
\end{figure}

\paragraph{Simulation 2.}In this experiment, we attempt to reconstruct piecewise constant coefficients. Reconstructions with both noiseless and noisy data are performed with $l_1$-regularization using the split Bregman iteration method. The noise level is $\alpha=0.1\%$. The results of the numerical experiment are shown in Figure \ref{E2sigma} and \ref{E2epsilon}. From the figures, we observe that the singularities of the coefficients create minor artifacts on the reconstructions and the error in the reconstruction is larger at the discontinuities than in the rest of the domain. The relative $L^2$ errors in the reconstructions are $\mathcal{E}^C_{\sigma_1}=4.0\%$, $\mathcal{E}^N_{\sigma_1}=17.6\%$, $\mathcal{E}^C_{\sigma_2}=12.8\%$, $\mathcal{E}^N_{\sigma_2}=48.1\%$, $\mathcal{E}^C_{\sigma_3}=4.5\%$, $\mathcal{E}^N_{\sigma_3}=16.5\%$; $\mathcal{E}^C_{\varepsilon_1}=0.1\%$, $\mathcal{E}^N_{\varepsilon_1}=16.3\%$, $\mathcal{E}^C_{\varepsilon_2}=0.5\%$, $\mathcal{E}^N_{\varepsilon_2}=35.2\%$, $\mathcal{E}^C_{\varepsilon_3}=0.1\%$, $\mathcal{E}^N_{\varepsilon_3}=16.2\%$.

\begin{figure}[htp]
  \centering
  \subfigure[true $\sigma_1$]{
      \includegraphics[width=37mm,height=35mm]{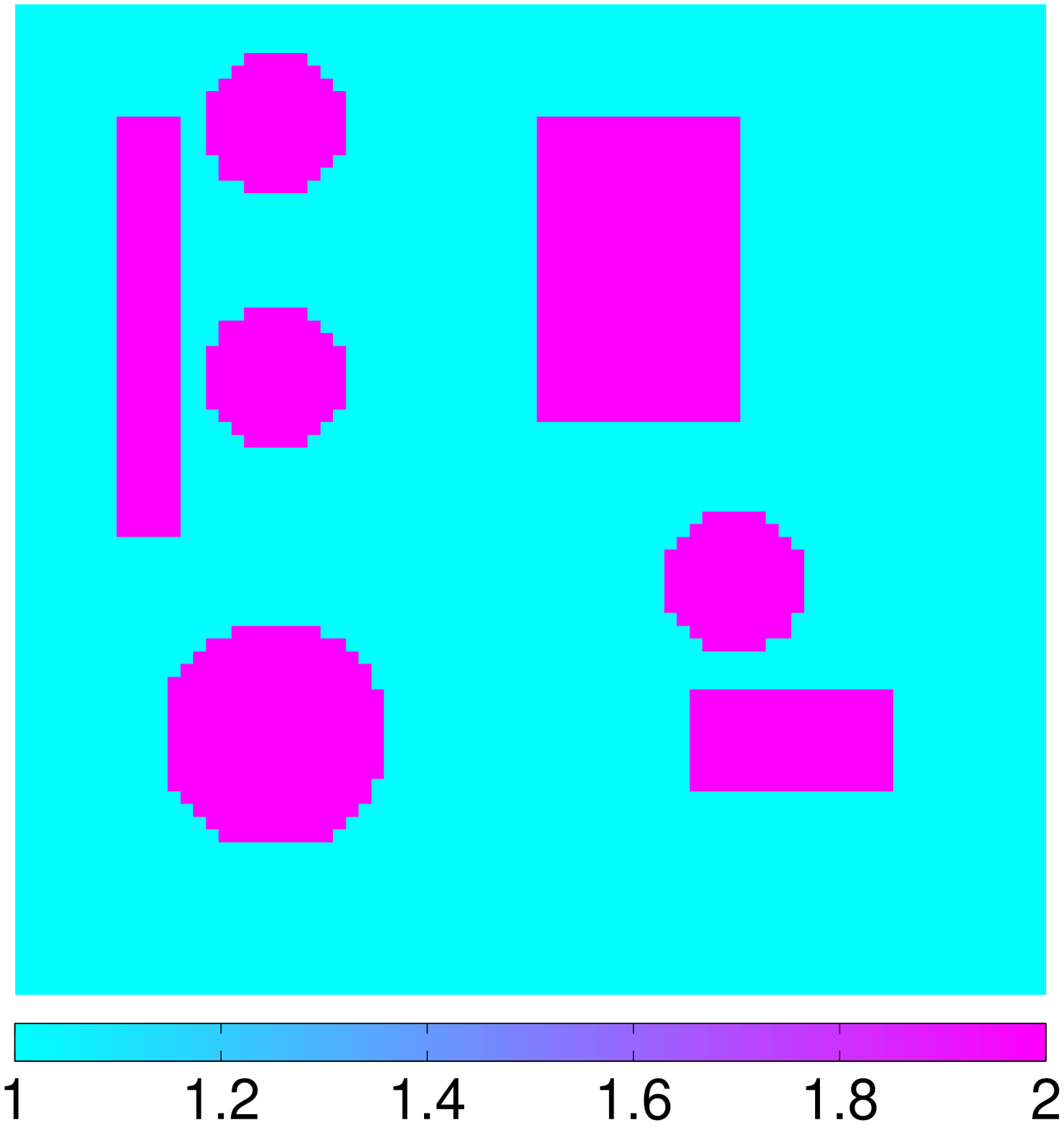}
      \label{ex1txi}
      } 
  \subfigure[$\sigma$ ($\alpha=0\%$)]{    
     \includegraphics[width=37mm,height=35mm]{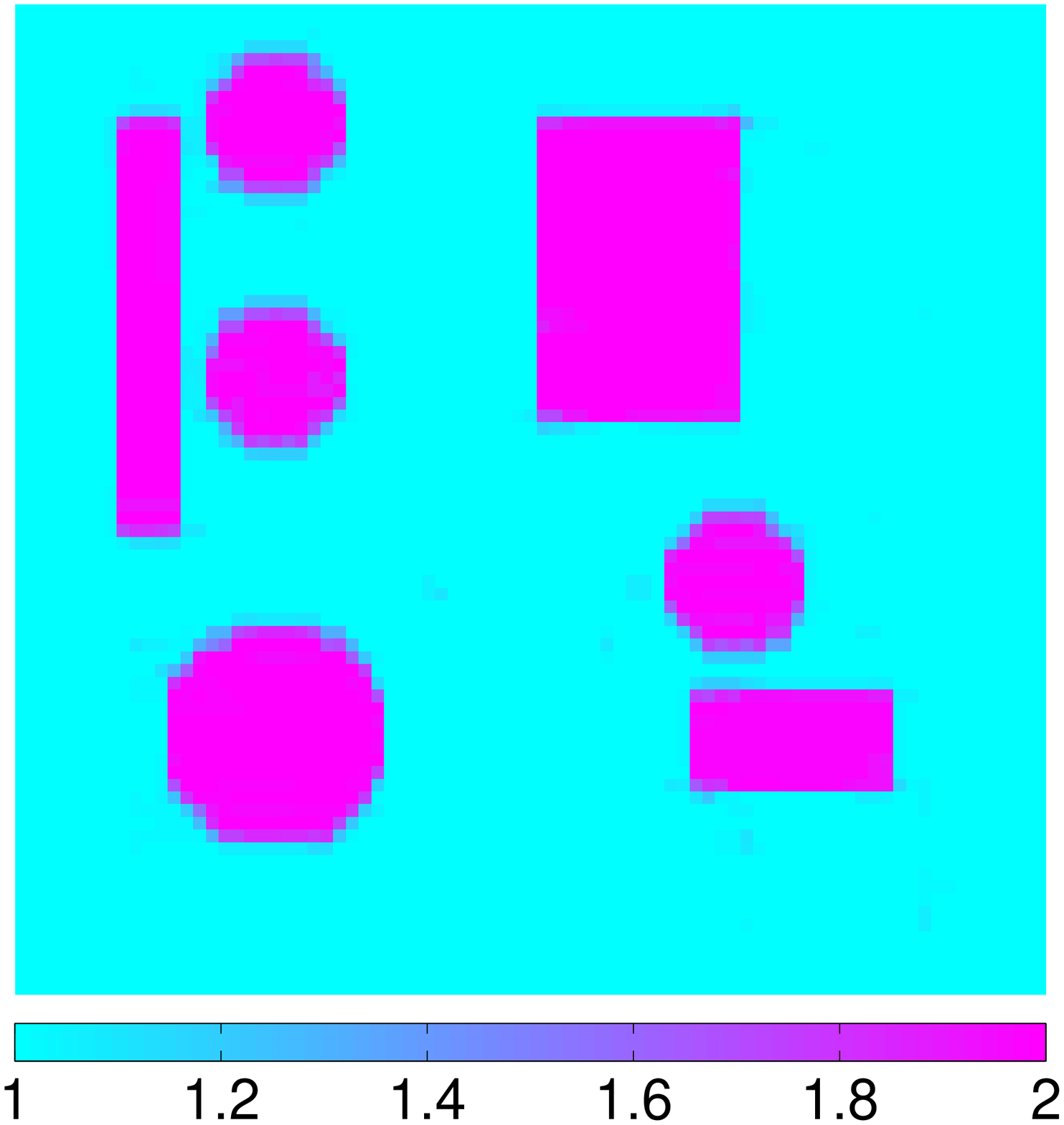}
     \label{ex1cxi}
     }
   \subfigure[$\sigma$ ($\alpha=0.1\%$)]{
    \includegraphics[width=37mm,height=35mm]{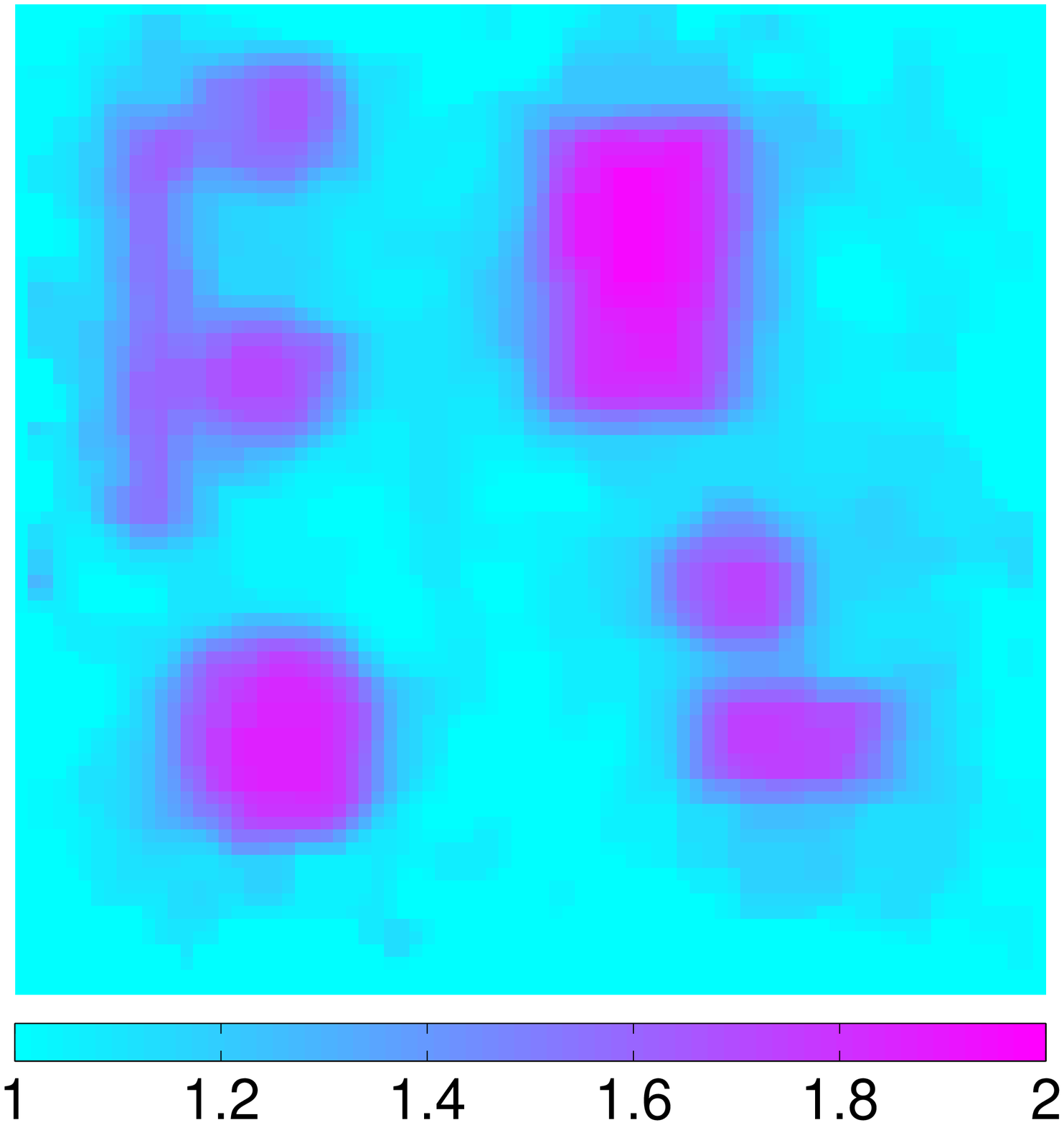}
    \label{ex1nxi}
    }
    \subfigure[$\sigma_1$ at \{$y=-0.5$\}]{
     \includegraphics[trim=10mm 5mm 10mm 0mm,clip,width=35mm,height=38mm]{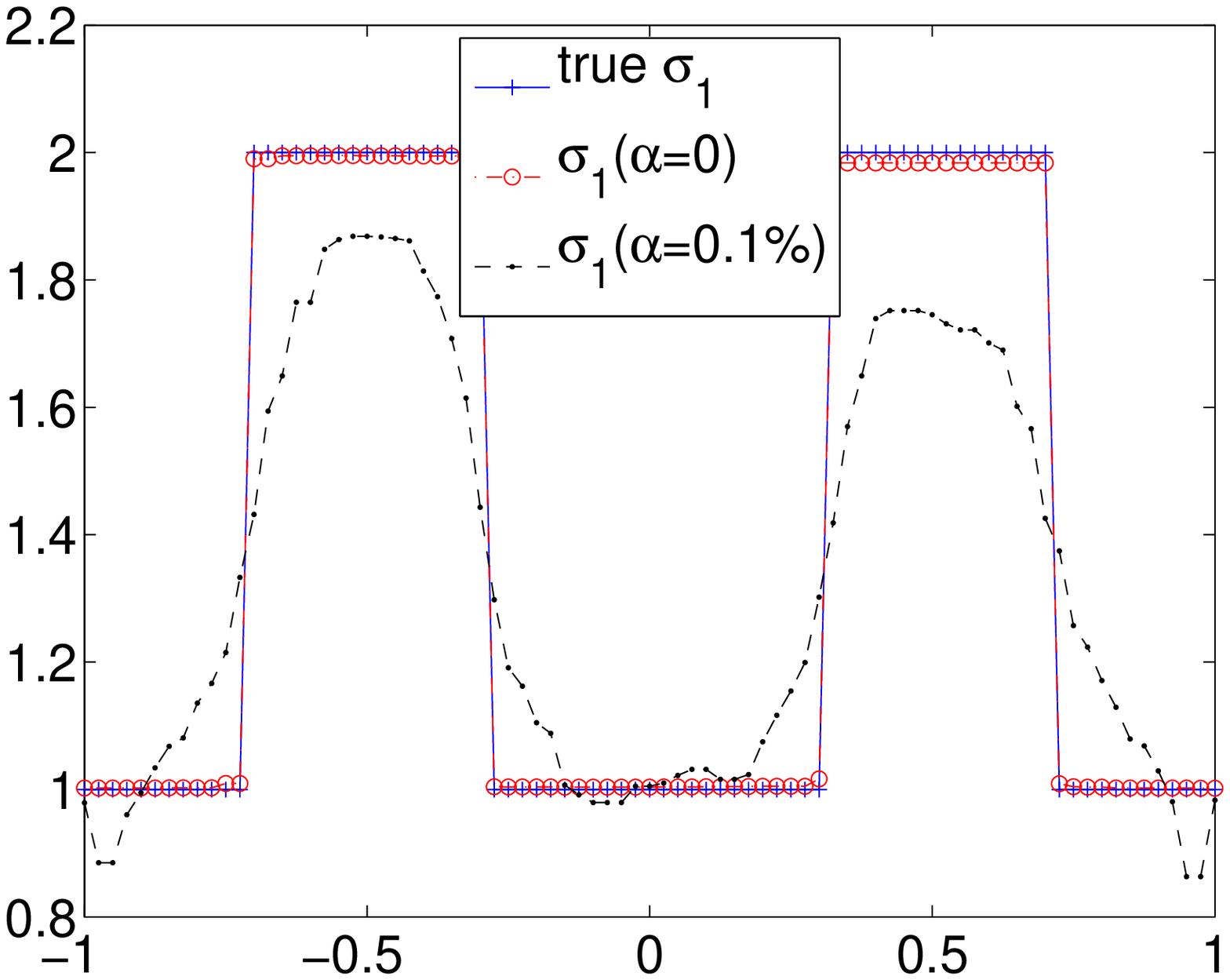} 
     \label{ex1rxi}
     }

     \subfigure[true $\sigma_2$]{
      \includegraphics[width=37mm,height=35mm]{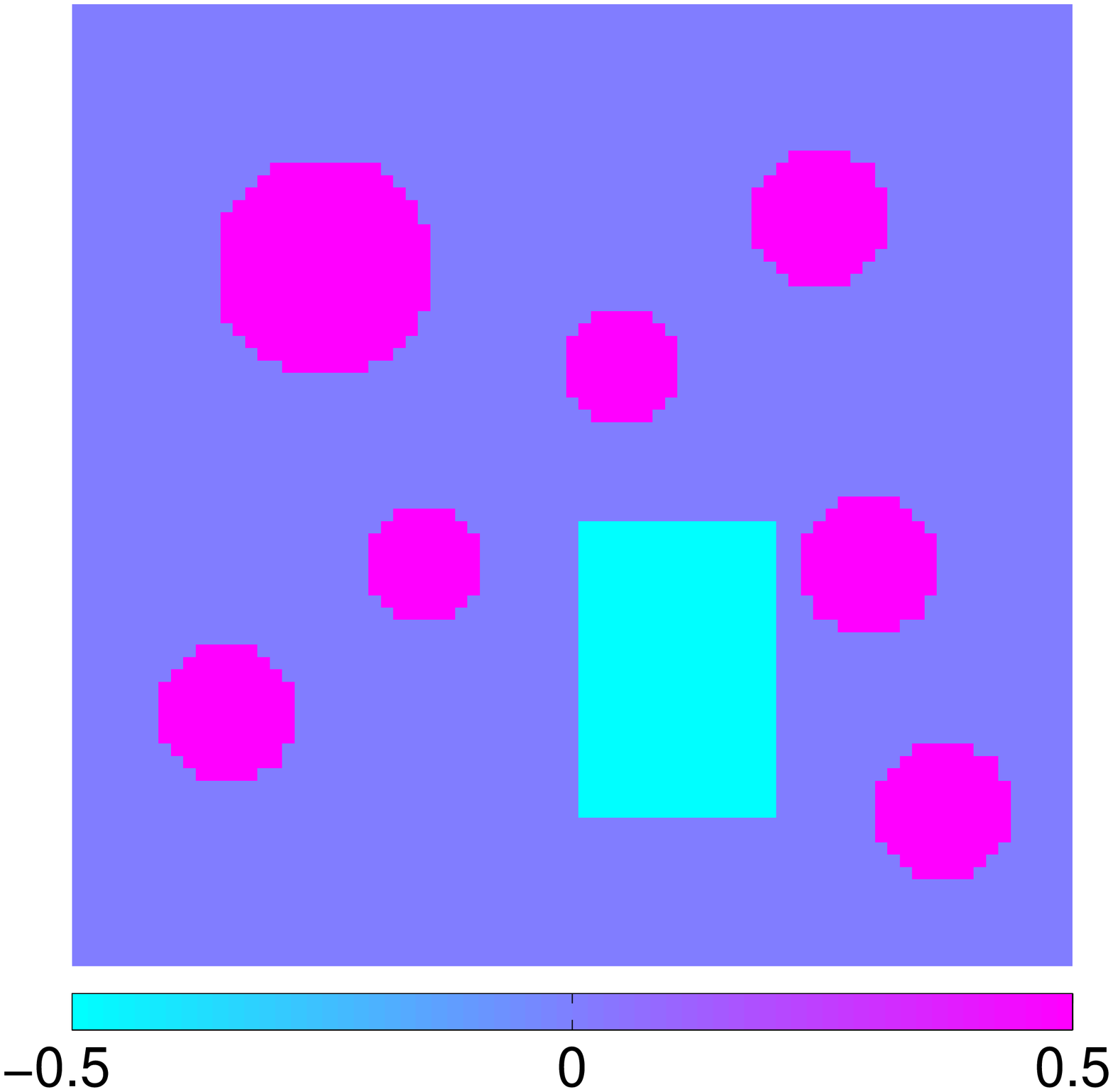}
      \label{ex1ttau}
      } 
    \subfigure[$\sigma_2$ ($\alpha=0\%$)]{ 
     \includegraphics[width=37mm,height=35mm]{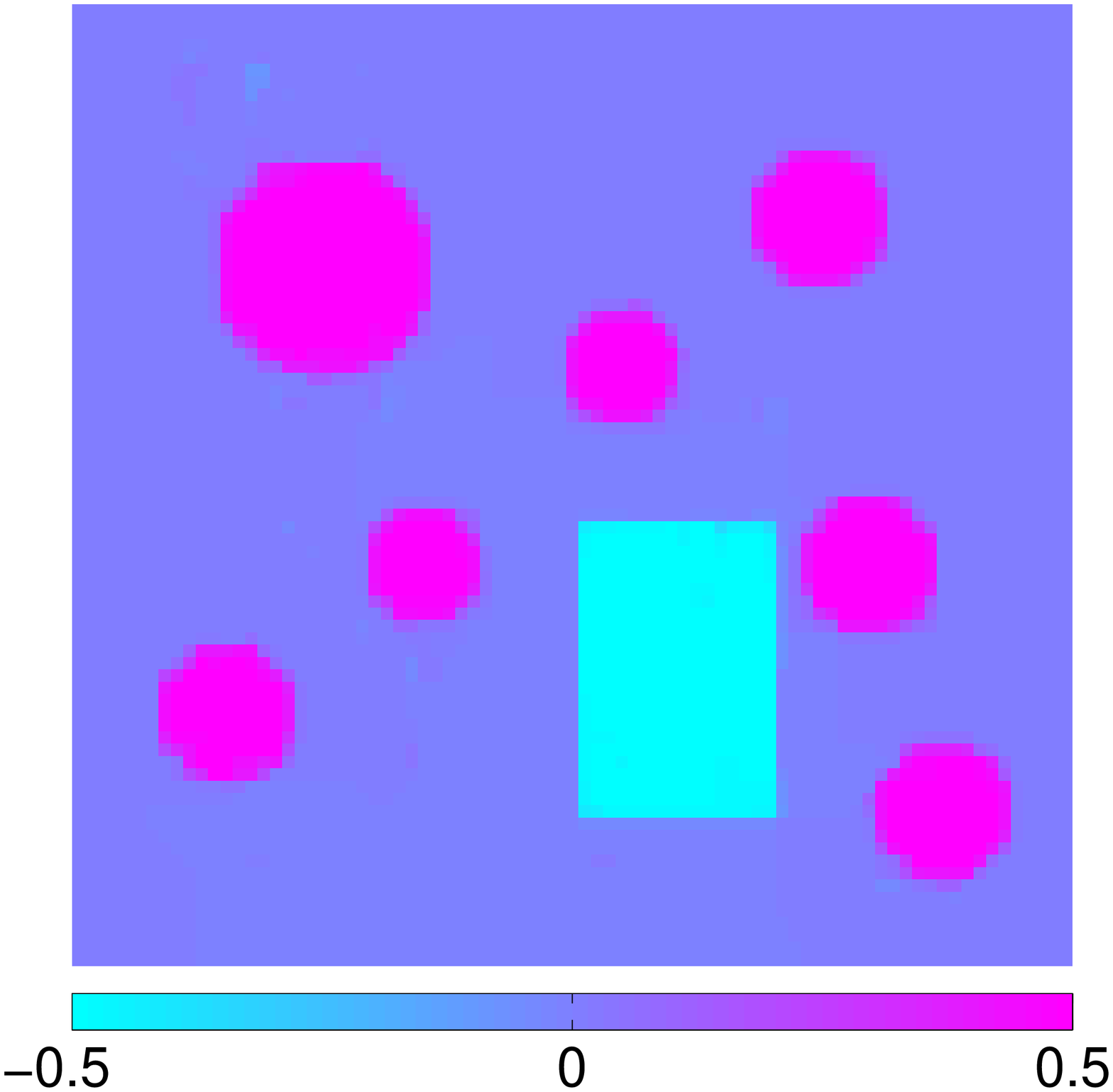}
     \label{ex1ctau}
     }
   \subfigure[$\sigma_2$ ($\alpha=0.1\%$)]{ 
    \includegraphics[width=37mm,height=35mm]{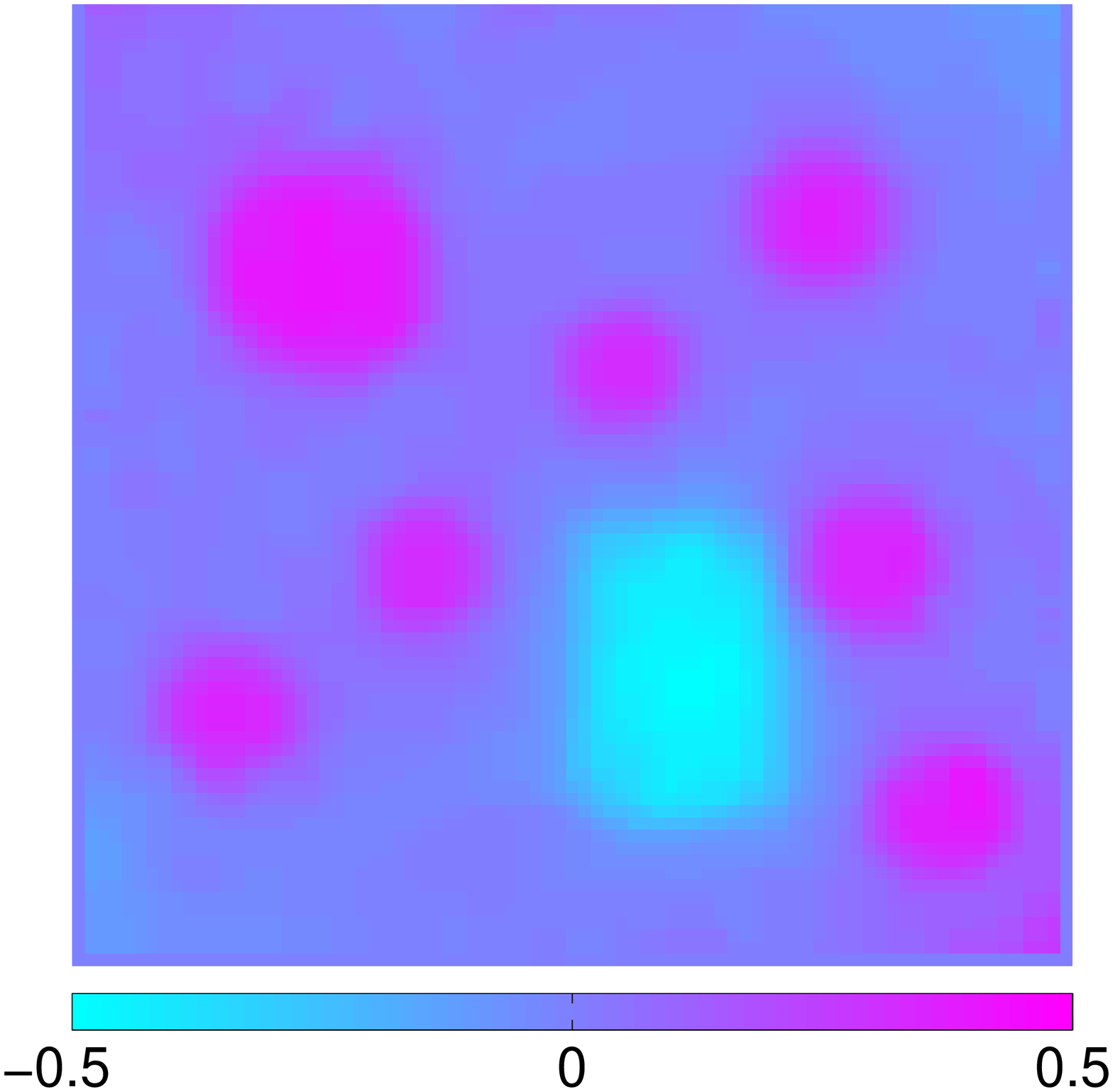}
    \label{ex1ntau}
    }
     \subfigure[$\sigma_2$ at \{$y=-0.5$\}]{
     \includegraphics[trim=10mm 5mm 10mm 0mm,clip,width=35mm,height=38mm]{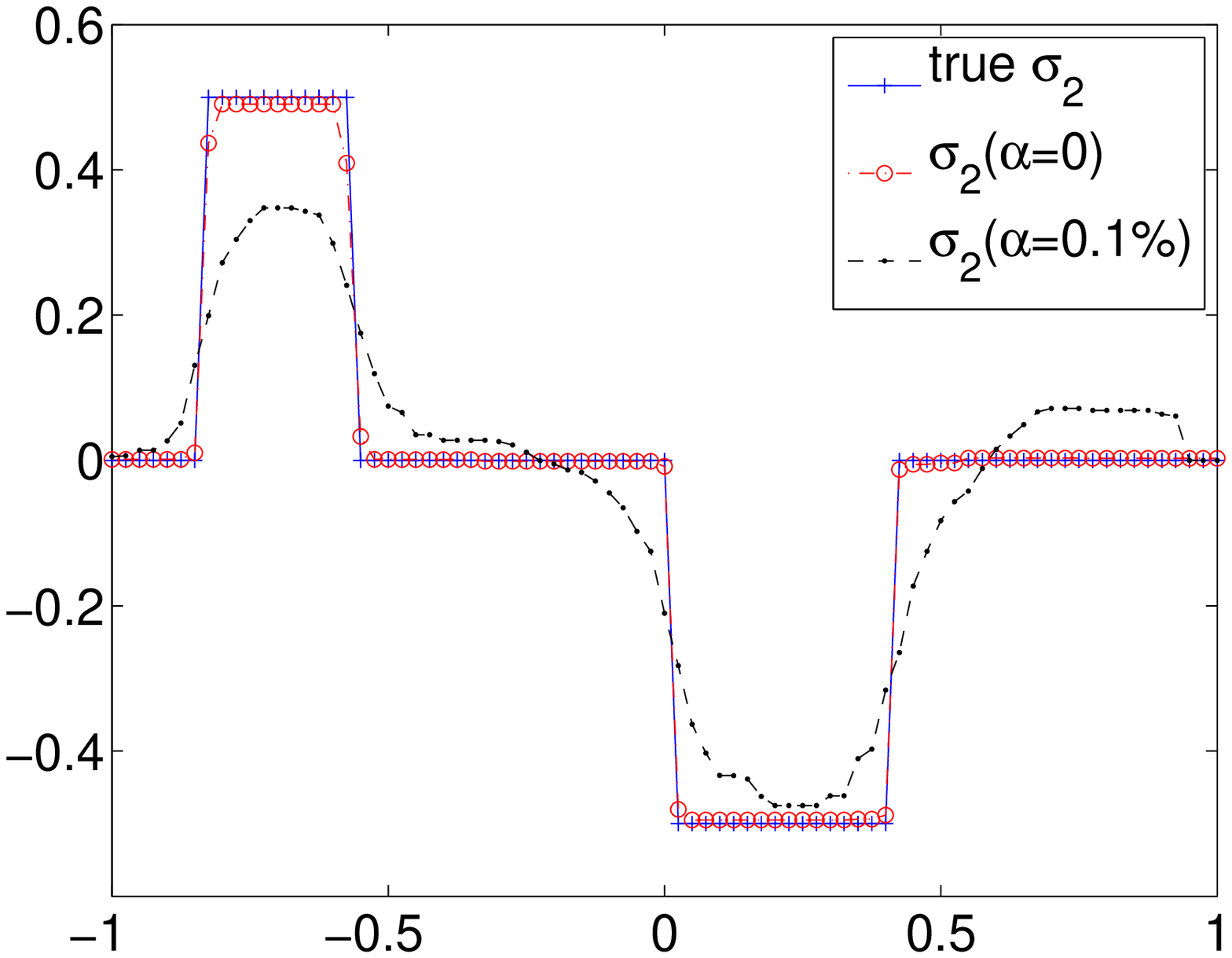} 
     \label{ex1rtau}
     }
     
     \subfigure[true $\sigma_3$]{
      \includegraphics[width=37mm,height=35mm]{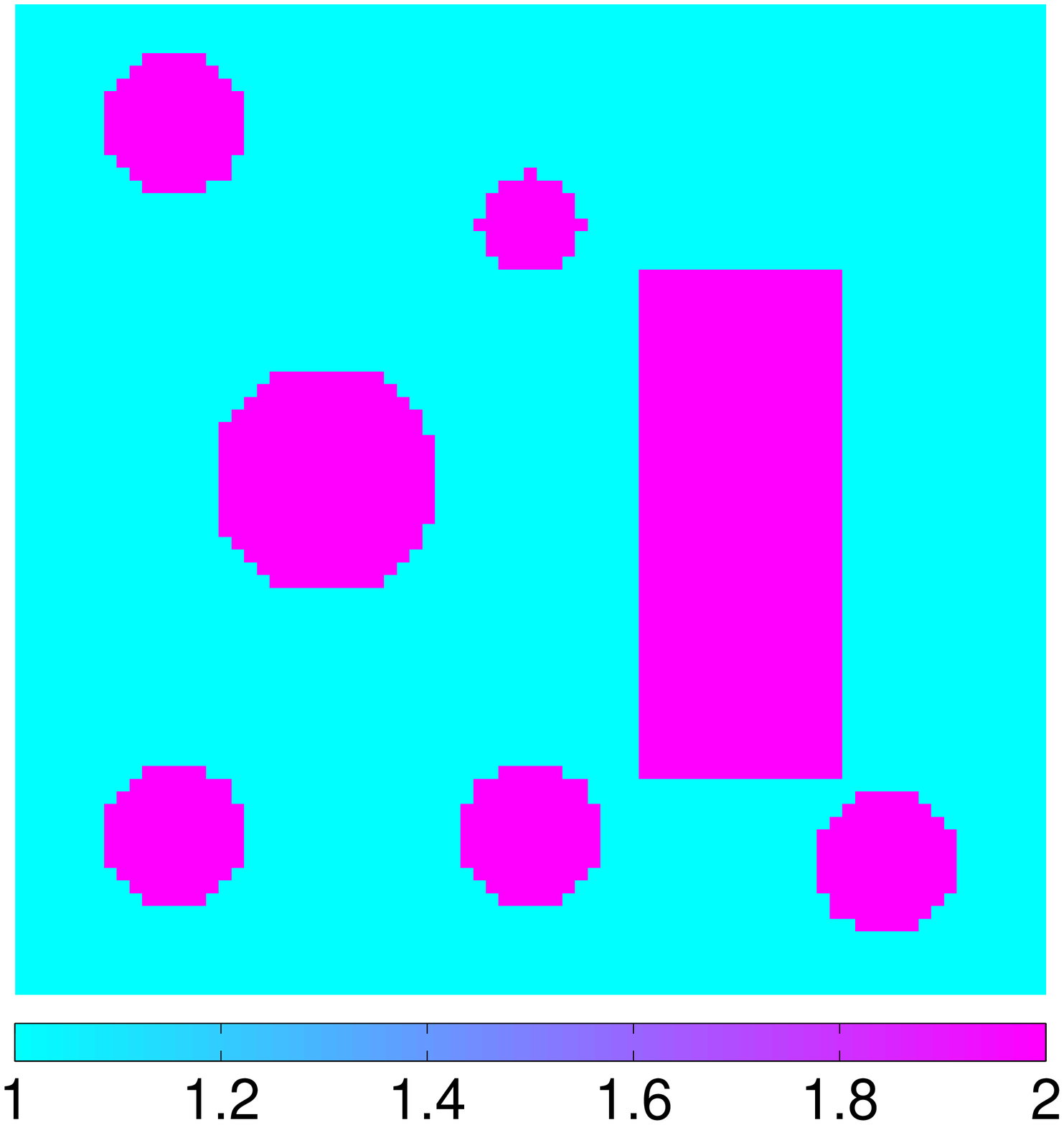}
      \label{ex1tbeta}
      } 
   \subfigure[$\sigma_3$ $(\alpha=0\%)$]{  
     \includegraphics[width=37mm,height=35mm]{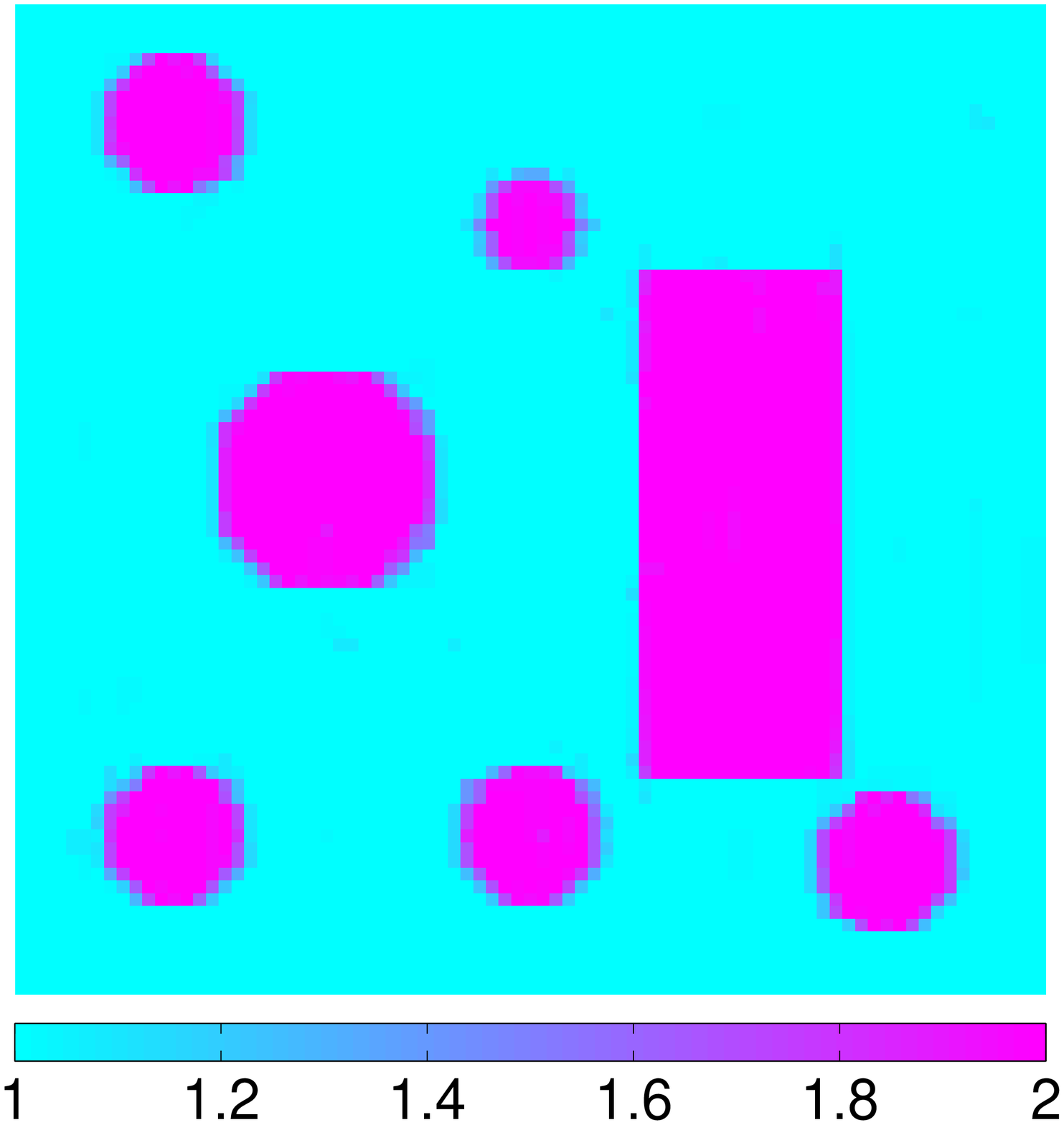}
     \label{ex1cbeta}
     }
   \subfigure[$\sigma_3$ ($\alpha=0.1\%$)]{ 
    \includegraphics[width=37mm,height=35mm]{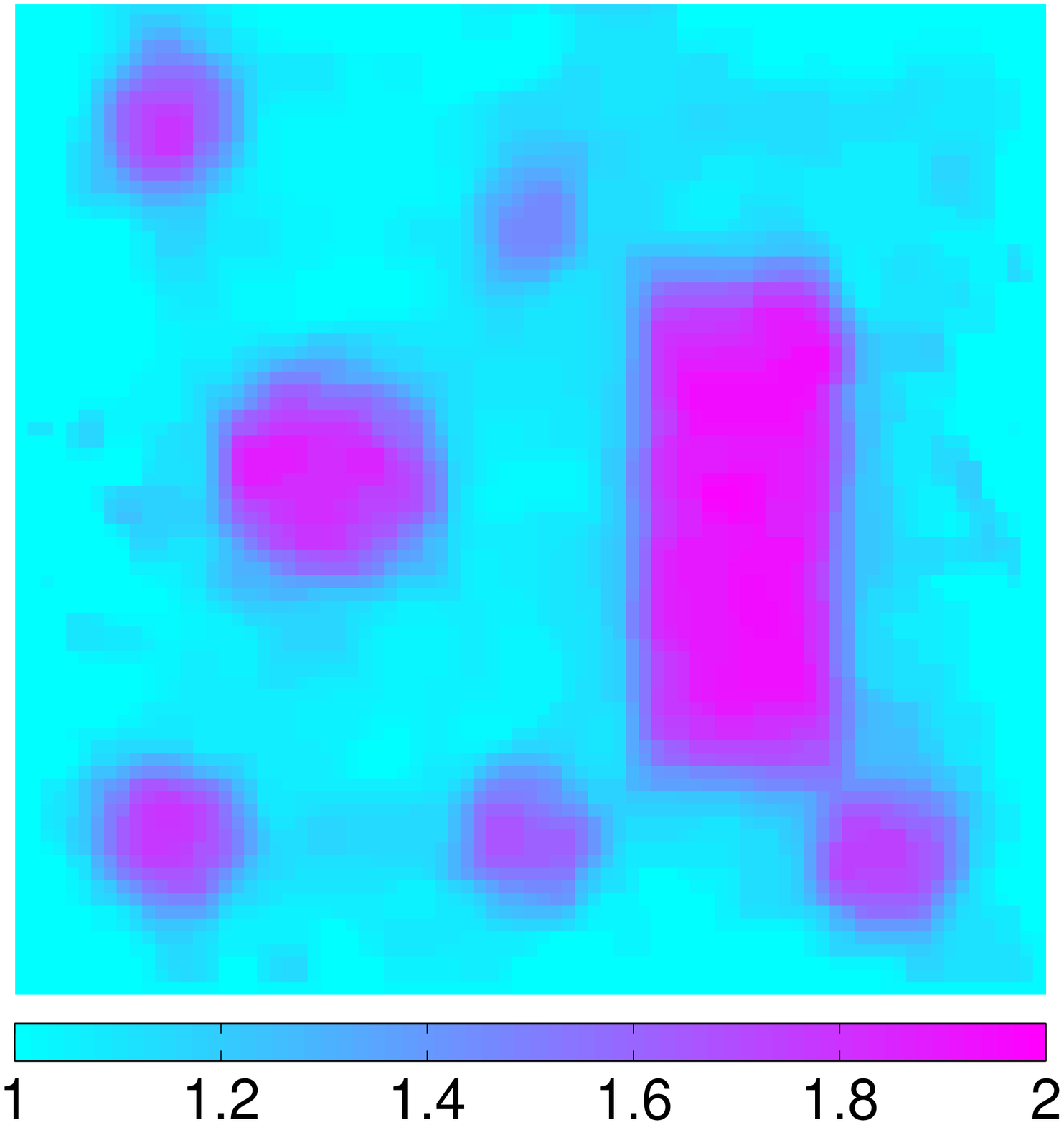}
    \label{ex1nbeta}
    }
   \subfigure[$\sigma_3$ at $\{y=0\}$]{ 
     \includegraphics[trim=10mm 5mm 10mm 0mm,clip,width=35mm,height=38mm]{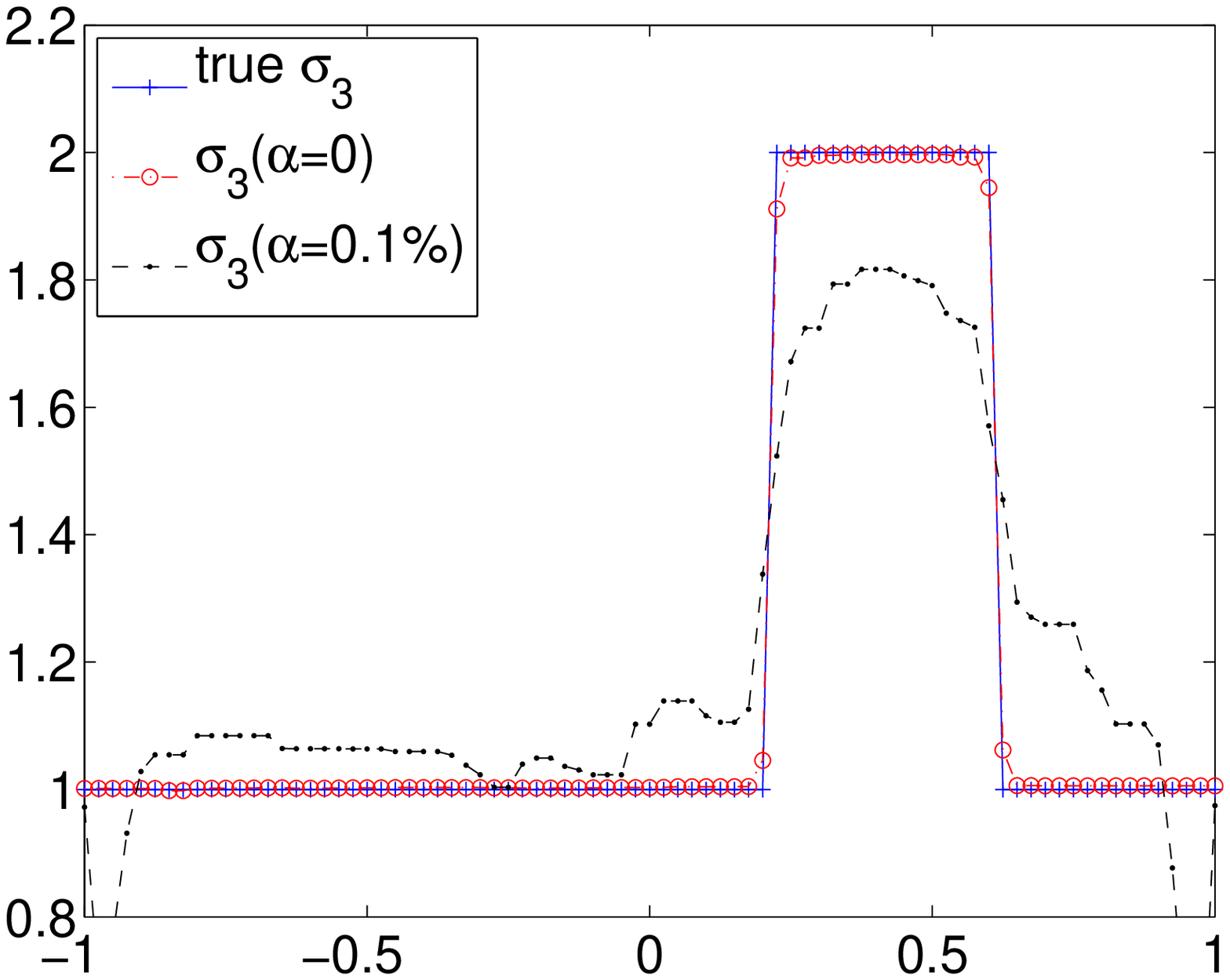} 
     \label{ex1rbeta}
     }
    \caption{$\sigma$ in Simulation 2. \subref{ex1txi}\&\subref{ex1ttau}\&\subref{ex1tbeta}: true values of $(\sigma_1, \sigma_2,\sigma_3)$. \subref{ex1cxi}\&\subref{ex1ctau}\&\subref{ex1cbeta}: reconstructions with noiseless data. \subref{ex1nxi}\&\subref{ex1ntau}\&\subref{ex1nbeta}: reconstructions with noisy data($\alpha=0.1\%$). \subref{ex1rxi}\&\subref{ex1rtau}\&\subref{ex1rbeta}: cross sections along $\{y=-0.5\}$.}
\label{E2sigma}
\end{figure}

\begin{figure}[htp]
  \centering
  \subfigure[true $\varepsilon_1$]{
      \includegraphics[width=37mm,height=35mm]{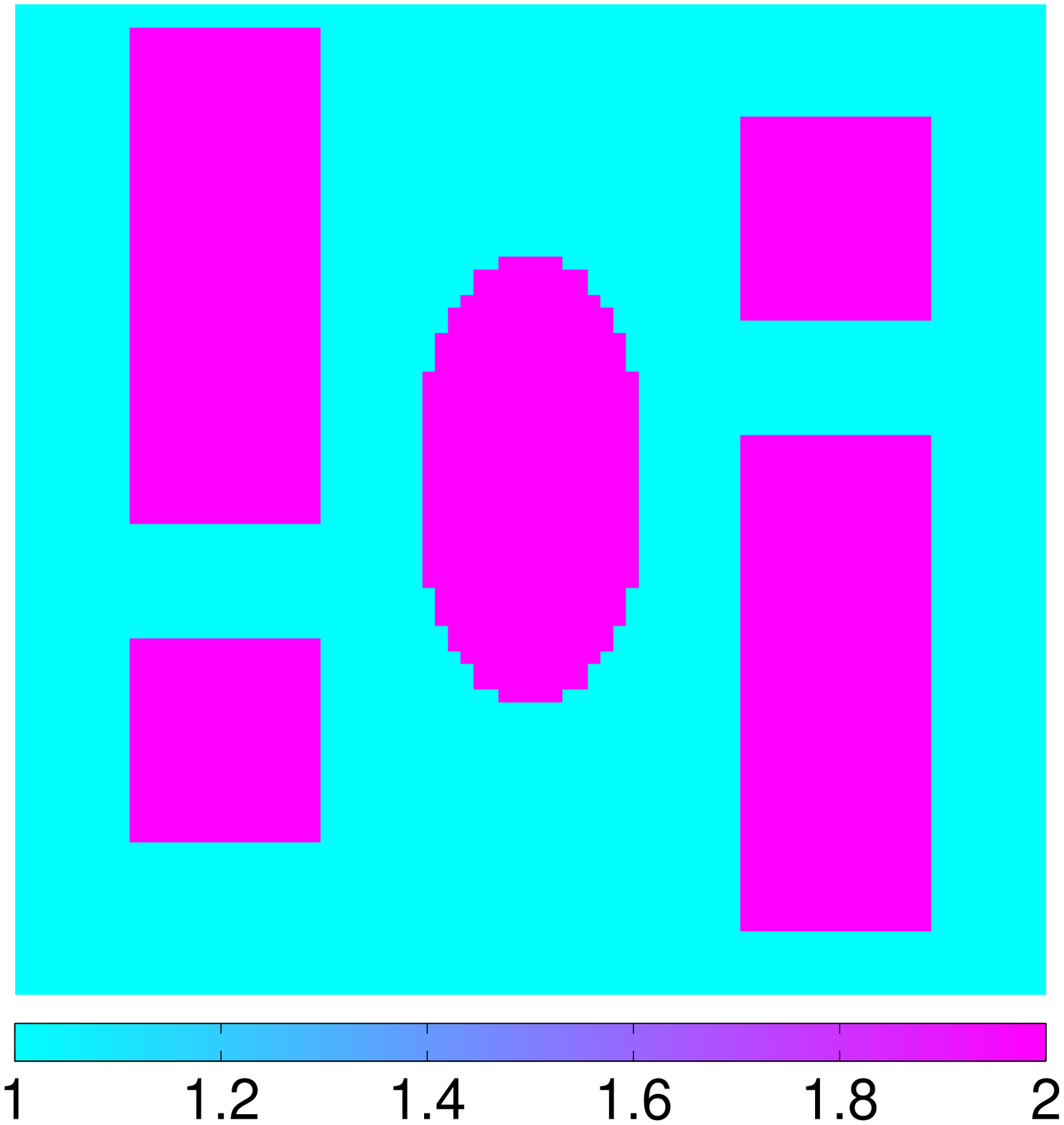}
      \label{ex1txi}
      } 
  \subfigure[$\varepsilon_1$ ($\alpha=0\%$)]{    
     \includegraphics[width=37mm,height=35mm]{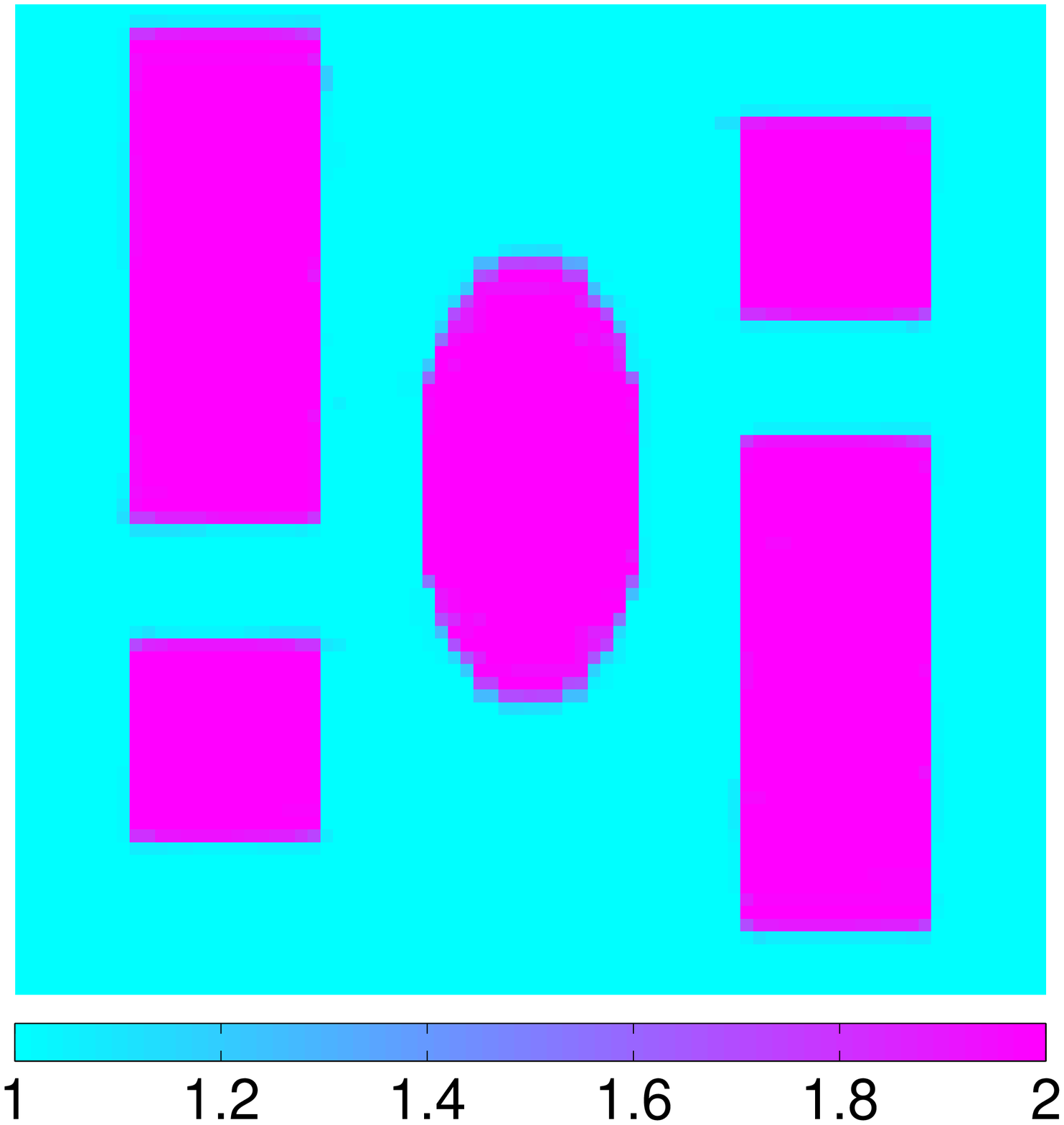}
     \label{ex1cxi}
     }
   \subfigure[$\varepsilon_1$ ($\alpha=0.1\%$)]{
    \includegraphics[width=37mm,height=35mm]{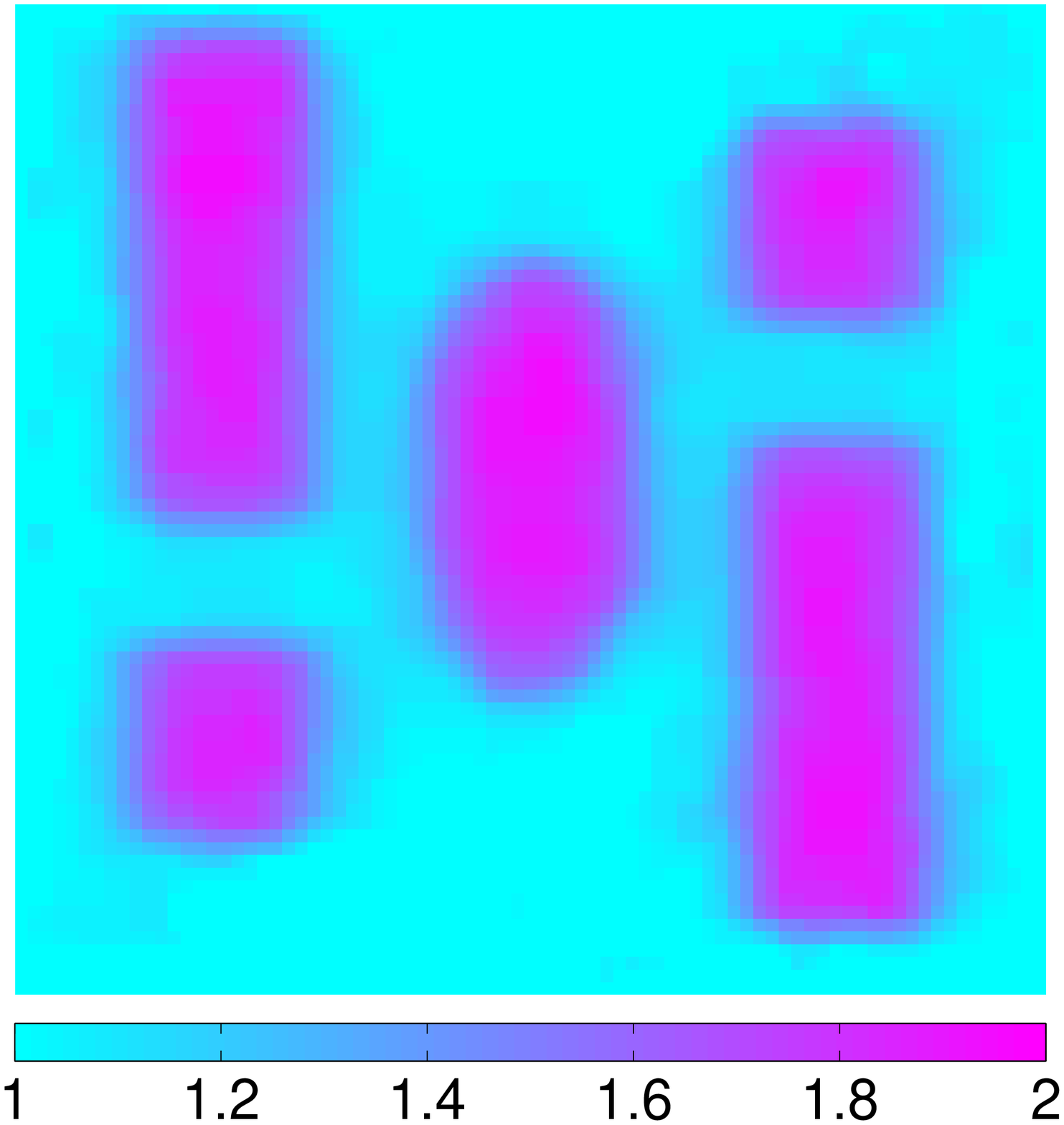}
    \label{ex1nxi}
    }
    \subfigure[$\varepsilon_1$ at \{$y=-0.5$\}]{
     \includegraphics[trim=10mm 5mm 10mm 0mm,clip,width=35mm,height=38mm]{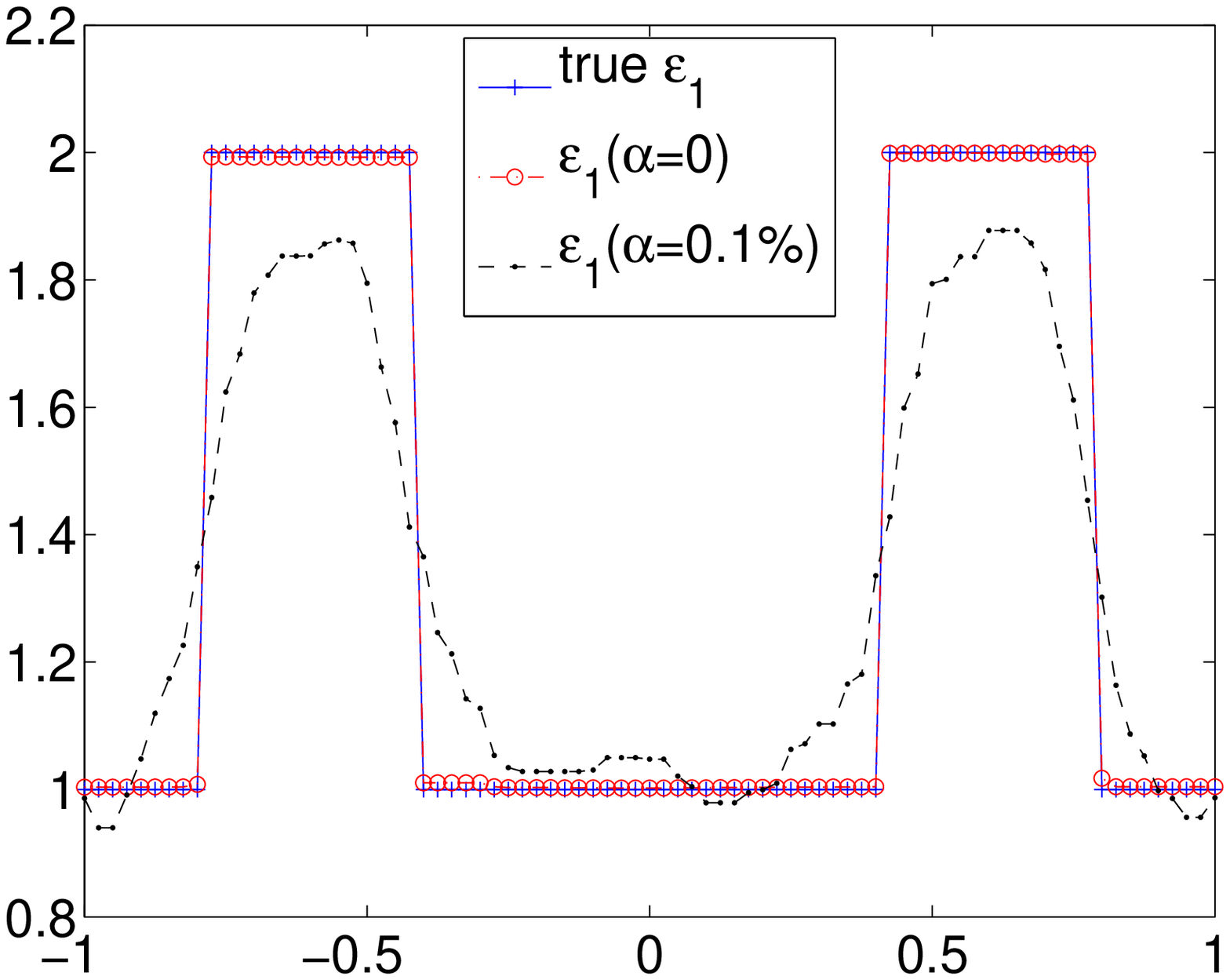} 
     \label{ex1rxi}
     }

     \subfigure[true $\varepsilon_2$]{
      \includegraphics[width=37mm,height=35mm]{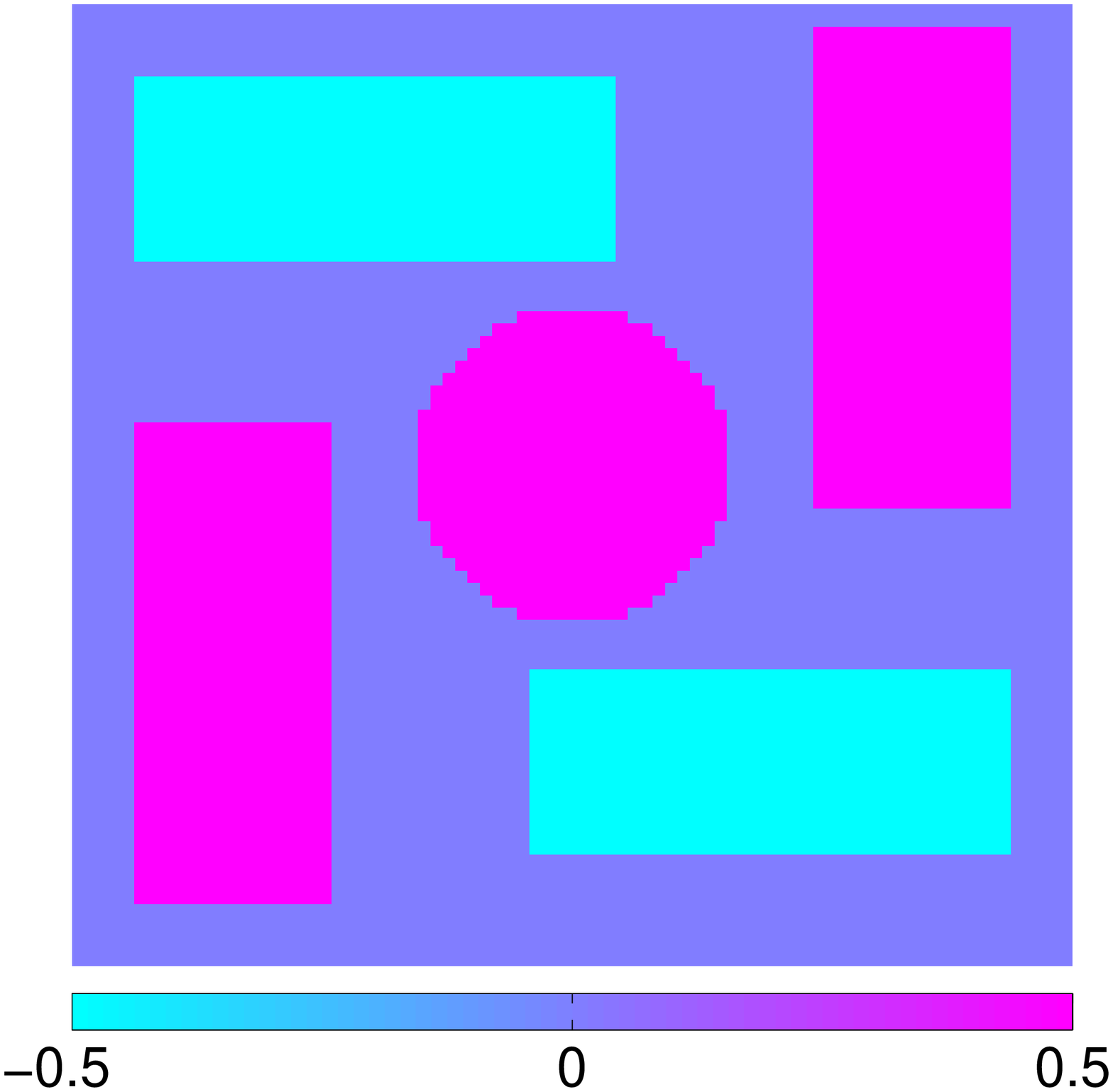}
      \label{ex1ttau}
      } 
    \subfigure[$\varepsilon_2$ ($\alpha=0\%$)]{ 
     \includegraphics[width=37mm,height=35mm]{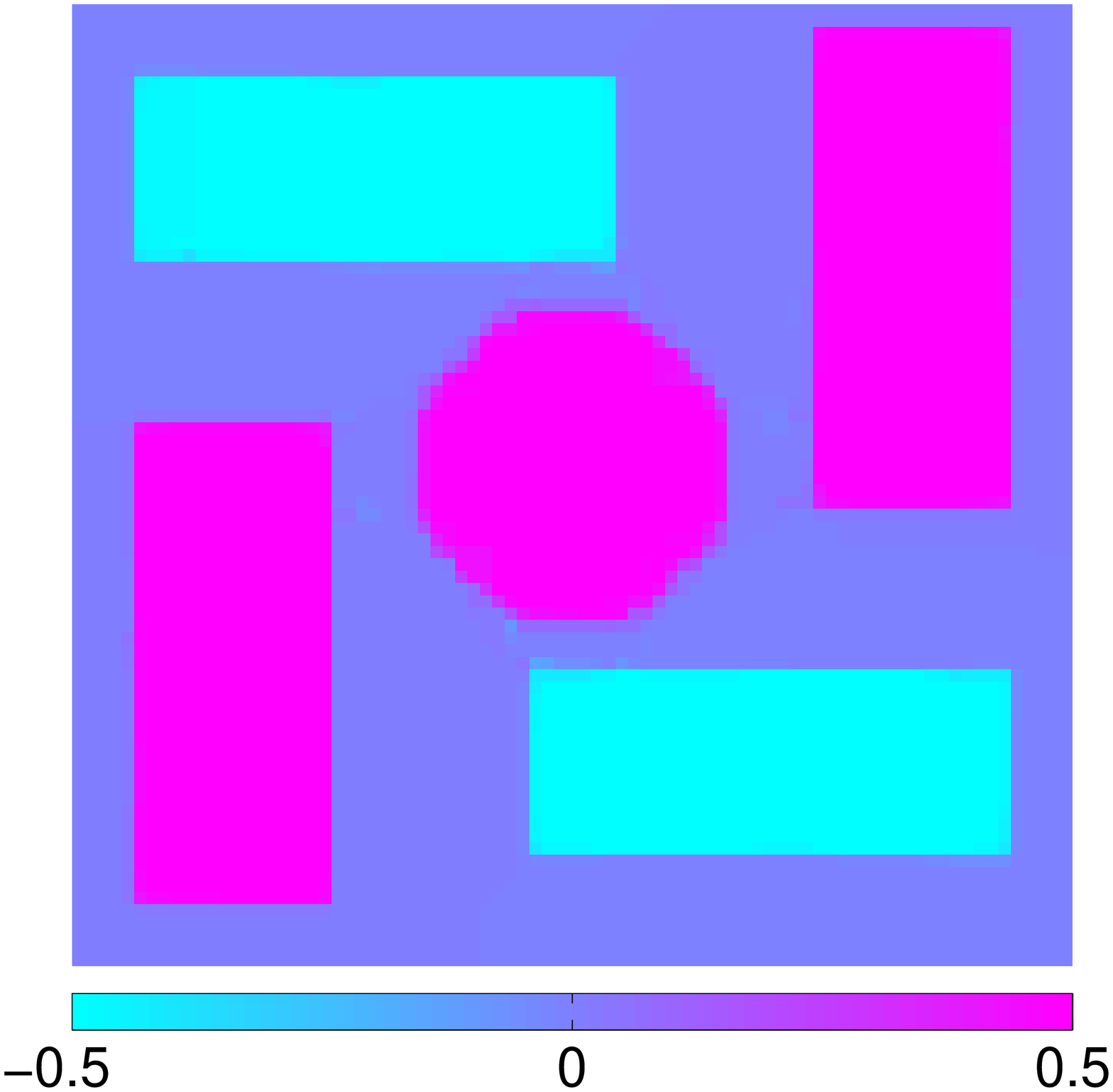}
     \label{ex1ctau}
     }
   \subfigure[$\varepsilon_2$ ($\alpha=0.1\%$)]{ 
    \includegraphics[width=37mm,height=35mm]{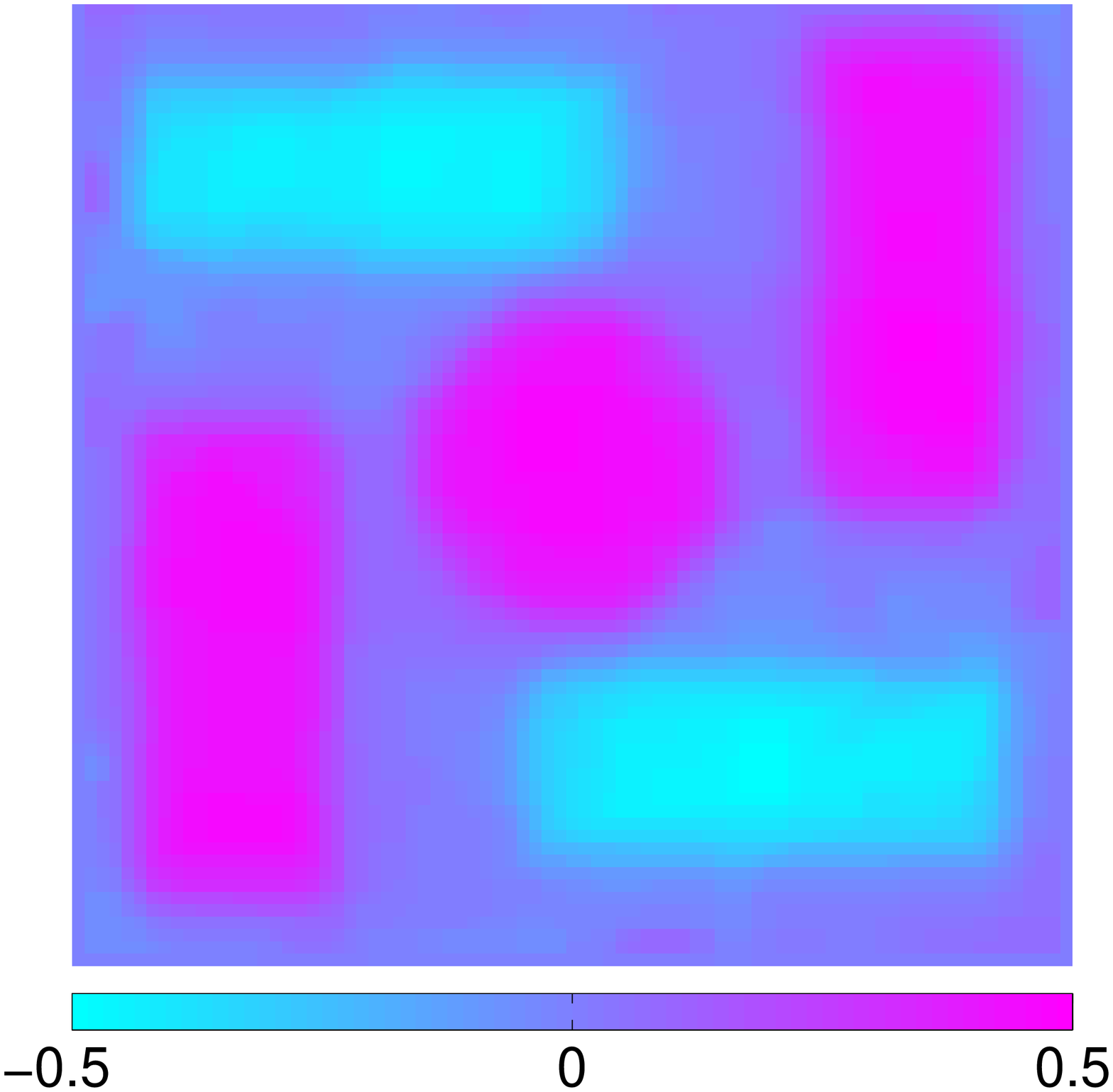}
    \label{ex1ntau}
    }
     \subfigure[$\varepsilon_2$ at \{$y=-0.5$\}]{
     \includegraphics[trim=10mm 5mm 10mm 0mm,clip,width=35mm,height=38mm]{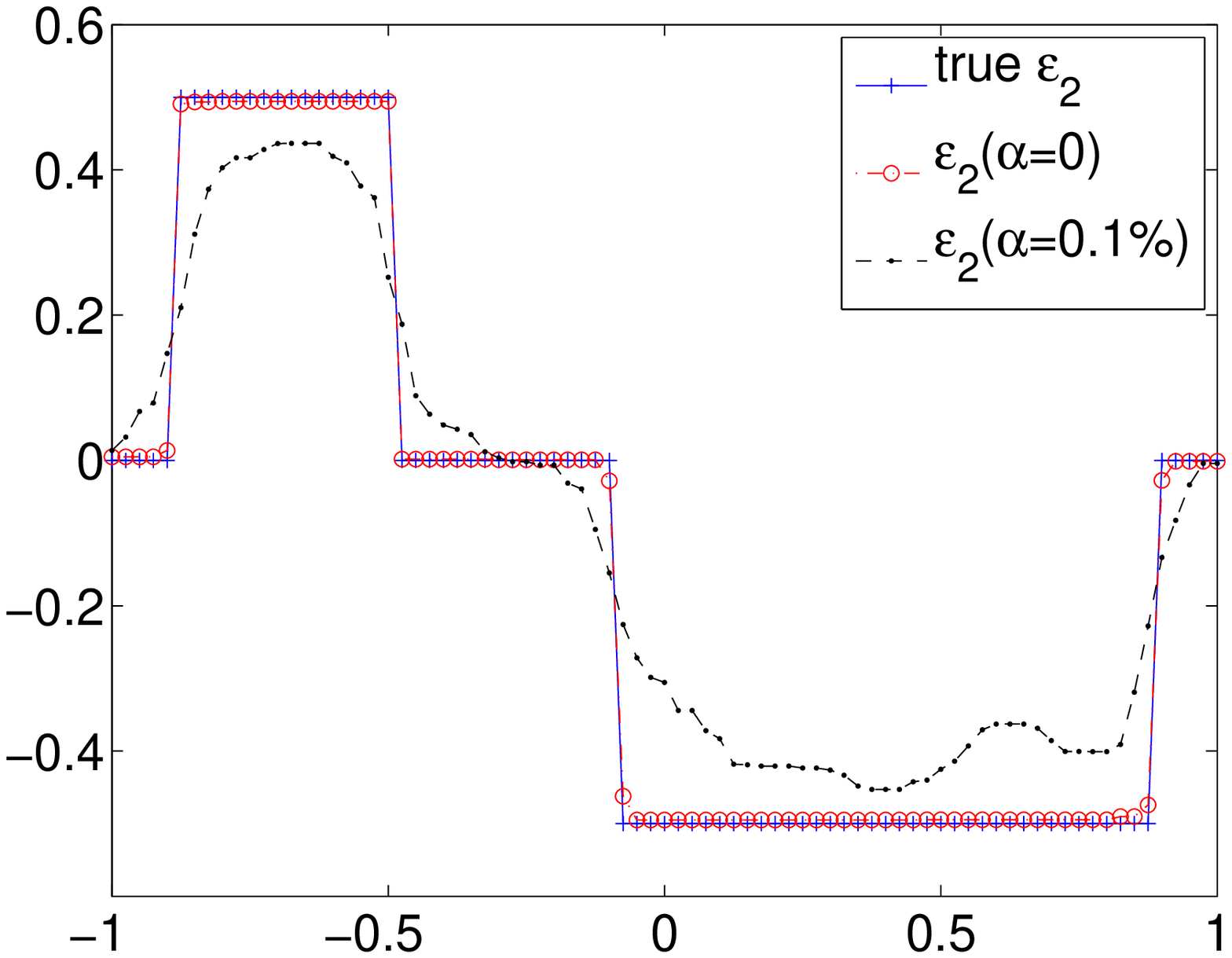} 
     \label{ex1rtau}
     }
     
     \subfigure[true $\varepsilon_3$]{
      \includegraphics[width=37mm,height=35mm]{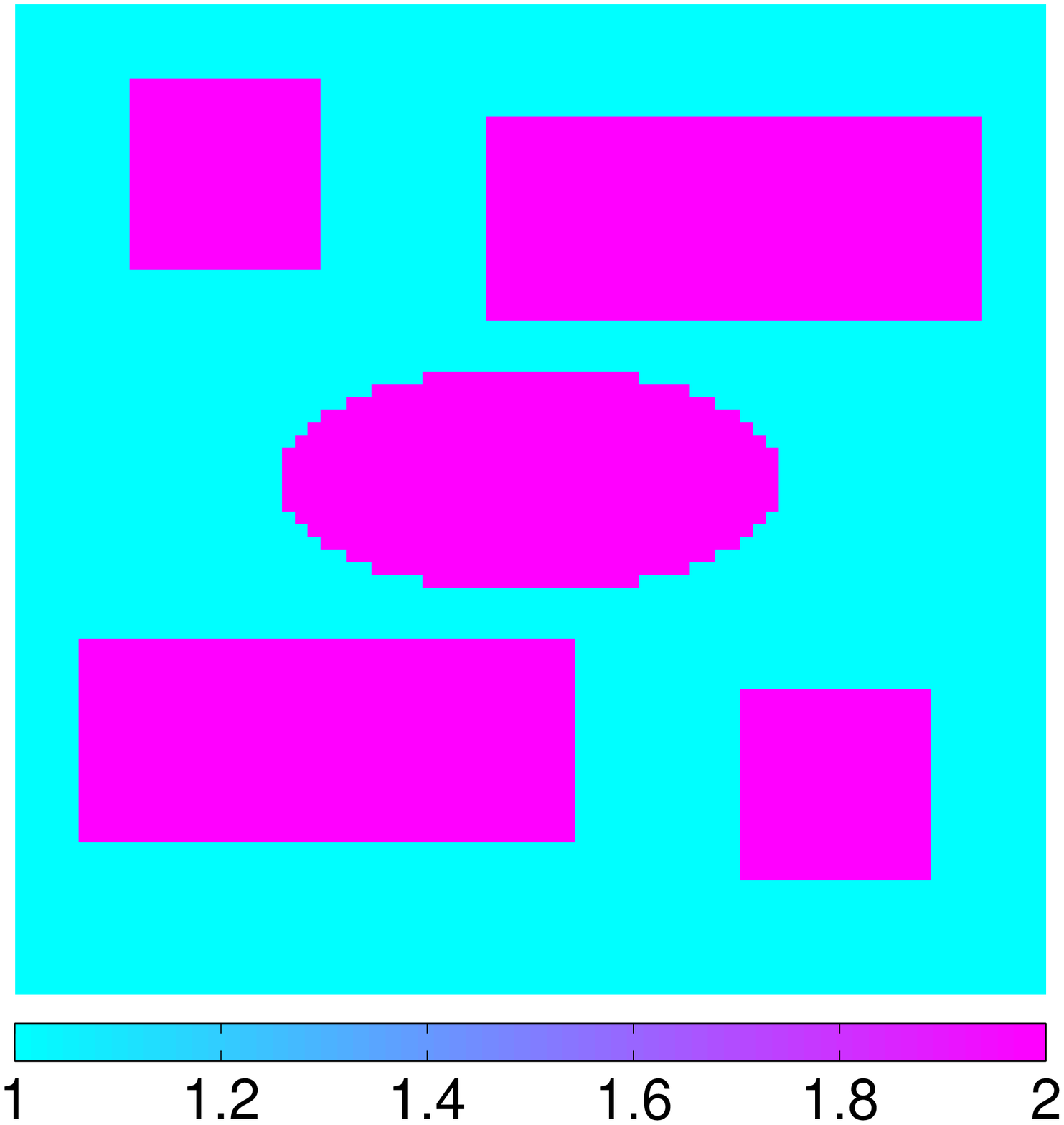}
      \label{ex1tbeta}
      } 
   \subfigure[$\varepsilon_3$ $(\alpha=0\%)$]{  
     \includegraphics[width=37mm,height=35mm]{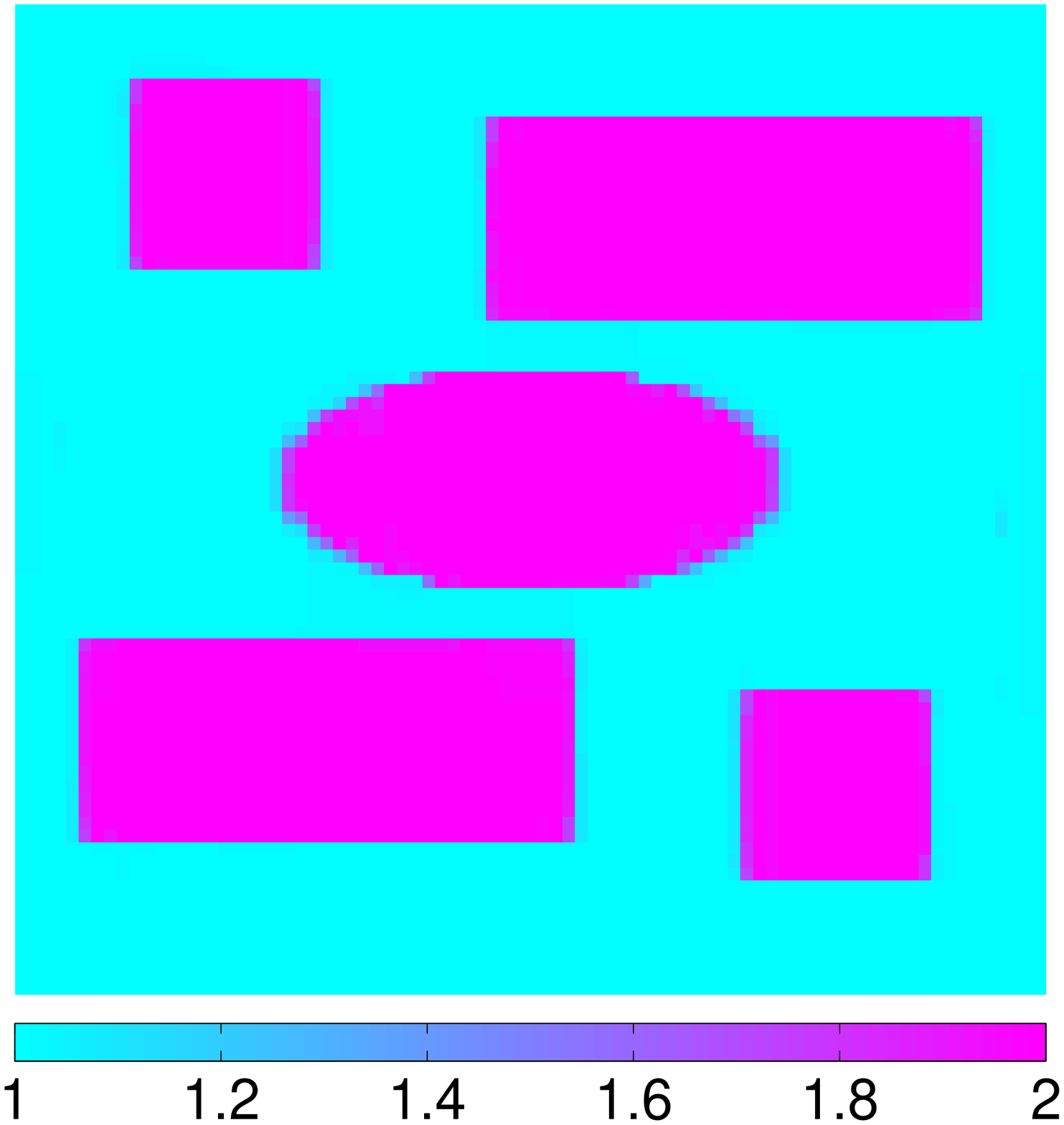}
     \label{ex1cbeta}
     }
   \subfigure[$\varepsilon_3$ ($\alpha=0.1\%$)]{ 
    \includegraphics[width=37mm,height=35mm]{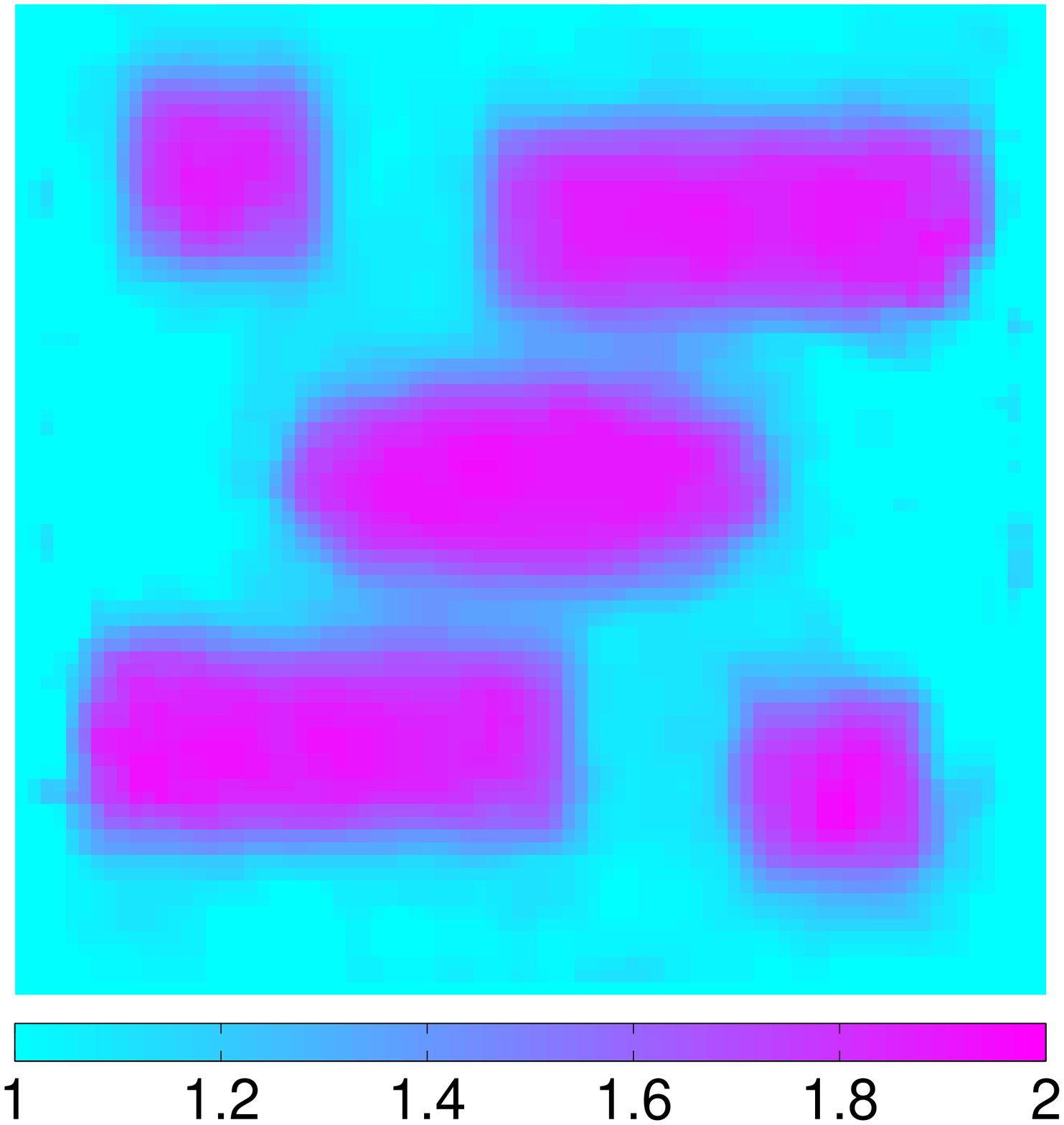}
    \label{ex1nbeta}
    }
   \subfigure[$\varepsilon_3$ at $\{y=0\}$]{ 
     \includegraphics[trim=10mm 5mm 10mm 0mm,clip,width=35mm,height=38mm]{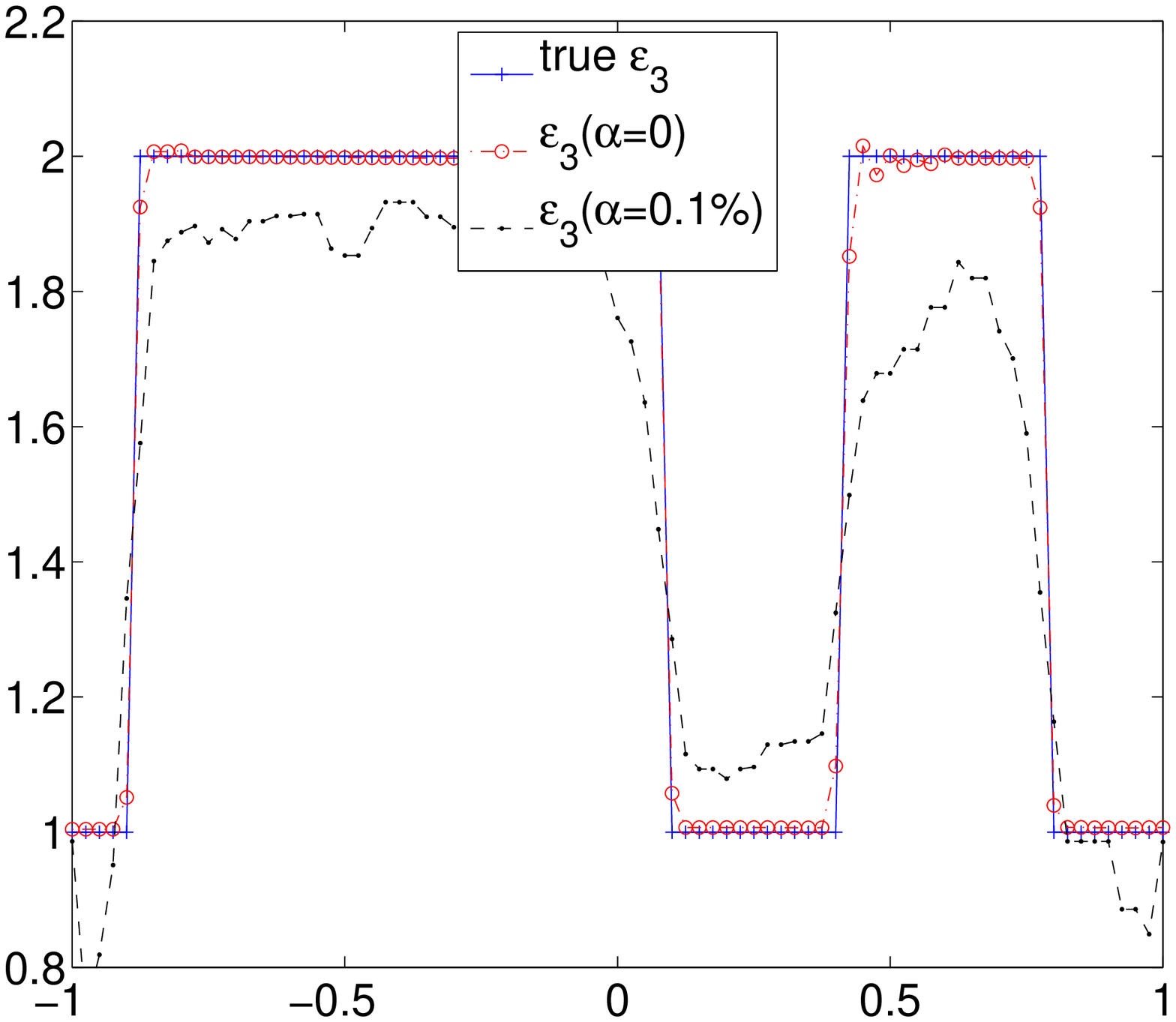} 
     \label{ex1rbeta}
     }
    \caption{$\varepsilon$ in Simulation 2. \subref{ex1txi}\&\subref{ex1ttau}\&\subref{ex1tbeta}: true values of $(\varepsilon_1, \varepsilon_2,\varepsilon_3)$. \subref{ex1cxi}\&\subref{ex1ctau}\&\subref{ex1cbeta}: reconstructions with noiseless data. \subref{ex1nxi}\&\subref{ex1ntau}\&\subref{ex1nbeta}: reconstructions with noisy data($\alpha=0.1\%$). \subref{ex1rxi}\&\subref{ex1rtau}\&\subref{ex1rbeta}: cross sections along $\{y=-0.5\}$.}
\label{E2epsilon}
\end{figure}

\section{Conclusion}\label{se:conclu}
We presented in this paper the reconstruction of $(\sigma, \varepsilon)$ from knowledge of several magnetic fields $H_j$, where the measurements $H_j$ solve the Maxwell's equations \eqref{Eq:maxwell} with prescribed illuminations $f=f_j$ on $\partial X$.

The reconstruction algorithms rely heavily on the linear independence of electric fields and the families of $\{M_j\}_j$ constructed in Hypothesis \ref{2 hypo}. These linear independence conditions can be checked by the available magnetic fields $\{H_j\}_j$ and additional measurements could be added if necessary. This method was used in the numerical simulations. We proved in Section \ref{Fulfilling Hypothesis} that these linear independence conditions could be satisfied by constructing CGO-like solutions for constant tensors. In fact, these conditions can be verified numerically for a large class of illuminations and more general tensors. The numerical simulations illustrate that both smooth and rough coefficients could be well reconstructed, assuming that the interior magnetic fields $H_j$ are accurate enough. However, the reconstructions are very sensitive to the additional noise on the functionals $H_j$. This fact is consistent with the stability estimate (with the loss of two derivatives from the measurements to the reconstructed quantities) in Theorem \ref{stability}.

\section*{Acknowledgment}

\end{document}